\documentclass[reqno,11pt]{amsart}
\usepackage[en-GB]{datetime2}
\usepackage{amssymb}
\usepackage{caption}
\usepackage{enumitem}
\setlist[itemize]{leftmargin=2.5em}
\setlist[enumerate]{leftmargin=2.5em,label=\upshape(\arabic*)}
\usepackage[T1]{fontenc}
\usepackage{setspace}
\usepackage{stix2}

\DeclareSymbolFont{mdcharterletters}{OML}{mdbch}{m}{it}
\SetSymbolFont{mdcharterletters}{bold}{OML}{mdbch}{b}{it}

\DeclareMathSymbol{\ell}{\mathord}{mdcharterletters}{"60}
\setbox0=\hbox{$x$}
\usepackage{tikz}
\usetikzlibrary{matrix,intersections,calc,shapes.geometric,patterns}
\tikzset{ 
	table/.style={
		matrix of nodes,
		nodes={rectangle,text width=1.75em,align=center},
		text depth=1.25ex,
		text height=2.5ex,
		nodes in empty cells
	}
}
\usepackage{mathtools,amsthm,thmtools}
\usepackage[noadjust,nocompress]{cite}

\usepackage{hyperref}
\hypersetup{
	colorlinks=true,
	allcolors=blue,
	pdftitle={Erd\H os--Hajnal for excluding paths},
	pdfauthor={Tung H. Nguyen}
}
\usepackage[capitalise,sort]{cleveref}

\makeatletter

\def\pdf@Cref#1{%
	\ifcsname r@#1@cref\endcsname
	\expandafter\expandafter\expandafter\pdf@Cref@aux
	\csname r@#1@cref\endcsname\pdf@Cref@end
	\else
	??%
	\fi
}
\def\pdf@Cref@aux#1#2#3\pdf@Cref@end{%
	\pdf@Cref@parse#1\@nil
}
\def\pdf@Cref@parse[#1][#2][#3]#4\@nil{%
	\csname Cref@#1@name\endcsname\space#4%
}

\def\pdf@cref#1{%
	\ifcsname r@#1@cref\endcsname
	\expandafter\expandafter\expandafter\pdf@cref@aux
	\csname r@#1@cref\endcsname\pdf@cref@end
	\else
	??%
	\fi
}
\def\pdf@cref@aux#1#2#3\pdf@cref@end{%
	\pdf@cref@parse#1\@nil
}
\def\pdf@cref@parse[#1][#2][#3]#4\@nil{%
	\csname cref@#1@name\endcsname\space#4%
}

\makeatother

\usepackage{geometry}
\newtheorem{theorem}{Theorem}[section]
\newtheorem{lemma}[theorem]{Lemma}
\newtheorem{claim}{Claim}[theorem]
\Crefname{claim}{Claim}{Claims}
\newenvironment{clm}[1]
{\renewcommand\theclaim{#1}\claim}
{\endclaim}
\newtheorem{conjecture}[theorem]{Conjecture}

\Crefname{prop}{Proposition}{Propositions}
\Crefname{subsection}{Subsection}{Subsections}
\Crefname{problem}{Problem}{Problems}
\allowdisplaybreaks

\newcommand{\blackqed}{%
	\leavevmode
	\hbox to .77777em{\hfil
		\vrule width .6em height .6em depth 0em
		\hfil}%
}

\expandafter\let\expandafter\oldproof\csname\string\proof\endcsname
\let\oldendproof\endproof
\renewenvironment{proof}[1][\proofname]{%
	\oldproof[\normalfont\bfseries #1]%
}{\oldendproof}
\newenvironment{subproof}[1][\normalfont\it Subproof]{%
	\begin{proof}[#1]%
	}{%
	\end{proof}%
}

\newcommand\mac[1]{\text{\usefont{OMS}{cmsy}{m}{n}#1}\mskip 2.5mu}
\newcommand{\mact}{\mac T\,}
\newcommand{\mab}{\mathbb}

\newcommand{\eps}{\varepsilon}

\newcommand{\poly}{\operatorname{poly}}
\newcommand{\ind}{\operatorname{ind}}
\renewcommand{\subset}{\subseteq}
\renewcommand{\supset}{\supseteq}
\newcommand{\erh}{Erd\H{o}s--Hajnal}

\DeclarePairedDelimiter\abs{\lvert}{\rvert}%
\DeclarePairedDelimiter\ceil{\lceil}{\rceil}%
\DeclarePairedDelimiter\floor{\lfloor}{\rfloor}%
\makeatletter
\newcommand{\leqnomode}{\tagsleft@true}
\newcommand{\reqnomode}{\tagsleft@false}
\makeatother

\DeclareSymbolFont{timesletters}{OML}{ztmcm}{m}{it}

\makeatletter
\newcommand*\greekrange[3]{%
	\@tempcnta=#2\relax
	\@for\g:=#3\do{%
		\expandafter\DeclareMathSymbol\csname\g\endcsname
		{\mathalpha}{#1}{\the\@tempcnta}%
		\advance\@tempcnta\@ne}}
\makeatother

\greekrange{timesletters}{11}{alpha,beta,gamma,delta,epsilon,zeta,eta,theta,iota,kappa,lambda,mu,nu,xi,pi,rho,sigma,tau,upsilon,phi,chi,psi,omega}
\greekrange{timesletters}{34}{varepsilon,vartheta,varpi,varrho,varsigma,varphi}

\newcommand{\compactsubs}[1]{%
	\fontdimen16\textfont2=#1\fontdimen16\textfont2
	\fontdimen17\textfont2=#1\fontdimen17\textfont2
	\fontdimen16\scriptfont2=#1\fontdimen16\scriptfont2
	\fontdimen17\scriptfont2=#1\fontdimen17\scriptfont2
	\fontdimen16\scriptscriptfont2=#1\fontdimen16\scriptscriptfont2
	\fontdimen17\scriptscriptfont2=#1\fontdimen17\scriptscriptfont2
}

\compactsubs{.8}
\begin{document}
	
	\title{
		Erd\H os--Hajnal for excluding paths
	}
	\author{Tung H. Nguyen}
	\address{Mathematical Institute and Christ Church, University of Oxford, Oxford, UK}
	\email{\href{mailto:nguyent@maths.ox.ac.uk}{nguyent@maths.ox.ac.uk}}
	\thanks{Supported by a Titchmarsh Research Fellowship and a Christ Church Research Centre Grant.}
	
	\begin{abstract}
		We prove the Erd\H os--Hajnal conjecture for excluding an arbitrary path.
		A major step in our proof is a novel Gy\'arf\'as path argument based on the ideas of bookkeeping blockades and splitting paths, which may be of independent interest.
	\end{abstract}
	\maketitle
	
	\section{Introduction}
	\label{sec:intro}
	All graphs in this paper are finite and simple.
	For a graph $G$ with vertex set $V(G)$, let $\abs G:=\abs{V(G)}$.
	For a graph $H$, a {\em copy} of $H$ in $G$ is an injective map $\varphi\colon V(H)\to V(G)$ such that for all $u,v\in V(H)$, $uv\in E(H)$ if and only if $\varphi(u)\varphi(v)\in E(G)$.
	We say that $G$ is {\em $H$-free} if there is no copy of $H$ in $G$; or equivalently if $G$ has no induced subgraph isomorphic to $H$.
	The graph $H$ has the {\em Erd\H os--Hajnal} property if there exists $b\ge2$ such that every $H$-free graph $G$ has a clique or stable set of size at least $\abs G^{1/b}$; and in this case we say that $b$ {\em witnesses} the \erh{} property of $H$.
	The following conjecture was proposed by Erd\H os and Hajnal in 1977~\cite{eh1977,eh1989} (see~\cite{chudnovsky2014,gyarfas1997} for surveys, and the author's PhD thesis~\cite{2025thes} for an account of recent progress):
	
	\begin{conjecture}
		[Erd\H os and Hajnal]
		\label{conj:eh}
		Every graph $H$ has the \erh{} property.
	\end{conjecture}
	
	This conjecture remains open, and has been confirmed for several special cases of $H$; see for instance~\cite{cs2008,csss2023a,density4,density6,density7,aps2001}.
	Erd\H os and Hajnal~\cite{eh1977,eh1989} first obtained a general bound of $2^{\Omega(\sqrt{\log\abs G})}$, which was recently improved by Buci\'c, Scott, Seymour, and the author~\cite{bnss2024} to $2^{\Omega(\sqrt{\log \abs G\log\log \abs G})}$.
	
	For an integer $t\ge1$, let $P_t$ denote the $t$-vertex path. Prior to this work, Scott, Seymour, and the author~\cite{density5} proved that every $P_t$ has the `near' \erh{} property (meaning with a bound of $2^{(\log \abs G)^{1-o(1)}}$ instead of $\abs G^{\Omega(1)}$).
	Our main result upgrades this to the full \erh{} property of all paths, as follows.
	
	\begin{theorem}
		\label{thm:path}
		$P_t$ has the \erh{} property for all $t\ge1$.
	\end{theorem}
	
	The method used to prove \cref{thm:path} can be developed to give more; here we provide an application.
	In what follows, let {\em E} be the graph obtained from $P_5$ by adding a pendant edge to the middle vertex, and let the {\em bird} be the graph obtained from E by adding an edge between a vertex of degree two and the middle leaf.
	Huang, Ju, and Zhou~\cite{hjz2026} proved that E and the bird have the \erh{} property.
	We provide a common generalisation of their result and \cref{thm:path} as follows.
	For integers $m,n,t\ge1$, let $S_{m,n,t}$ be the graph obtained from the claw $K_{1,3}$ by subdividing each of its edges $m-1$ times, $n-1$ times, and $t-1$ times, respectively.
	For $m=1$ and $n=2$, one can represent $S_{1,2,t}$ as the tree with vertex set $\{r_0,q_1\}\cup\{p_1,p_2\}\cup\{z_1,\ldots,z_t\}$ and paths
	\[r_0\text-q_1, \qquad r_0\text-p_1\text-p_2,\qquad r_0\text-z_1\text-\cdots\text-z_t.\]
	Let $S_{1,2,t}'$ be obtained from $S_{1,2,t}$ by adding the edge $p_1q_1$.
	Then $S_{1,2,2}$ is E, $S'_{1,2,2}$ is the bird, and $S_{1,2,1}'$ is the well-known {\em bull}.
	Our common generalisation asserts that:
	\begin{theorem}
		\label{thm:extend}
		$S_{1,2,t}$ and $S_{1,2,t}'$ have the \erh{} property for all $t\ge1$.
	\end{theorem}
	The proof of this is a technical extension of the proof of \cref{thm:path}, without using the \erh{} property of E and the bird.
	As such, the rest of this paper is devoted to \cref{thm:path}; and we will explain the modifications to prove \cref{thm:extend} along the way.
	
	We will review preliminary materials in \cref{sec:prelim}, and will give a more detailed organisation of the rest of this paper at the end of \cref{subsec:keylemmas}.

	\subsection*{Concurrent work}
	During the preparation of this paper, we learned that the following projects were carried out independently of our research:
	\begin{itemize}
		\item While we were preparing the current paper, a team of researchers at Google Research and DeepMind (including Junehyuk Jung, Joonkyung Lee, Thang Luong, and Aranyak Mehta) contacted and informed us that a DeepMind internal model had produced a proof of the \erh{} property of $P_6$; the proof was verified by Joonkyung Lee and accompanied by a Lean formalisation.
		
		We then notified the Google team about our $P_t$ result (\cref{thm:path}) and our unfinished draft at the time, and suggested that both parties announce their papers on the arXiv independently on the same day while acknowledging each other's work.
		
		A few days later, the Google team replied to us that in parallel to our notification, they continued exploring several special cases of the \erh{} conjecture using their internal model, and obtained further results including the $P_t$ case.
		
		After discussing how to present our respective $P_t$ results, both parties agreed to independently publicise their papers, but on different platforms:
		Our preprint would be submitted to the arXiv; meanwhile, the Google team would upload their AI-generated manuscript (together with the accompanying Lean proof) to \href{https://github.com/google-deepmind/superhuman/tree/main/FutureMath/Erdos-Hajnal}{their GitHub repository}, since they wanted to ensure that their material on the arXiv meets the standards of the mathematical community.
		At the time of submission, neither team had seen the other's manuscript.
		
		\item We also learned that \href{https://github.com/machine-qed/erdos-hajnal-six-vertex-path}{an anonymous GitHub repository} contains an AI-generated manuscript claiming to prove the \erh{} property of $P_6$, produced by Claude Opus 5.5; and \href{https://github.com/OctopusTide/erdos-hajnal-all-paths}{another GitHub repository} by Yuwen Zhou contains a manuscript claiming to prove \cref{thm:path} (also available on \href{https://zenodo.org/records/23270286}{Zenodo}), assisted by Claude Opus 5.5.
		According to the information given in these hyperlinks, both manuscripts are accompanied by corresponding Lean formalisations.
		We are not aware of independent expert reviews of either manuscript.
	\end{itemize}

	\section{Preliminaries}
	\label{sec:prelim}
	In this section, we collect relevant notation and definitions, together with several facts that will be used in this paper.
	These facts are well-known or can be deduced more or less straightforwardly from existing results and techniques in the literature.
	At the end of the section (\cref{subsec:keylemmas}), we will introduce the two key lemmas, \cref{lem:pathpure,lem:niceeh}, and quickly deduce \cref{thm:path} from these two lemmas and the collected facts.
	\subsection{Notation and definitions}
	For an integer $k\ge0$, let $[k]:=\{1,\ldots,k\}$ if $k\ge1$ and $[k]:=\empty$ if $k=0$. We use $\log$ to denote the binary logarithm.
	For sets $X,Y$, a map $f\colon X\to Y$, and $Z\subset X$, let $f\vert_Z$ denote the restriction of $f$ to $Z$, and let $f(Z):=\{f(z):z\in Z\}$.
	
	For a graph $G$ with $v\in V(G)$, let $N_G(v)$ be the set of its neighbours in $G$, and let $N_G[v]:=N_G(v)\cup\{v\}$.
	Let $\overline G$ be the complement graph of $G$, that is, the graph with vertex set $V(G)$ and edge set $\{uv:v\in V(G),u\in V(G)\setminus N_G[v]\}$.
	A {\em component} of $G$ is a maximal connected induced subgraph of $G$.
	For $S\subset V(G)$, let $G[S]$ be the subgraph induced on $S$, let $G\setminus S:=G[V(G)\setminus S]$, and let $G\setminus v:=G\setminus\{v\}$ for all $v\in V(G)$.
	If $v\not\in S$, we also say that $v$ is {\em pure} to $S$ if either $S\subset N_G(v)$ or $S\cap N_G(v)=\emptyset$, and {\em mixed} on $S$ if it is not pure to $S$.
	
	For disjoint $A,B\subset V(G)$, we say that $(A,B)$ is a {\em complete} pair in $G$, or $A$ is {\em complete} to $B$ in $G$, if $uv\in E(G)$ for all $u\in A$ and $v\in B$; we also say that $(A,B)$ is an {\em anticomplete} pair in $G$, or $A$ is {\em anticomplete} to $B$ in $G$, if $G$ has no edge between $A$ and $B$.
	We say that $A$ is {\em pure} to $B$ if every vertex in $A$ is pure to $B$.
	The pair $(A,B)$ is {\em pure} in $G$ if it is either complete or anticomplete in $G$; thus, $(A,B)$ is pure in $G$ if and only if $A,B$ are pure to each other in $G$.
	
	A {\em blockade} in $G$ is a sequence $\mac B=(B_1,\ldots,B_k)$ of pairwise disjoint and nonempty subsets of $V(G)$; its {\em length} is $k$, its {\em width} is $\min_{i\in[k]}\abs{B_i}$, and its {\em blocks} are $B_1,\ldots,B_k$.
	For $\ell,w\ge0$, we say that $\mac B$ is an {\em $(\ell,w)$-blockade} if its length is at least $\ell$ and its width is at least $w$.
	The blockade $\mac B$ is {\em complete} if its blocks are pairwise complete to each other, is {\em anticomplete} if its blocks are anticomplete to each other, and {\em pure} if its blocks are pairwise pure to each other.
	
	\subsection{Basic facts}
	For two graphs $H_1$ and $H_2$, and for $v\in V(H_1)$, {\em substituting $H_2$ for $v$ in $H_1$} means creating a new graph by taking a disjoint union of $H_1\setminus v$ and $H_2$, then adding an edge between every vertex in $H_2$ and every neighbour of $v$ in $H_1$.
	We need the following fundamental result of Alon, Pach, and Solymosi~\cite{aps2001} on substitution and \erh{}.
	
	\begin{theorem}
		[Alon, Pach, and Solymosi]
		\label{thm:aps}
		If $H_1$ and $H_2$ have the \erh{} property, then so does the graph obtained by substituting $H_2$ for any given vertex of $H_1$.
	\end{theorem}
	
	For an integer $k\ge1$ and a real $w>0$, a {\em $(k,w)$-comb} in a graph $G$ is a sequence of pairs $((a_i,B_i):i\in[k])$ such that $\abs{B_i}\ge w$ for all $i\in[k]$; $B_i\subset N_G(a_i)$ for all $i\in[k]$; and $a_i\ne a_j$ and $B_i\subset N_G(a_i)\setminus N_G[a_j]$ for all distinct $i,j\in [k]$.
	A {\em $(k,w)$-anticomb} in $G$ is a $(k,w)$-comb in $\overline G$.
	We require the following useful special case of the `comb' lemma~\cite{csss2023a}, which was used also in~\cite{density7}. 
	
	\begin{lemma}
		[Chudnovsky, Scott, Seymour, and Spirkl]
		\label{lem:comb}
		Let $G$ be a graph and let $A,B\subseteq V(G)$ be nonempty and disjoint, such that
		each vertex in $A$ has at most $\Delta>0$ neighbours in $B$.
		Then either:
		\begin{itemize}
			\item at most $20\sqrt{\abs B\Delta}$ vertices in $B$ have a neighbour in $A$; or
			
			\item
			for some integer $k\ge1$, there is a $(k,\abs B/k^2)$-comb $((a_i,B_i):i\in[k])$ in $G$ where $a_i\in A$ and $B_i\subset B$ for all $i\in[k]$.
		\end{itemize} 
	\end{lemma}
	
	\subsection{Density}
	For $x\in(0,1)$ and a graph $G$ with disjoint $A,B\subseteq V(G)$, we say that $B$ is {\em $x$-sparse} to $A$ (or $(A,B)$ is {\em $x$-sparse} in $G$) if every vertex in $B$ has fewer than $x\abs A$ neighbours in $A$, {\em $x$-dense} to $A$ (or $(A,B)$ is {\em $x$-dense} in $G$) if every vertex in $B$ has more than $x\abs A$ neighbours in $A$, and {\em $x$-pure} to $A$ if every vertex in $B$ is either $x$-sparse or $(1-x)$-dense to $A$.
	We say that $(A,B)$ is {\em $x$-pure} in $G$ if either $A,B$ are $x$-sparse to each other or $(1-x)$-dense to each other. The following lemma is an extension of the aforesaid fact that $(A,B)$ is pure in $G$ if and only if $A,B$ are pure to each other in $G$.
	\begin{lemma}
		\label{lem:xpure}
		Let $x\in(0,\frac12]$, and let $G$ be a graph with nonempty disjoint $A,B\subset V(G)$ that are $x^3$-pure to each other. Then there exist $X\subset A$ and $Y\subset B$ such that $\abs X\ge(1-x^2)\abs A$, $\abs Y\ge (1-x^2)\abs B$, and $(X,Y)$ is $x^2$-pure in $G$.
	\end{lemma}
	\begin{proof}
		Since $A$ is $x^3$-pure to $B$, it admits a partition $(A_0,A_1)$ such that $A_0$ is $x^3$-sparse to $B$ and $A_1$ is $(1-x^3)$-dense to $B$.
		If one of $A_0,A_1$ is empty then we are done; so we assume that they are nonempty.
		Then there exists $B'\subset B$ with $\abs{B'}\ge(1-2x^2)\abs B>0$, such that $B'$ is $x$-sparse to $A_0$ and $(1-x)$-dense to $A_1$. Since $B'\subset B$ is $x^3$-pure to $A$, it follows that $\min(\abs{A_0},\abs{A_1})\le \frac{x^3}{1-x}\abs A\le 2x^3\abs A$.
		Hence, there exists $X\in\{A_0,A_1\}$ such that $\abs X\ge (1-2x^3)\abs A\ge (1-x^2)\abs A\ge\frac12\abs A$ and $X$ is $x^3$-sparse or $(1-x^3)$-dense to $B$.
		Similarly, there exists $Y\subset B$ with $\abs Y\ge (1-2x^3)\abs B\ge(1-x^2)\abs B\ge\frac12\abs B$ such that $Y$ is $x^3$-sparse or $(1-x^3)$-dense to $A$.
		Counting edges between $X$ and $Y$ then implies that $(X,Y)$ is $2x^3$-pure in $G$.
		Since $2x^3\le x^2$, this proves \cref{lem:xpure}.
	\end{proof}
	
	For $\eps\in(0,1)$, a graph $G$ is {\em $\eps$-sparse} if its maximum degree is less than $\eps\abs G$, is {\em $\eps$-dense} if its complement $\overline G$ is $(1-\eps)$-sparse, and is {\em $\eps$-restricted} if it is $\eps$-sparse or $(1-\eps)$-dense.
	The following fundamental result of R\"odl~\cite{rodl1986} says that graphs with a forbidden induced subgraph have substantially uneven edge distribution.
	
	\begin{theorem}
		[R\"odl]
		\label{thm:rodl}
		For every $\eps\in(0,\frac12]$ and every graph $H$, there exists $\delta>0$ such that every $H$-free graph $G$ contains an $\eps$-restricted induced subgraph on at least $\delta\abs G$ vertices.
	\end{theorem}
	A graph $H$ has the {\em polynomial R\"odl} property if there exists $d\ge1$ such that, for every $\eps\in(0,\frac12]$, the $\delta$ in \cref{thm:rodl} can be taken as $\eps^d$; and Fox and Sudakov~\cite{fs2008} conjectured that every graph $H$ is polynomial R\"odl, which remains open.
	We will not need the polynomial R\"odl property in this paper, but instead the more flexible viral property, defined as follows.
	For graphs $G$ and $H$, let $\ind_H(G)$ denote the number of copies of $H$ in $G$.
	We say that $H$ is {\em viral} if there exists $d\ge2$ such that for every $\eps\in(0,\frac12]$, every graph $G$ with $\ind_H(G)<(\eps^d\abs G)^{\abs H}$ has an $\eps$-restricted induced subgraph $F$ with $\abs F\ge \eps^d\abs G$;
	in this case, we also say that $d$ {\em witnesses} the viral property of $H$.
	The viral property was first introduced by Fox, Scott, Seymour, and the author~\cite{density2} as an extension of the \erh{} property; and
	Buci\'c, Fox, and Pham~\cite{bfp2024} proved that the two properties are actually equivalent (and so equivalent to the polynomial R\"odl property), as follows.
	
	\begin{theorem}
		[Buci\'c, Fox, and Pham]
		\label{thm:bfp}
		A graph has the \erh{} property if and only if it is viral.
	\end{theorem}
	
	We need the following lemma on counting induced subgraphs on blockades, whose proof is essentially the `using the leaf' argument in~\cite[Section 5]{density4}.
	
	\begin{lemma}
		\label{lem:blockcount}
		Let $H$ be a graph with a vertex $v$ of degree one.
		Let $x\in(0,(2\abs H)^{-1}]$, and let $G$ be an $H$-free graph with a blockade $(B_1,\ldots,B_k)$ where $k\in[2,1/x]$.
		Let $I$ be the set of ordered pairs $(i,j)$ with $1\le i<j\le k$ such that $B_j$ is $x^2$-sparse to $B_i$ in $G$.
		Then for every $y\in(0,1)$, either:
		\begin{itemize}
			\item there exist $i\in[k]$ and $D_i\subset B_i$ such that $\abs{D_i}\ge\abs{B_i}/2$ and $\ind_{H\setminus v}(G[D_i])<y\abs{D_i}^{\abs H-1}$; or
			
			\item for each $i\in[k]$, there exists $D_i\subset B_i$ with $\abs{D_i}\ge y\abs{B_i}/2$, such that $D_j$ is anticomplete to $D_i$ for all $(i,j)\in I$.
		\end{itemize}
	\end{lemma}	
	\begin{proof}
		Let $u$ be the unique neighbour of $v$ in $H$.
		Assume that the first outcome fails.
		Let $r\in\{0,1,\ldots,k\}$ be minimal such that there exists $D_i\subset B_i$ for each $i\in[k]$ satisfying:
		\begin{itemize}
			\item $\abs{D_i}\ge (1-(k-r)\abs Hx^2)\abs{B_i}$ for all $i\in[r]$, and $\abs{D_j}\ge y\abs{B_j}/2$ for all $j\in[k]\setminus[r]$; and
			
			\item for all $(i,j)\in I$ such that $i<j$ and $j\in[k]\setminus[r]$, $D_j$ is anticomplete to $D_i$.
		\end{itemize}
		(These conditions are clearly satisfied for $r=k$.)
		
		If $r=0$ then the second outcome holds; so suppose for a contradiction that $r\ge1$.
		Since $k\le 1/x$ and $x\le(2\abs H)^{-1}$, we see that for all $i\in[r]$,
		\[\abs{D_i}\ge(1-(k-r)\abs Hx^2)\abs{B_i}\ge(1-(k-1)\abs Hx^2)\abs{B_i}\ge(1-\abs Hx)\abs{B_i}\ge\abs{B_i}/2.\]
		In particular, since the first outcome fails, we have $\ind_{H\setminus v}(G[D_r])\ge y\abs{D_r}^{\abs H-1}$.
		Let $L$ be the set of copies $\phi$ of $H\setminus\{u,v\}$ in $G[D_r]$; and for each $\phi\in L$, let $S_{\phi}$ be the set of copies $\varphi$ of $H\setminus v$ in $G[D_r]$ with $\varphi\vert_{V(H)\setminus\{u,v\}}=\phi$.
		Then $\ind_{H\setminus v}(G[D_r])=\sum_{\phi\in L}\abs{S_\phi}$; and so there exists $\phi\in L$ such that 
		\[\abs{S_\phi}\ge \ind_{H\setminus v}(G[D_r])/\abs L\ge y\abs{D_r}\ge y\abs{B_r}/2.\]
		Let $A:=\phi(V(H)\setminus\{u,v\})$, and let $D_r':=\{\varphi(u):\varphi\in S_\phi\}$; then $\abs{D_r'}=\abs{S_\phi}\ge y\abs{B_r}/2$.
		
		Now, for each $i\in[r-1]$ with $(i,r)\in I$, let $D_i'$ be the set of vertices in $D_i$ with no neighbour in $A$; then
		\[\abs{D_i'}\ge\abs{D_i}-(\abs H-2)x^2\abs{B_i}\ge (1-(k-r+1)\abs H x^2)\abs{B_i}.\]
		For each such $i$, if there is an edge between $D_i'$ and $D_r'$ then there would be a copy of $H$ in $G$, violating the $H$-freeness of $G$.
		Hence, $D_r'$ is anticomplete to $D_i'$ for all $i\in[r-1]$ with $(i,r)\in I$.
		For every remaining $i\in[k]\setminus\{r\}$, set $D_i':=D_i$; the resulting sets contradict the minimality of $r$.
		This proves \cref{lem:blockcount}.
	\end{proof}
	
	For $x\in(0,1)$ and a graph $G$ with a blockade $\mac B=(B_1,\ldots,B_k)$, we say that $\mac B$ is {\em $x$-sparse} if $B_j$ is $x$-sparse to $B_i$ for all $1\le i<j\le k$, and {\em $x$-dense} if $B_j$ is $x$-dense to $B_i$ for all $1\le i<j\le k$.
	We also say that $\mac B$ is {\em $x$-pure} if $(B_i,B_j)$ is $x$-pure for all distinct $i,j\in[k]$, and that $\mac B$ is {\em $x$-semidense} if for all distinct $i,j\in[k]$, either $B_i$ is anticomplete to $B_j$, or $B_i,B_j$ are $x$-dense to each other.
	(The notion of semidense is a `complement' analogue of the notion of semisparse introduced in~\cite{density7}.)
	Thus, for $x\in(0,\frac12]$, the $(1-x)$-semidense property is stronger than $x$-pureness, because sparse pairs become fully anticomplete.
	
	In \cref{lem:blockcount}, letting $y=k^{-a}$ and $I=\{(i,j):1\le i<j\le k\}$ yields the following lemma, which essentially implies that if $H\setminus v$ is viral, then every sparse $H$-free graph contains either a long and wide anticomplete blockade, or a much more restricted induced subgraph with polynomial shrinkage factor.
	
	\begin{lemma}
		\label{lem:antiviral}
		Let $H$ be a graph with a vertex $v$ of degree one.
		Let $x\in(0,(2\abs H)^{-1}]$, and let $G$ be an $H$-free graph with an $x^2$-sparse blockade $(B_1,\ldots,B_k)$ where $k\in[2,1/x]$.
		Then for every $a\ge1$, either:
		\begin{itemize}
			\item there exist $i\in[k]$ and $D_i\subset B_i$ such that $\abs{D_i}\ge\abs{B_i}/2$ and $\ind_{H\setminus v}(G[D_i])<\abs{D_i}^{\abs H-1}/k^a$; or
			
			\item there is an anticomplete blockade $(D_1,\ldots,D_k)$ such that for all $i\in[k]$, $D_i\subset B_i$ and $\abs{D_i}\ge \abs{B_i}/(2k^a)$.
		\end{itemize}
	\end{lemma}
	
	A graph $H$ is {\em pure} if there exists $d\ge2$ such that for every $\eps\in(0,\frac12]$, every $H$-free graph $G$ with $\abs G>\eps^{-d}$ contains an $\eps$-pure $(\eps^{-1},\eps^d\abs G)$-blockade.
	The argument of Fox, Pach, and Suk~\cite{fps2019} shows that every pure graph has the `near' \erh{} property.
	We use \cref{lem:blockcount} to deduce the following result.
	
	\begin{lemma}
		\label{lem:nearnice}
		Let $H$ be a graph with a vertex $v$ of degree one, such that $H\setminus v$ is viral and $H$ is pure. 
		Then there exists $d\ge2$ such that for every $\eps\in(0,\frac12]$, 
		every $H$-free graph $G$ contains either an $\eps$-sparse induced subgraph with at least $\eps^d\abs G$ vertices, or a $(1-\eps)$-semidense $(\eps^{-1},\eps^d\abs G)$-blockade.
	\end{lemma}
	\begin{proof}
		Let $a\ge2$ witness the viral property of $H\setminus v$ and the purity of $H$; by increasing $a$ if necessary, we may assume $2^a\ge 2\abs H$. 
		Let $b:=8a\abs H$.
		We claim that $d:=2ab$ suffices.
		Thus, it suffices to show that for every $\eps\in(0,\frac12]$, every $H$-free graph $G$ contains either an $\eps$-sparse induced subgraph with at least $\eps^d\abs G$ vertices, or a $(1-\eps)$-semidense $(\eps^{-1},\eps^d\abs G)$-blockade.
		If $\abs G\le \eps^{-d}$ then the first outcome holds by taking any induced subgraph of $G$ with $\ceil{\eps^d\abs G}\in\{0,1\}$ vertices; and so we may assume $\abs G>\eps^{-d}$.
		Hence $\abs G>\eps^{-d}\ge\eps^{-ab}$ by the choice of $d$; and the choice of $a$ gives an $\eps^{b}$-pure $(\eps^{-b},\eps^{ab}\abs G)$-blockade $\mac B$ in $G$.
		Let $k:=\floor{\eps^{-a}}\ge\frac12\eps^{-a}\ge\frac12\eps^{-2}\ge\eps^{-1}$, and choose $k$ blocks $B_1,\ldots,B_k$ from $\mac B$.
		Let $I$ be the set of all ordered pairs $(i,j)$ with $1\le i<j\le k$ such that $B_j$ is $\eps^{b}$-sparse to $B_i$ in $G$.
		By \cref{lem:blockcount} with $x=\eps^a\in(0,(2\abs H)^{-1}]$ and $y=\eps^{4a\abs H}=\eps^{b/2}$, either:
		\begin{itemize}
			\item there exist $i\in[k]$ and $D_i\subset B_i$ such that $\abs{D_i}\ge\abs{B_i}/2$ and $\ind_{H\setminus v}(G[D_i])<y\abs{D_i}^{\abs H-1}$; or
			
			\item for each $i\in[k]$, there exists $D_i\subset B_i$ with $\abs{D_i}\ge y\abs{B_i}/2$, such that $D_j$ is anticomplete to $D_i$ for all $(i,j)\in I$.
		\end{itemize}
		
		If the second bullet holds, then $\abs{D_i}\ge y\abs{B_i}/2\ge \eps^{4a\abs H+ab}\abs G/2\ge \eps^{2ab}\abs G\ge \eps^d\abs G$ for all $i\in[k]$.
		Thus, since $\eps^b/(y/2)=2\eps^{4a\abs H}\le \eps$ and $B_i,B_j$ are $(1-\eps^b)$-dense to each other for all $i,j\in[k]$ with $i<j$ and $(i,j)\not\in I$, we deduce that $(D_1,\ldots,D_k)$ is a $(1-\eps)$-semidense $(\eps^{-1},\eps^d\abs G)$-blockade in $G$, as desired.
		
		Hence, we may assume the first bullet holds. Then $\ind_{H\setminus v}(G[D_i])<y\abs{D_i}^{\abs H-1}\le (\eps^{4a}\abs {D_i})^{\abs{H}-1}$;
		and so the choice of $a$ gives an $\eps^4$-restricted induced subgraph $F$ of $G[D_i]$ with
		\[\abs F\ge \eps^{4a}\abs{D_i}\ge\eps^{4a}\abs{B_i}/2\ge \eps^{4a+ab}\abs G/2\ge\eps^{4a+ab+1}\abs G\ge\eps^{d-3}\abs G\]
		by the choice of $d$.
		Thus, if $F$ is $\eps^4$-sparse then we are done; and so we may assume that $F$ is $(1-\eps^4)$-dense.
		Let $\ell:=\ceil{\eps^{-1}}$. Since $\ell\cdot\ceil{\eps^{3}\abs F}\le 4\eps^{2}\abs F\le \abs F$,
		there is an $(\ell,\eps^{3}\abs F)$-blockade in $F$ whose blocks are pairwise $(1-\eps)$-dense to each other. Since $\eps^3\abs F\ge\eps^d\abs G$,
		this proves \cref{lem:nearnice}. 
	\end{proof}
	
	In this paper, we will frequently apply the following iteration result to extract long and wide blockades, which implicitly appeared in previous works on \erh{} such as~\cite{bnss2024,density4,density5,density6,density7}.
	
	\begin{lemma}
		\label{lem:iteration}
		Let $x,y\in(0,\frac12]$, and let $p\ge1$. Let $G$ be a graph, such that every induced subgraph $F$ of $G$ with $\abs F\ge \frac14\abs G$ contains an $x$-sparse pair $(X,Y)$ with $\abs X\ge p$ and $\abs Y\ge(1-y)\abs F$.
		Then $G$ contains an $x$-sparse blockade of length at least $1/y$ and width at least $p$.
	\end{lemma}
	\begin{proof}
		Let $k\ge0$ be maximal such that $G$ contains an $x$-sparse blockade $(B_0,B_1,\ldots,B_k)$ where $\abs{B_k}\ge (1-y)^k\abs G$ and $\abs{B_{i-1}}\ge p$ for all $i\in[k]$. If $k\ge 1/y$ then we are done; and so we may assume that $k<1/y$. Then
		\[\abs{B_k}\ge (1-y)^k\abs G\ge 4^{-yk}\abs G\ge\abs G/4\]
		where we used the inequality $1-y\ge 4^{-y}$ since $y\le\frac12$. Thus, the hypothesis with $F=G[B_k]$ gives an $x$-sparse pair $(X,Y)$ in $G[B_k]$ such that $\abs X\ge p$ and $\abs Y\ge (1-y)\abs{B_k}\ge (1-y)^{k+1}\abs G$.
		Hence $(B_0,B_1,\ldots,B_{k-1},X,Y)$ violates the maximality of $k$. This proves \cref{lem:iteration}.
	\end{proof}
	
	The proof of \cref{thm:path} involves the method of iterative sparsification introduced in~\cite{density4}, which played an important role in recent progress on \erh{} such as~\cite{density3,density6,density7}.
	In this paper, we need the following `sparsification' variant of~\cite[Lemma 3.2]{density6}.
	
	\begin{lemma}
		\label{lem:sparsify}
		Let $c\in(0,1)$, $a\ge2$, and $t\ge1$.
		Suppose that $x\in(0,c]$, and $G$ is a graph satisfying:
		\begin{itemize}
			\item there is a $c$-sparse induced subgraph of $G$ with at least $c^t\abs G$ vertices; and
			
			\item for every $y\in[x,c]$ and every $y$-sparse induced subgraph $F$ of $G$ with $\abs F\ge y^{2t}\abs G$,
			there is a $y^a$-sparse induced subgraph of $F$ with at least $y^{at}\abs F$ vertices.
		\end{itemize}
		Then $G$ contains an $x$-sparse induced subgraph with at least $x^{2at}\abs G$ vertices.
	\end{lemma}
	\begin{proof}
		By the first condition of the lemma, there exists $y\in[x^a,c]$ minimal of the form $c^{a^s}$ for some integer $s\ge0$, such that $G$ has a $y$-sparse induced subgraph $F$ with $\abs F\ge y^{2t}\abs G$.
		If $y\ge x$, then by the second condition of the lemma and since $a\ge2$, $F$ has a $y^a$-sparse induced subgraph with at least $y^{at}\abs F\ge y^{at+2t}\abs G\ge y^{2at}\abs G$ vertices;
		but this contradicts the minimality of $y$ since $x^a\le y^a<y$.
		Thus $x^a\le y<x$; and so $F$ is $x$-sparse, and $\abs F\ge y^{2t}\abs G\ge x^{2at}\abs G$.
		This proves \cref{lem:sparsify}.
	\end{proof}
	To deduce the \erh{} property, we will apply the following simple result that is essentially~\cite[Lemma 3.3]{density6}. We give a proof for completeness.
	
	\begin{lemma}
		\label{lem:eh}
		Let $\mac C$ be a class of graphs closed under taking induced subgraphs.
		Suppose that there exists $b\ge2$ such that for every $x\in(0,2^{-b}]$ and every $G\in\mac C$,
		either:
		\begin{itemize}
			\item $G$ has an $x$-restricted induced subgraph with at least $x^b\abs G$ vertices; or
			
			\item there is a complete or anticomplete $(k,\abs G/k^b)$-blockade in $G$, for some integer $k\ge2$.
		\end{itemize}
		Then for $a=4b^2$, every graph $G\in \mac C$ has a clique or stable set of size at least $\abs G^{1/a}$.
	\end{lemma}
	\begin{proof}
		A {\em cograph} is a graph with no induced four-vertex path;
		and it is well-known that every $k$-vertex cograph has a clique or stable set of size at least $k^{1/2}$.
		Thus, it suffices to prove by induction that every $G\in\mac C$ contains an induced cograph of size at least $\abs G^{1/(2b^2)}$.
		We may assume $\abs G> 2^{2b^2}$.
		Let $x:=\abs G^{-1/(2b)}\in(0,2^{-b})$.
		By the hypothesis, $G$ contains either:
		\begin{itemize}
			\item an $x$-restricted induced subgraph $F$ with $\abs F\ge x^b\abs G$; or
			
			\item a complete or anticomplete $(k,\abs G/k^b)$-blockade, for some integer $k\ge2$.
		\end{itemize}
		
		If the first bullet holds, then since $\abs F\ge x^b\abs G\ge x^{-1}$,
		Tur\'an's theorem gives a clique or stable set in $F$ of size at least
		$(2x)^{-1}>x^{-1/2}=\abs G^{1/(4b)}$.
		If the second bullet holds, then by induction and since $k\ge2$, the graph $G$ contains an induced cograph of size at least
		$$k(\abs G/k^b)^{1/(2b^2)}=k^{1-1/(2b)}\abs G^{1/(2b^2)}>\abs G^{1/(2b^2)}.$$
		This proves \cref{lem:eh}.
	\end{proof}

	\subsection{Forests and niceness}
	We now collect facts on excluding forests.
	We require the following fundamental result of Chudnovsky, Scott, Seymour, and Spirkl~\cite{csss2020} on linear anticomplete pairs in sparse graphs with a forbidden induced forest; a simpler proof for excluding an induced path had previously been obtained by Bousquet, Lagoutte, and Thomass\'e~\cite{blt2015}.
	
	\begin{theorem}
		[Chudnovsky, Scott, Seymour, and Spirkl]
		\label{thm:sparsetree}
		For every forest $H$, there exists $\eta\in(0,\frac12]$ such that every $\eta$-sparse $H$-free graph $G$ with $\abs G>\eta^{-1}$ contains an anticomplete $(2,\eta\abs G)$-blockade.
	\end{theorem}
	
	Iterating \cref{thm:sparsetree} gives:
	\begin{lemma}
		\label{lem:treesparse}
		For every forest $H$, there exists $s\ge2$ such that for every $\eps\in(0,\frac12]$ and for
		every $\eps^s$-sparse $H$-free graph $G$ with $\abs G>\eps^{-s}$,
		there is an anticomplete $(\eps^{-1},\eps^s\abs G)$-blockade in $G$.
	\end{lemma}
	\begin{proof}
		We claim that $s:=2\log\eta^{-1}$ suffices,
		where $\eta\in(0,\frac12]$ is given by \cref{thm:sparsetree}.
		To see this, let $G$ be an $\eps^s$-sparse $H$-free graph;
		and let $k\in\{0,1,\ldots,\ceil{\log\frac1\eps}\}$ be maximal such that there is an anticomplete $(2^k,\eta^{k}\abs G)$-blockade $(B_1,\ldots,B_{2^k})$ in $G$.
		
		If $k<\ceil{\log\frac1\eps}$, then
		for every $i\in[2^k]$, since 
		$\abs{B_i}\ge \eta^k\abs G\ge \eta^{-1} \cdot\eta^{2\log\frac1\eps}\abs G=\eta^{-1}\cdot \eps^s\abs G>\eta^{-1}$,
		$G[B_i]$ is $\eta$-sparse;
		and so there are disjoint and anticomplete $D_{2i-1},D_{2i}\subset B_i$ with $\abs{D_{2i-1}},\abs{D_{2i}}\ge \eta\abs{B_i}\ge \eta^{k+1}\abs G$.
		Then $(D_1,D_2,\ldots,D_{2^{k+1}})$ is an anticomplete $(2^{k+1},\eta^{k+1}\abs G)$-blockade in $G$ while $k+1\le\ceil{\log\frac1\eps}$,
		contrary to the maximality of $k$.
		
		Hence $k=\ceil{\log\frac1\eps}$.
		Now, since $2^k\ge 2^{\log\frac1\eps}=\eps^{-1}$
		and $\eta^k\ge\eta^{2\log\frac1\eps}=\eps^s$,
		there is an anticomplete $(2^k,\eta^k\abs G)$-blockade in $G$.
		This proves \cref{lem:treesparse}.
	\end{proof}
	
	A graph $H$ is {\em nice} if there exists $d\ge1$ such that for every $\eps\in(0,\frac12]$, every $H$-free graph $G$ with $\abs G>\eps^{-d}$ contains a $(1-\eps)$-semidense $(\eps^{-1},\eps^d\abs G)$-blockade.
	Niceness was introduced in~\cite{density7} (up to a slight change in the `noise' between non-pure pairs) as a natural step towards verifying the \erh{} property for all forests; and in~\cite{density7} it was also noted that if $H$ is nice then it is a forest.
	The following result asserts that, in order to show that a forest $H$ is nice, it suffices to prove its purity and the viral property of any forest obtained by removing a leaf of $H$.
	We omit the proof, which is a straightforward combination of \cref{lem:nearnice,lem:treesparse}.
	
	\begin{lemma}
		\label{lem:nice}
		Let $H$ be a forest with a leaf $v$.
		If $H\setminus v$ is viral and $H$ is pure, then $H$ is nice.
	\end{lemma}
	
	\subsection{Two key lemmas}
	\label{subsec:keylemmas}
	Given the results we have collected so far, it suffices to prove \cref{thm:path} via the following two intermediate results.
	
	\begin{lemma}
		\label{lem:pathpure}
		For every integer $t\ge3$, $P_t$ is pure.
	\end{lemma}

	\begin{lemma}
		\label{lem:niceeh}
		For every integer $t\ge2$, if $P_t$ has the \erh{} property and $P_{t+1}$ is nice, then $P_{t+1}$  has the \erh{} property.
	\end{lemma}
	
	\begin{proof}
		[Proof of \cref{thm:path}, assuming \cref{lem:pathpure,lem:niceeh}]
		We proceed by induction on $t\ge1$. The theorem is true for $t\in\{1,2\}$.
		For $t\ge2$, we aim to prove the theorem for $t+1$, assuming that it holds for $t$.
		Now, since $P_t$ has the \erh{} property, it is viral by \cref{thm:bfp}. Thus, since $P_{t+1}$ is pure by \cref{lem:pathpure}, we deduce that $P_{t+1}$ is nice by \cref{lem:nice}.
		Therefore, \cref{lem:niceeh} implies that $P_{t+1}$ has the \erh{} property. This proves \cref{thm:path}.
	\end{proof}
	The rest of this paper is mainly devoted to the proofs of \cref{lem:pathpure,lem:niceeh}.
	We will explain the proofs of \cref{lem:pathpure,lem:niceeh} in \cref{sec:split,sec:nicetoeh}, respectively.
	
	\section{The Gy\'arf\'as path argument via bookkeeping blockades and splitting paths}
	\label{sec:split}
	This section provides the proof of \cref{lem:pathpure}.
	To explain the key conceptual ideas behind the proof, we will assume familiarity with the well-known Gy\'arf\'as path argument~\cite{gyarfas1987}, the `near' \erh{} proof in~\cite{density5}, and the proof that $P_5$ is nice as shown in~\cite[Sections 5 and 6]{density7}.
	To make our exposition consistent and as clear as possible, we will focus on dense $P_t$-free graphs, as in~\cite{density5}.
	As such, any explanation involving~\cite[Sections 5 and 6]{density7} will assume dense $P_5$-free graphs (instead of sparse $\overline{P_5}$-free graphs); and it may be more intuitive to say `iterative densification' instead of `iterative sparsification' (needless to say, the fundamental ideas remain unchanged).
	For notational convenience, for a graph $G$, a {\em poly-blockade} in $G$ is a blockade in $G$ of width at least $\abs G/\poly(k)$, where $k$ is its length.
	
	\subsection{Bookkeeping blockades}
	\label{subsec:bookkeep}
	This subsection explains the first key idea towards \cref{lem:pathpure}.
	To describe the thought process leading to our idea, we will carefully review the material from~\cite{density5} and~\cite[Sections 5 and 6]{density7}.
	
	To begin our explanation, let us summarise the argument in~\cite[Section 5]{density7}, as follows.
	We are given a density threshold $x$ and a density parameter $y$ with $x\ll y$ sufficiently small, and a large $(1-y)$-dense induced subgraph $G$ of some $P_5$-free host graph $\Gamma$, with $\abs G\ge \poly(y)\abs{\Gamma}$.
	Our plan is to perform iterative densification; that is, either $G$ contains a sufficiently large induced subgraph with much higher density (so we can eventually push $y$ down to $x$), or it contains some pure or $(1-x)$-dense poly-blockade of length between $\poly(y^{-1})$ and $\poly(x^{-1})$ (so that this blockade would also be a poly-blockade in $\Gamma$).
	To carry out this process, we first observe that if $G$ is actually $(1-y^2)$-dense then we are done with the densification step; so we may assume there exists $v\in V(G)$ with at least $y^2\abs G$ non-neighbours in $G$.
	Let $A:=N_G(v)$ and $B:=V(G)\setminus N_G[v]$; then $\abs A\ge (1-O(y))\abs G$ and $\abs B\ge y^2\abs G$.
	Then $B$ is very dense to $A$; and so after a few averaging steps, an application of \cref{lem:comb} implies that either:
	\begin{itemize}
		\item we can remove a negligible portion from $A$ and at most half of $B$ to get $A'\subset A$ and $B_0\subset B$, such that $\abs{A'}\ge (1-O(y^{1/2}))\abs A\ge(1-O(y^{1/2}))\abs G$, $\abs{B_0}\ge \abs B/2\ge \Omega(y^2)\abs G$, and $A'$ is $(1-x)$-dense to $B_0$; or
		
		\item there is an anticomb $((a_i,B_i):i\in[k])$ from $A$ to $B$, where $y^{-1/2}\le k\le \poly(x^{-1})$ and $\abs{B_i}\ge \abs B/\poly(k)$ for all $i\in[k]$.
	\end{itemize}
	
	If the second bullet holds, then after a couple of structural arguments involving $P_5$, we can get a long and wide poly-blockade $\mac B$ of length $k$, such that $\mac B$ either is an anticomplete blockade inside some $G[B_i]$, or $\mac B$ is a pure sub-blockade of $(B_1,\ldots,B_k)$ (see \cref{subsec:split} below for more discussion).
	Thus, provided that we can never get such blockades $\mac B$, the first bullet above holds; and so, since $G[A']$ is still very dense, we can iterate this procedure (by applying \cref{lem:iteration} in the complement) to extract a $(1-x)$-dense blockade in $G$ of length at least $\Omega(y^{-1/2})$ and width at least $\poly(y)\abs G\ge \poly(y)\abs{\Gamma}$.
	Hence, $G$ contains either a pure or a $(1-x)$-dense poly-blockade of length between $\poly(y^{-1})$ and $\poly(x^{-1})$.
	
	Now, for $\Gamma$, we apply R\"odl's theorem \ref{thm:rodl} with some small constant density parameter $\eps>0$. The sparse case yields a linear anticomplete pair; and the dense case outputs either a pure or $(1-x)$-dense poly-blockade of length at most $\poly(x^{-1})$, or a $(1-x)$-dense induced subgraph with $\poly(x)\abs{\Gamma}$ vertices.
	Thus, we may assume that $\Gamma$ always contains a pure or $(1-x)$-dense poly-blockade of length at most $\poly(x^{-1})$; and it then remains to apply the conversion theorem~\cite[Theorem 6.1]{density7} to $\Gamma$ to prove the niceness of $P_5$. 
	
	Let us now delve deeper into the above iteration inside $G$. Given the aforesaid $A'$, let $F:=G[A']$. The iteration essentially says that either:
	\begin{itemize}
		\item $F$ is $(1-y^2)$-dense; or 
		
		\item there exists $v'\in A'$ with at least $y^2\abs F$ non-neighbours in $F$, such that for $B'=V(F)\setminus N_F[v']$, one of the following holds:
		\begin{itemize}
			\item there exists $A''\subset N_F(v')$ and $B_0'\subset B'$ such that $\abs{A''}\ge (1-O(y^{1/2}))\abs F$, $\abs{B_0'}\ge \abs{B'}/2\ge\Omega(y^2)\abs F$, and $A''$ is $(1-x)$-dense to $B_0'$; or
			
			\item there is an anticomb $((a_i',B_i'):i\in[k'])$ from $N_F(v')$ to $B'$, where $\Omega(y^{-1/2})\le k'\le \poly(x^{-1})$ and $\abs{B_i'}\ge \abs {B'}/\poly(k')$ for all $i\in[k']$.
		\end{itemize}
	\end{itemize}
	
	The hidden structure from this configuration is the extra vertex $v'$ with a decent number of non-neighbours in $F$, which plays the role of the aforesaid $v$.
	Our first change here is to pull out such vertices $v$ and $v'$ repeatedly for about $\poly(y^{-1})$ steps.
	Under the assumption that $G$ is large and has no $(1-y^2)$-dense induced subgraph with at least $\abs G/100$ vertices say, this will produce a `half' anticomb $((v_i,U_i):i\in[\ell])$ where $\ell\approx y^{-c}$ for some appropriately small $c>0$, such that:
	\begin{itemize}
		\item for all $i\in[\ell]$, $\abs{U_i}\ge\Omega(y^2)\abs G$ and $U_i\subset V(G)\setminus N_G[v_i]$ (so $\abs{U_i}<y\abs G$); and
		
		\item for all $i,j\in[\ell]$ with $i<j$, $\{v_j\}\cup U_j$ is a subset of $N_G(v_i)$, and is $(1-y^{2c})$-dense to $U_i$.
	\end{itemize}
	(Here, $v_1$ and $v_2$ are actually $v$ and $v'$ in the above discussion.)
	
	Given this `half' anticomb, for each $i\in[\ell]$ we can safely discard a small portion in $U_i$ consisting of those with a non-neighbour in $\{v_j:j>i\}$ to obtain $V_i\subset U_i$ with $\abs{V_i}\ge\abs{U_i}/2$, such that the sub-blockade $(V_1,\ldots,V_\ell)$ of $(U_1,\ldots,U_\ell)$ satisfies:
	\begin{itemize}
		\item for all $i\in[\ell]$, $\Omega(y^2)\abs G\le \abs{V_i}< y\abs G$ and $V_i\subset V(G)\setminus N_G[v_i]$;
		
		\item for all {\em distinct} $i,j\in[\ell]$, $\{v_j\}\cup V_j\subset N_G(v_i)$; and
		
		\item $(V_1,\ldots,V_\ell)$ is $(1-y^c)$-dense.
	\end{itemize}
	
	The pruning step is illustrated in \cref{fig:half-anticomb} below.
	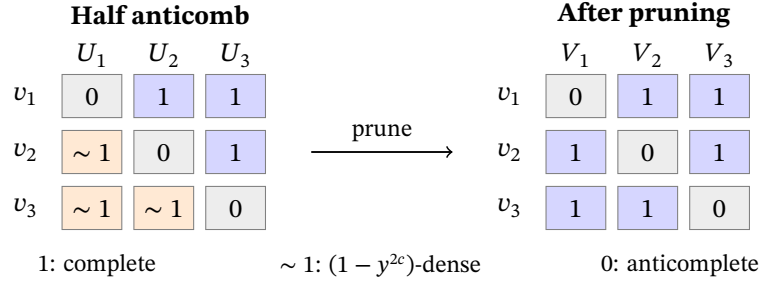
\begin{figure}[htbp]
		\centering
		\begin{tikzpicture}[
			font=\small,
			cell/.style={draw=black!50,minimum width=.79cm,minimum height=.57cm,inner sep=0pt},
			certain/.style={cell,fill=blue!16},
			almost/.style={cell,fill=orange!17},
			absent/.style={cell,fill=black!7}
			]
			\node[font=\small\bfseries] at (2.1,3.2) {Half anticomb};
			\node[font=\small\bfseries] at (8.5,3.2) {After pruning};
			\foreach \j/\x in {1/1.2,2/2.15,3/3.1} {
				\node at (\x,2.66) {$U_{\j}$};
			}
			\foreach \j/\x in {1/7.6,2/8.55,3/9.5} {
				\node at (\x,2.66) {$V_{\j}$};
			}
			\foreach \j/\y in {1/2.12,2/1.38,3/.64} {
				\node at (.32,\y) {$v_{\j}$};
				\node at (6.72,\y) {$v_{\j}$};
			}
			\node[absent]  at (1.2,2.12) {$0$};
			\node[certain] at (2.15,2.12) {$1$};
			\node[certain] at (3.1,2.12) {$1$};
			\node[almost]  at (1.2,1.38) {$\sim1$};
			\node[absent]  at (2.15,1.38) {$0$};
			\node[certain] at (3.1,1.38) {$1$};
			\node[almost]  at (1.2,.64) {$\sim1$};
			\node[almost]  at (2.15,.64) {$\sim1$};
			\node[absent]  at (3.1,.64) {$0$};
			\node[absent]  at (7.6,2.12) {$0$};
			\node[certain] at (8.55,2.12) {$1$};
			\node[certain] at (9.5,2.12) {$1$};
			\node[certain] at (7.6,1.38) {$1$};
			\node[absent]  at (8.55,1.38) {$0$};
			\node[certain] at (9.5,1.38) {$1$};
			\node[certain] at (7.6,.64) {$1$};
			\node[certain] at (8.55,.64) {$1$};
			\node[absent]  at (9.5,.64) {$0$};
			\draw[->,line width=.7pt] (4.1,1.38) -- node[above,font=\footnotesize] {prune} (5.97,1.38);
			\node[font=\footnotesize,anchor=west] at (.35,-.12) {$1$: complete};
			\node[font=\footnotesize,anchor=west] at (3.55,-.12) {$\sim1$: $(1-y^{2c})$-dense};
			\node[font=\footnotesize,anchor=west] at (7.8,-.12) {$0$: anticomplete};
		\end{tikzpicture}
		\caption{Apex-tooth adjacencies for three indices; the general pattern is the same. Initially, later apexes are only almost complete to earlier teeth. Removing from each $U_i$ the vertices missed by a later apex gives $V_i\subset U_i$ and makes every off-diagonal apex-tooth pair complete. The teeth still form a $(1-y^c)$-dense blockade; their cross-pairs need not be complete.}
		\label{fig:half-anticomb}
	\end{figure}
	
	Such an anticomb was considered in~\cite[Chapter 6]{2025thes}, and also implicitly appeared in~\cite{density3} (both in the sparse setting, by passing to the complement).
	A special feature of this anticomb is that the `teeth' poly-blockade $(V_1,\ldots,V_\ell)$ has its length $\ell\approx y^{-c}$, width $\poly(y)\abs G\ge \poly(y)\abs{\Gamma}$, and `cross' density $1-y^c$ all {\em inherited} from the global density parameter $y$ of $G$.
	Now we can apply \cref{lem:comb} to each pair $(V_j,V_i)$ with $i<j$, to obtain either a $(1-x)$-dense poly-sub-blockade of $(V_1,\ldots,V_\ell)$, or an anticomb from $V_j$ to $V_i$ for some $i<j$ (in other words, a small anticomb `inside' a large anticomb).
	The discussion on our change so far is actually a proof sketch of the following standalone lemma, which may be of independent interest and whose proof we provide in \cref{sec:buildcombs}.
	
	\begin{lemma}
		\label{lem:getcomb0}
		Let $x,y\in(0,2^{-99}]$ with $x\le y$. Then every $(1-y)$-dense graph $G$ with $\abs G> y^{-1}$ contains either:
		\begin{itemize}
			
			\item a $(1-y^2)$-dense induced subgraph with at least $\abs G/2$ vertices;
			
			\item a $(1-x)$-dense $(y^{-1/6},y^3\abs G)$-blockade; or
			
			\item adjacent $u,v\in V(G)$ and a $(k,\abs G/k^{99})$-anticomb $((a_i,B_i):i\in[k])$ in $G\setminus\{u,v\}$, where $k\in[y^{-1/24},x^{-1}]$, such that $a_i\in N_G(v)\setminus N_G(u)$ and $B_i\subset N_G(u)\setminus N_G(v)$ for all $i\in[k]$.
		\end{itemize}
	\end{lemma}
	
	We remark that this standalone lemma (switched to the complement) can be incorporated into the first round of iterative sparsification in~\cite[Section 5]{density7}.
	
	Let us now return to the aforesaid `anticomb inside anticomb' idea that leads to the third outcome of \cref{lem:getcomb0}, and discuss how it may be extended to grow induced paths longer than $P_5$.
	Let $\mac V=(V_1,\ldots,V_\ell)$ be the above blockade of teeth with apexes $v_1,\ldots,v_\ell$, and assume that for some $i<j$, there is a smaller anticomb $((a_p,B_p):p\in[k])$ from $V_j$ to $V_i$, such that the blockade $\mac B=(B_1,\ldots,B_k)$ (a slight abuse of notation with respect to the above summary of~\cite[Section 5]{density7}) has width at least $\abs{V_i}/\poly(k)\ge \abs G/\poly(k)\ge \abs{\Gamma}/\poly(k)$ (and so it is a poly-blockade in $G$ and in $\Gamma$).
	This configuration yields plenty of induced $P_3$ patterns, for instance via $v_i\text-a_p\text-B_q$ for all distinct $p,q\in[k]$.
	We also observe that:
	\begin{itemize}
		\item induced $P_3$ patterns could also arise if $((a_p,B_p):p\in[k])$ was a {\em comb} instead of an anticomb, as one could go through $v_i\text-a_p\text-B_p$ for any $p\in[k]$; and
		
		\item for a similar reason, one could also obtain many $P_3$ patterns if the apexes $a_1,\ldots,a_k$ actually come from $\bigcup_{j\ne i}V_j$, by starting from $v_i$ in the case of anticomb and from $\{v_j:j\ne i\}$ in the case of comb. 
	\end{itemize} 
	
	The two ways of obtaining a $P_3$ from nested (anti)combs are shown in \cref{fig:nested-combs} below.
	\begin{figure}[htbp]
		\centering
		\begin{tikzpicture}[
			panel/.style={draw=black!30,rounded corners=3pt,fill=black!1},
			apex/.style={circle,draw=blue!65!black,fill=blue!12,minimum size=.69cm,font=\small},
			tooth/.style={draw=blue!65!black,rounded corners=2pt,fill=blue!6,
				minimum width=.96cm,minimum height=.61cm,font=\small},
			edge/.style={draw=black!80,line width=.85pt},
			nonedge/.style={draw=red!65!black,densely dashed,line width=.7pt}
			]
			\draw[panel] (0,0) rectangle (5.95,3.15);
			\node[font=\small\bfseries] at (2.98,2.78) {Inner anticomb};
			\node[apex] (vanti) at (.84,1.63) {$v_i$};
			\node[apex] (aanti) at (2.81,1.63) {$a_p$};
			\node[tooth] (bqanti) at (5.02,1.63) {$B_q$};
			\node[tooth] (bpanti) at (5.02,.63) {$B_p$};
			\draw[edge] (vanti) -- (aanti) -- (bqanti);
			\draw[nonedge] (vanti.north) .. controls (2.6,2.38) and (3.45,2.38) .. (bqanti.north);
			\draw[nonedge] (aanti.south east) -- (bpanti.north west);
			\node[font=\footnotesize] at (2.81,.87) {$a_p\in V_j$};
			\node[font=\footnotesize] at (2.98,.26) {$B_p,B_q\subset V_i$};
			\begin{scope}[xshift=6.35cm]
				\draw[panel] (0,0) rectangle (5.95,3.15);
				\node[font=\small\bfseries] at (2.98,2.78) {Inner comb};
				\node[apex] (vcomb) at (.84,1.63) {$v_i$};
				\node[apex] (acomb) at (2.81,1.63) {$a_p$};
				\node[tooth] (bpcomb) at (5.02,1.63) {$B_p$};
				\node[tooth] (bqcomb) at (5.02,.63) {$B_q$};
				\draw[edge] (vcomb) -- (acomb) -- (bpcomb);
				\draw[nonedge] (vcomb.north) .. controls (2.6,2.38) and (3.45,2.38) .. (bpcomb.north);
				\draw[nonedge] (acomb.south east) -- (bqcomb.north west);
				\node[font=\footnotesize] at (2.81,.87) {$a_p\in V_j$};
				\node[font=\footnotesize] at (2.98,.26) {$B_p,B_q\subset V_i$};
			\end{scope}
		\end{tikzpicture}
		\caption{How an inner (anti)comb extends the outer anticomb. For an inner anticomb choose $q\ne p$ and a vertex of $B_q$; for an inner comb choose a vertex of $B_p$. In both cases, $v_i$ is adjacent to $a_p\in V_j$ and anticomplete to the chosen tooth in $V_i$.}
		\label{fig:nested-combs}
	\end{figure}
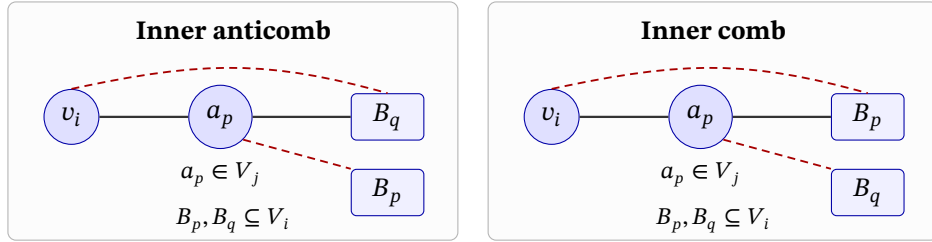
	
	Continuing in this fashion, if we could build a smaller comb/anticomb from $\bigcup_{q\ne p}B_q$ to $B_p$ for some $p\in[k]$, then this would lead to many induced $P_4$ patterns.
	In general, for longer induced $P_t$ patterns, it would be highly desirable to keep building `(anti)combs inside (anti)combs' in this way, where each new inner (anti)comb has its teeth blocks lying in a large tooth block of its preceding outer (anti)comb, and its apexes are `scattered' around the other large tooth blocks.
	Generating such small (anti)combs, however, already appears challenging at $t=4$, because we have no information about the cross density within $\mac B$ (in fact, we cannot even force every two blocks of $\mac B$ to be sparse or dense to each other).
	
	To examine this issue in depth, let us recall that we could construct $\mac B$ because of the cross density $1-y^c$ between $V_1,\ldots,V_\ell$, which in turn came from the global density parameter $y$ of $G$.
	Following this global density idea, the argument in~\cite{density5} {\em shrinks} each $G[V_i]$ into a {\em sufficiently large} $y$-restricted induced subgraph $G[W_i]$, and then assumes that $G[W_i]$ is $(1-y)$-dense for all $i\in[\ell]$, because the $y$-sparse case yields anticomplete blockades of length $\poly(y^{-1})$ by \cref{lem:treesparse}.
	The proof method in~\cite{density5} then essentially builds the blockade $\mac B$ from $\bigcup_{j\ne i}W_j$ to $W_i$ for some $i$, such that the teeth of $\mac B$ are now guaranteed to form a $(1-y^c)$-dense blockade inside $G[W_i]$.
	Repeating this leads to the process of constructing smaller anticombs with prescribed cross density $1-y^c$; the whole process is formulated in~\cite{density5} as an analogue of the Gy\'arf\'as path argument.
	The difficulty now lies in the {\em width} of $(W_1,\ldots,W_\ell)$, that is, how much one can afford to shrink each $G[V_i]$ into a $y$-restricted induced subgraph.
	Asking for the shrinkage factor to be $\poly(y)$ is not feasible, because this would amount to assuming the polynomial R\"odl property that is essentially our desired conclusion.
	As an alternative, the argument in~\cite{density5} chooses the most natural shrinkage factor $\poly(x)$, which means each $W_i$ is only guaranteed to have size $\poly(x)\abs{V_i}\approx \poly(x)\abs G\approx \poly(x)\abs{\Gamma}$.
	This, in turn, breaks the favourable `poly' property of $(W_1,\ldots,W_\ell)$, since $\ell\approx y^{-c}$; and the same situation arises for $\mac B$ as well, because accommodating its $(1-y^c)$-denseness forces its length to be at most $y^{-c}$.
	
	Let us now explain how generating highly restricted induced subgraphs can be done in~\cite{density5}.
	This is because of the following `transference' theorem~\cite[Theorem 2.1]{density5} (slightly modified to be consistent with other results in this paper) that allows one to `transfer' blockades into density:
	\begin{theorem}
		\label{thm:trans}
		Let $f\colon(0,\frac12]\to\mab R_{\ge0}$ be a non-increasing function satisfying $1<f(x)\le 1/x$ for all $x\in(0,\frac12]$.
		Let $H$ be a graph, and suppose that there exists $b\ge2$ such that for every $x\in(0,2^{-b}]$ and every $H$-free graph $F$ with $\abs F>x^{-b}$,
		there is an $x$-sparse or $(1-x)$-dense $(f(x),{x^b\abs F})$-blockade in $F$.
		Then there exists $d\ge2$ such that for every $\eps\in(0,\frac12]$ and for $\delta>0$ satisfying
		\[\log\frac1\delta=\frac{d(\log\frac1\eps)^2}{\log(f(\eps))},\]
		every $H$-free graph $G$ contains an $\eps$-restricted induced subgraph with at least $\delta\abs G$ vertices.
	\end{theorem}
	(A version for counting copies of $H$ was previously proved in~\cite{bnss2024}.)
	
	The argument in~\cite{density5} then carefully selects a sharply decreasing sequence $(y_s)_{s\ge0}$ of positive functions of the density threshold $x$ and a sharply decreasing sequence $(c_s)_{s\ge0}$ of positive numbers tending to $0$, starting with $y_0=1/2$ and $c_0=1$, such that \cref{thm:trans} allows for the following:
	\begin{quote}
		\it For every $s\ge0$, if every $P_t$-free graph $F$ with $\abs F>\poly(x^{-1})$ contains an $x$-sparse or $(1-x)$-dense $(y_s^{-c_s},\poly(x)\abs F)$-blockade, then every $P_t$-free graph $\Gamma$ contains a $y_{s+1}$-restricted induced subgraph with at least $\poly(x)\abs{\Gamma}$ vertices.
	\end{quote}
	
	This facilitates induction on $s$, because as we discussed above, the proof method in~\cite{density5} allows one to build an anticomplete or $(1-x)$-dense $(y_{s+1}^{-c_{s+1}},\poly(x)\abs G)$-blockade in every $P_t$-free $y_{s+1}$-restricted graph $G$ with $\abs{ G}>\poly(x^{-1})$. The base case $s=0$ follows from the fundamental theorem of Erd\H os and Hajnal~\cite{eh1989} that for every $H$, every sufficiently large $H$-free graph $F$ contains an $x$-sparse or $(1-x)$-dense $(2,\poly(x)\abs F)$-blockade (see also~\cite[Lemma 2.1]{fs2008}).
	The choice of $(y_s)_{s\ge0}$ then leads to the `near' \erh{} bound $2^{(\log \abs{\Gamma})^{1-o(1)}}$.
	As such, the argument in~\cite{density5} can be summarised as follows:
	\begin{quote}
		\it Start from long and wide blockades, then use \cref{thm:trans} to generate density from blockades, then use density to build anticombs, then use anticombs to construct longer blockades that are still sufficiently wide.
	\end{quote}
	
	There are two main shortcomings with this approach:
	\begin{itemize}
		\item The favourable `poly' property of the aforesaid $(W_1,\ldots,W_\ell)$ and $\mac B$ was already broken in the second step.
		
		\item Using high density could only produce anticombs and blockades that are {\em uniformly} anticomplete or dense; meanwhile, our goal is to produce $x$-pure poly-blockades, and we noted above that combs may also be useful in growing induced paths.
	\end{itemize}
	
	If our final goal is just to produce sufficiently long $x$-pure poly-blockades, then generating highly restricted induced subgraphs via \cref{thm:trans} may not really be needed.
	Thus, our first key idea in the proof of \cref{lem:pathpure} is to produce $x$-pure blockades within each aforesaid $G[V_i]$ {\em directly}, without going through density, so that the two shortcomings above can be avoided.
	In more detail, since we cannot build a long induced path, given the above discussion we can assume that for the aforesaid $(V_1,\ldots,V_\ell)$, for all $i\in[\ell]$ there is no appropriate comb or anticomb going from $\bigcup_{j\ne i}V_j$ to $V_i$.
	We now aim to design a suitable structural framework (`splitting paths', see \cref{subsec:split} below) which essentially implies that under this assumption, for each $i\in[\ell]$ and for every induced subgraph $F_i$ of $G[V_i]$ with $\abs {F_i}\ge \poly(y)\abs{V_i}$, either:
	\begin{itemize}
		\item $F_i$ contains an $x$-pure poly-blockade; or
		
		\item $\bigcup_{j\ne i}V_j$ is $x^d$-pure to a significant portion of $V(F_i)$ (for some large $d$, note that the cross density $1-y^c$ will no longer be needed).
	\end{itemize} 
	
	For every $i$, if the second outcome occurs for some $F_i$ with $D_i$ the significant portion of interest, then we can use \cref{lem:xpure} to prune $(D_1,\ldots,D_\ell)$ a little further to obtain an $x$-pure poly-sub-blockade.
	If the first outcome holds for some fixed $i$ and every admissible $F_i$, then we can `bookkeep' the poly-blockades within $G[V_i]$ to eventually produce a sufficiently long $x$-pure poly-blockade inside it.
	The two routes are summarised in \cref{fig:bookkeeping-routes}.
	\begin{figure}[htbp]
		\centering
		\begin{tikzpicture}[
			box/.style={draw=blue!60!black,rounded corners=3pt,fill=blue!5,
				align=center,font=\small,minimum height=.76cm},
			arrow/.style={->,draw=black!75,line width=.75pt}
			]
			\node[box,text width=9.6cm] (start) at (0,3.65)
			{Teeth $V_1,\ldots,V_\ell$ of the initial anticomb; no further useful cross-block (anti)comb};
			\node[box,text width=4.8cm,minimum height=1.13cm] (local) at (-2.85,2.05)
			{For some fixed $i$, every large $F_i\subseteq G[V_i]$ has a local $x$-pure poly-blockade};
			\node[box,text width=4.8cm,minimum height=1.13cm] (cross) at (2.85,2.05)
			{For every $i$, some large $F_i\subseteq G[V_i]$ gives $D_i$ with external $x^d$-purity};
			\node[box,text width=4.8cm] (keep) at (-2.85,.65)
			{Bookkeep inside $G[V_i]$};
			\node[box,text width=4.8cm] (prune) at (2.85,.65)
			{Prune $(D_1,\ldots,D_\ell)$ using \cref{lem:xpure}};
			\node[font=\small\bfseries] at (-2.85,-.3) {long $x$-pure poly-blockade};
			\node[font=\small\bfseries] at (2.85,-.3) {global $x$-pure poly-blockade};
			\draw[arrow] (start.south west) -- (local.north);
			\draw[arrow] (start.south east) -- (cross.north);
			\draw[arrow] (local) -- (keep);
			\draw[arrow] (cross) -- (prune);
			\draw[arrow] (keep.south) -- (-2.85,-.08);
			\draw[arrow] (prune.south) -- (2.85,-.08);
		\end{tikzpicture}
		\caption{The two ways to assemble local information from the teeth of the initial anticomb, once further useful (anti)combs have been excluded. The left-hand route replaces generating highly restricted induced subgraphs.}
		\label{fig:bookkeeping-routes}
	\end{figure}
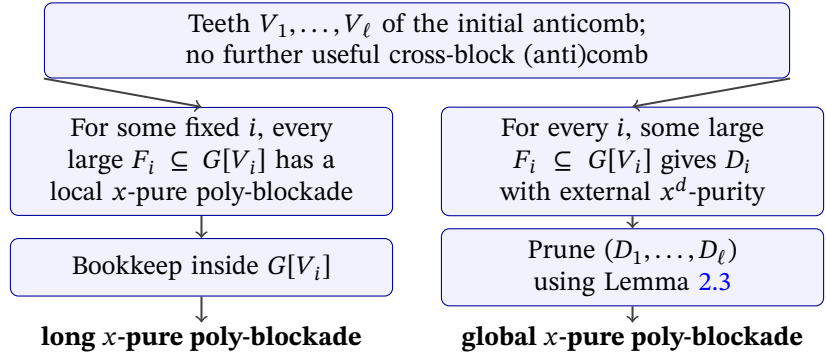
	
	Compared to recent works on \erh{}~\cite{density3,density4,density5,density6,density7}, the chief difference of this `bookkeeping blockades' idea lies in the left-hand route in \cref{fig:bookkeeping-routes}, that is, sufficiently long (poly-)blockades no longer come mainly from the global density of any large induced subgraph, but instead are built in a much more direct way by `bookkeeping' blockades obtained throughout the process.
	As such, our proof of \cref{lem:pathpure} will not involve iterative densification.
	Our argument will employ the following simple variant of the conversion theorems~\cite[Theorem 6.1]{density7} and~\cite[Theorem 4.2]{density4}, which in turn can all be viewed as distant relatives of \cref{thm:trans}.
	
	\begin{theorem}
		\label{thm:convsemi}
		Let $b\ge2$, and let $x,y\in(0,\frac12]$ with $x\le y^{2b+1}$.
		Let $G$ be a graph, such that for every induced subgraph $F$ of $G$ with $\abs F\ge y^{b}\abs G$, there exist an integer $k\in[2,1/x]$ and an $x$-pure $(k,\abs F/k^b)$-blockade in $F$.
		Then $G$ contains an $xy^{-2b}$-pure $(\ell,\abs G/\ell^{2b})$-blockade for some $\ell\in[y^{-1},x^{-1}]$.
	\end{theorem}
	\begin{proof}
		Let $s\ge1$ be maximal such that $G$ contains a blockade $(A_1,\ldots,A_s)$ satisfying:
		\begin{itemize}
			\item $\abs{A_i}\ge y^{2b}\abs G$ for all $i\in[s]$;
			
			\item $\sum_{i\in[s]}\abs{A_i}^{1/b}\ge \abs G^{1/b}$; and
			
			\item for all distinct $i,j\in[s]$, either every vertex in $A_i$ has fewer than $x\abs G$ neighbours in $A_j$, or every vertex in $A_i$ has fewer than $x\abs G$ non-neighbours in $A_j$.
		\end{itemize}
		(These conditions are clearly satisfied for $s=1$, by taking $A_1=V(G)$.)
		
		If $s\ge y^{-1}$ then we are done since $x\abs G\le xy^{-2b}\abs{A_i}$ for all $i\in[s]$.
		So we may assume that $s<y^{-1}$. Since $b\ge2$, $y\le\frac12$, and $x\le y^{2b+1}$, we have $s<y^{-1}\le y^{-2b}\le x^{-1}$ and $xy^{-2b}\le y\le\frac12$.
		Let $i\in[s]$ satisfy $\abs{A_i}=\max_{j\in[s]}\abs{A_j}$; then $s\abs{A_i}^{1/b}\ge \abs G^{1/b}$ and so $\abs {A_i}\ge \abs G/s^b\ge \abs G/s^{2b}$ and $\abs{A_i}\ge \abs G/s^b\ge y^b\abs G$.
		Thus, the hypothesis gives an $x$-pure $(k,\abs{A_i}/k^b)$-blockade $(B_1,\ldots,B_k)$ in $G[A_i]$.
		If $k\ge y^{-1}$ then
		\[\abs{B_p}\ge\abs{A_i}/k^b\ge y^b\abs G/k^b\ge \abs G/k^{2b}\quad\text{ for all $p\in[k]$}\]
		and we are done.
		Thus, suppose for a contradiction that $k<y^{-1}$.
		Then $\abs{B_p}\ge\abs{A_i}/k^b\ge y^{2b}\abs G$ for all $p\in[k]$.
		Observe that $\sum_{p\in[k]}\abs{B_p}^{1/b}\ge\abs{A_i}^{1/b}$, which yields
		\[\sum_{j\in[s]\setminus\{i\}}\abs{A_j}^{1/b}+\sum_{p\in[k]}\abs{B_p}^{1/b}
		\ge \sum_{j\in[s]}\abs{A_j}^{1/b}\ge\abs G^{1/b}.\]
		Now, let $\mac B$ be the blockade obtained from $(A_1,\ldots,A_s)$ by replacing $A_i$ by the sequence $(B_1,\ldots,B_k)$ and maintaining the natural order. Then $\mac B$ has length $s+k-1>s$ since $k\ge2$. Moreover, $x\abs G\le xy^{-2b}\abs{A_j}\le\abs{A_j}/2$ and $x\abs G\le xy^{-2b}\abs{B_p}\le\abs{B_p}/2$ for all $j\ne i$ and $p\in[k]$, so edge counting rules out opposite density types in the two directions.
		\begin{itemize}
			
			\item for every $j\in[s]\setminus\{i\}$, either every vertex in $A_j$ has fewer than $x\abs G$ neighbours in $A_i\supset B_1\cup\cdots\cup B_k$, or every vertex in $A_j$ has fewer than $x\abs G$ non-neighbours in $A_i\supset B_1\cup\cdots\cup B_k$;
			
			\item for every $j\in[s]\setminus\{i\}$, either every vertex in $A_i\supset B_1\cup\cdots\cup B_k$ has fewer than $x\abs G$ neighbours in $A_j$, or every vertex in $A_i\supset B_1\cup\cdots\cup B_k$ has fewer than $x\abs G$ non-neighbours in $A_j$; and
			
			\item for all distinct $p,q\in[k]$, either every vertex in $B_p$ has fewer than $x\abs{B_q}\le x\abs G$ neighbours in $B_q$, or every vertex in $B_p$ has fewer than $x\abs{B_q}\le x\abs G$ non-neighbours in $B_q$.
		\end{itemize}
		
		Therefore, the blockade $\mac B$ violates the maximality of $s$. This proves \cref{thm:convsemi}.
	\end{proof}
	
	\subsection{Splitting paths}
	\label{subsec:split}
	We now explain the second key conceptual idea in the proof of \cref{lem:pathpure}.
	To begin with, we recall the following fundamental result, which was proved in the proof of~\cite[Lemma 4.1]{density7} (stated for the complement).
	
	\begin{lemma}
		\label{lem:conn}
		Let $k\ge2$ be an integer, and let $G$ be a graph with $\abs G\ge1$ whose components have size less than $\abs G/k$. Then $G$ contains an anticomplete $(k,\abs G/(2k))$-blockade.
	\end{lemma}
	
	In~\cite[Section 5]{density7}, this was applied to the `teeth' $B_1,\ldots,B_k$ of the anticomb given in the third outcome of \cref{lem:getcomb0}, assuming that each $B_i$ is sufficiently large.
	Provided that a long anticomplete poly-blockade never occurs in any $G[B_i]$, \cref{lem:conn} would give a sub-blockade 
	$(D_1,\ldots,D_k)$ where $D_i\subset B_i$, $\abs{D_i}\ge\abs{B_i}/\poly(k)$, and $G[D_i]$ is connected for all $i\in[k]$.
	The blockade $(D_1,\ldots,D_k)$ is then pure, because if some $z\in D_j$ is mixed on $D_i$ for $i\ne j$ then one can build an induced $P_5$ by starting from the edge $va_i$.
	The mixed property works here because each $G[D_i]$ is {\em connected}, that is, for {\em any} partition $(X,Y)$ of $D_i$ into {\em nonempty} $X,Y$, there is an edge of $G[D_i]$ between $X$ and $Y$; and in particular this holds for $X=D_i\setminus N_G(z)$ and $Y=D_i\cap N_G(z)$.
	If one follows this approach to get $P_6$, then one would require that each $G[D_i]$ satisfies the following much stronger property:
	{\it for any partition $(X,Y)$ of $D_i$ into nonempty $X,Y$, there is an induced path $v_1\text-v_2\text-v_3$ in $G[D_i]$ with $v_2,v_3\in X$ and $v_1\in Y$.}
	A slight issue here is that $X$ could have just one vertex, so asking for this property to hold for {\em any} such partition $(X,Y)$ with $X,Y$ {\em merely} nonempty would not be feasible.
	To fix this, given the setting of \cref{lem:getcomb0} with density threshold $x$, and given our choice $X=D_i\setminus N_G(z)$ and $Y=D_i\cap N_G(z)$ as above, it would be natural to require $\abs X\ge x\abs{D_i}$ in our formulation, so that there is `room to spare' for $v_2$ and $v_3$.
	
	Thus, our `desired' extension of \cref{lem:conn} will take the following form.
	For $x>0$, we say that a graph $F$ is {\em $x$-splittable} if for every partition $(X,Y)$ of $V(F)$ with $\abs X\ge x\abs F$ and $Y\ne\emptyset$, there is an induced path $v_1\text-v_2\text-v_3$ in $F$ such that $v_2,v_3\in X$ and $v_1\in Y$; see \cref{fig:splittability} below.
	\begin{figure}[htbp]
		\centering
		\begin{tikzpicture}[
			panel/.style={draw=black,rounded corners=3pt,fill=black!1},
			inY/.style={circle,draw=orange!70!black,fill=orange!18,minimum size=.72cm,font=\small},
			inX/.style={circle,draw=blue!65!black,fill=blue!13,minimum size=.72cm,font=\small},
			edge/.style={draw=black!80,line width=.85pt},
			nonedge/.style={draw=red!65!black,densely dashed,line width=.7pt}
			]
			\draw[panel] (1,0) rectangle (4.6,2);
			\draw (2.8,2) -- (2.8,0);
			\node[font=\small\bfseries] at (2.8,2.4) {Connectedness};
			\node[inY] (yleft) at (1.9,1) {$v_1$};
			\node[inX] (xleft) at (3.7,1) {$v_2$};
			\draw[edge] (yleft) -- (xleft);
			\node[font=\footnotesize] at (1.9,.34) {$Y$};
			\node[font=\footnotesize] at (3.7,.34) {$X$};
			\begin{scope}[xshift=6.05cm]
				\draw[panel] (0,0) rectangle (6,2);
				\draw (2,2) -- (2,0);
				\node[font=\small\bfseries] at (3,2.4) {$x$-splittability};
				\node[inY] (yright) at (1,1) {$v_1$};
				\node[inX] (xoneright) at (3.1,1) {$v_2$};
				\node[inX] (xtworight) at (4.9,1) {$v_3$};
				\draw[edge] (yright) -- (xoneright) -- (xtworight);
				\draw[nonedge] (yright.north) .. controls (2.3,1.8) and (3.35,1.8) .. (xtworight.north);
				\node[font=\footnotesize] at (1,.34) {$Y$};
				\node[font=\footnotesize] at (4,.34) {$X$};
			\end{scope}
		\end{tikzpicture}
		\caption{For a partition $(X,Y)$ of $V(F)$ into nonempty parts, connectedness supplies an edge crossing from $Y$ to $X$. When $\abs X\ge x\abs F$, $x$-splittability supplies an induced $P_3$ with one endpoint in $Y$ and the other two vertices in $X$.}
		\label{fig:splittability}
	\end{figure}
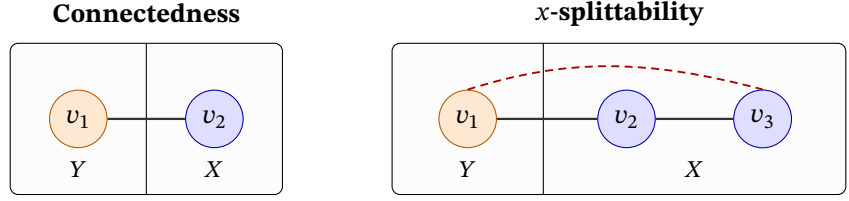
	
	In our extension, we aim to find, in every large graph, either a large $x$-splittable induced subgraph, or a long `homogeneous' poly-blockade analogous to the anticomplete one in \cref{lem:conn}.
	Intuitively, it appears difficult to force such a `homogeneous' blockade to be fully complete or fully anticomplete, because in our favourite path $v_1\text-v_2\text-v_3$ as above, $v_1$ is mixed on $X$ since it is adjacent to $v_2$ and nonadjacent to $v_3$.
	Nonetheless, if we wish to extend the argument of~\cite[Section 5]{density7} to $P_6$ in this way, then a fully complete or fully anticomplete outcome is not really needed, because a pure outcome would already suffice.
	It turns out that we can make this work in our extension of \cref{lem:conn}, as follows.
	
	\begin{lemma}
		\label{lem:p3}
		Let $x,y\in(0,2^{-8}]$ with $x\le y$. Then every graph $G$ with $\abs G> x^{-4}$ contains either:
		\begin{itemize}
			
			\item a pure $(\ell,\abs G/\ell^{8})$-blockade, for some integer $\ell\in[y^{-1/2},x^{-1}]$; or
			
			\item an $x$-splittable induced subgraph with at least $y^2\abs G$ vertices.
		\end{itemize}
	\end{lemma}
	
	This lemma plays a role in the proof of \cref{thm:extend} and may be of independent interest. 
	We prove this lemma in \cref{sec:p3}; it is not needed for the proof of \cref{lem:pathpure}.
	
	Going back to our discussion on $P_6$, by applying \cref{lem:p3} (with $y=1/k$) to the teeth $B_1,\ldots,B_k$ (as in the third outcome of \cref{lem:getcomb0}), we obtain either a long and wide pure blockade inside some $G[B_i]$, or a sub-blockade $(D_1,\ldots,D_k)$ where $D_i\subset B_i$ and $G[D_i]$ is $x$-splittable for all $i\in[k]$.
	In the second scenario, for all distinct $i,j\in[k]$, each vertex in $D_j$ is anticomplete or $(1-x)$-dense to $D_i$ (else we get $P_6$; see \cref{fig:splitting-p6}); and so we can trim $(D_1,\ldots,D_k)$ a little more (as in the proof of \cref{lem:xpure}) to a $(1-O(x))$-semidense sub-blockade. 
	Hence, both scenarios will eventually lead to the niceness of $P_6$ via the argument in~\cite[Sections 5 and 6]{density7}.
	
	\begin{figure}[htbp]
		\centering
		\begin{tikzpicture}[
			tooth/.style={draw=blue!45!black,rounded corners=3pt,fill=blue!4},
			vertex/.style={circle,draw=blue!65!black,fill=blue!13,minimum size=.66cm,font=\small},
			outside/.style={vertex,draw=orange!70!black,fill=orange!18},
			edge/.style={draw=black!80,line width=.85pt},
			nonedge/.style={draw=red!65!black,densely dashed,line width=.7pt}
			]
			\draw[tooth] (2.91,.24) rectangle (4.11,1.44);
			\draw[tooth] (5,-.1) rectangle (10,1.78);
			\node[outside] (v) at (.59,.83) {$v$};
			\node[outside] (ai) at (1.91,.83) {$a_i$};
			\node[vertex] (z) at (3.51,.83) {$z$};
			\node[vertex] (wone) at (6,.83) {$v_1$};
			\node[vertex] (wtwo) at (7.5,.83) {$v_2$};
			\node[vertex] (wthree) at (9,.83) {$v_3$};
			\draw[edge] (v) -- (ai) -- (z) -- (wone) -- (wtwo) -- (wthree);
			\draw[nonedge] (z.north) .. controls (4.84,1.7) and (6.17,1.7) .. (wtwo.north);
			\draw[nonedge] (z.north) .. controls (5.34,2.44) and (7.17,2.44) .. (wthree.north);
			\node[font=\small] at (3.52,-.5) {$D_j$};
			\node[font=\small] at (7.5,-.5) {$D_i$};
			\node[font=\footnotesize] at (6,.175) {$Y$};
			\node[font=\footnotesize] at (8.25,.175) {$X$};
		\end{tikzpicture}
		\caption{Why a vertex $z\in D_j$ cannot have both a neighbour and at least $x\abs{D_i}$ non-neighbours in $D_i$ when $i\ne j$. For $Y=D_i\cap N_G(z)$ and $X=D_i\setminus N_G(z)$, splittability gives the induced path $v_1\text-v_2\text-v_3$ with $v_1\in Y$ and $v_2,v_3\in X$. The anticomb gives $va_i$ and $a_iz$, while $v$ misses all teeth and $a_i$ misses $D_i$, so the displayed $P_6$ is induced.}
		\label{fig:splitting-p6}
	\end{figure}
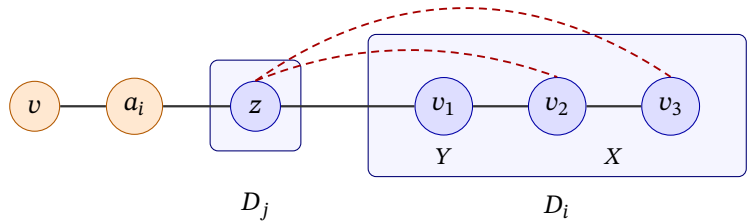
	
	What we have discussed so far is a special instance of a more general setting on `splitting paths', as follows.
	For an integer $r\ge 2$ and a graph $G$ with disjoint $A,B\subset V(G)$, we say that the ordered pair $(A,B)$ {\em splits} an induced $P_r$ in $G$, and write $A\to_rB$, if $G$ contains an induced path $v_1\text-\cdots\text-v_r$ such that $v_1\in A$ and $v_2,\ldots,v_r\in B$; and in this case, we also say that $(A,B)$ is a {\em directed split} of $v_1\text-\cdots\text-v_r$.
	For $x>0$, we say that $G$ is {\em $(r,x)$-splittable} if $Y\to_rX$ for all partitions $(X,Y)$ of $V(G)$ such that $\abs X\ge x\abs G$ and $Y\ne\emptyset$; thus $(3,x)$-splittability is equivalent to the aforesaid $x$-splittability.
	We remark that the more general idea of `splitting' any given graph $H$ with a distinguished vertex $g$ (that is, there is a copy $\varphi$ of $H$ in $G$ with $\varphi(g)\in A$ and $\varphi(V(H)\setminus\{g\})\subset B$)
	was a major ingredient behind the improvement~\cite{bnss2024} on the general bound towards \erh{}.
	However, in the rest of this paper we will only focus on splitting paths. 
	
	It remains unclear to us whether \cref{lem:p3} can be extended to $(4,x)$-splittability; so pursuing such an extension to get $P_7$ may not be feasible.
	But all is not lost; a closer look at the above sketch for $P_6$ actually suggests that we do not really need each $G[B_i]$ to contain a large $x$-splittable induced subgraph; and instead, $B_i\to_4B_j$ for some $i\ne j$ would already suffice.
	In this way, a more flexible `splitting' statement for disjoint vertex subsets in the host graph $G$ would be sufficient.
	To this end, let us record such a statement, which turns out to be a corollary of \cref{lem:p3}.
	
	\begin{lemma}
		\label{lem:splitp4}
		Let $x,y\in(0,2^{-8}]$ with $x\le y$, and let $G$ be a graph with disjoint $A,B\subset V(G)$ satisfying $\abs B>x^{-4}$. Then either:
		\begin{itemize}
			\item $A\to_4B$ in $G$;
			
			\item there exists $E\subset B$ with $\abs E\ge y^2\abs B$ such that every vertex in $A$ is either anticomplete or $(1-x)$-dense to $E$; or
			
			\item $G[B]$ contains a pure $(\ell,\abs B/\ell^8)$-blockade, for some integer $\ell\in[y^{-1/2},x^{-1}]$.
		\end{itemize}
	\end{lemma}
	
	To see how this result would be useful for the aforesaid `teeth' blockade $(B_1,\ldots,B_k)$, apply it with $y=1/k$ and $(A,B)=\left(\bigcup_{j\ne i}B_j,B_i\right)$ for each $i\in[k]$. We omit the details.
	
	Before discussing \cref{lem:splitp4} in more detail, let us deduce it from \cref{lem:p3}, as follows.
	
	\begin{proof}
		[Proof of \cref{lem:splitp4}, assuming \cref{lem:p3}]
		Assume that the second outcome fails.
		Since $\abs B>x^{-4}$, \cref{lem:p3} implies that $G[B]$ contains either:
		\begin{itemize}
			\item a pure $(k,\abs B/k^8)$-blockade, for some integer $k\in[y^{-1/2},x^{-1}]$; or
			
			\item an $x$-splittable induced subgraph $F$ with $\abs F\ge y^2\abs B$.
		\end{itemize}
		
		If the first bullet holds then the third outcome of the lemma holds; so assume that the second bullet holds.
		Let $E:=V(F)$. Since the second outcome of the lemma fails, there exists $v\in A$ with at least one neighbour and at least $x\abs E$ non-neighbours in $E$.
		Since $F$ is $(3,x)$-splittable, it contains an induced path $v_1\text-v_2\text-v_3$ such that $v_1\in E\cap N_G(v)$ and $v_2,v_3\in E\setminus N_G(v)$; and so $v\text-v_1\text-v_2\text-v_3$ verifies the first outcome of the lemma.
		This proves \cref{lem:splitp4}.
	\end{proof}
	
	We now give a few remarks about \cref{lem:splitp4}.
	Firstly, this lemma can be used to give a clean replacement of the comb argument in the first round of iterative `densification' for $P_5$ in~\cite[Section 5]{density7} (we omit the details).
	Secondly, and perhaps more importantly, a reader familiar with the literature on forbidden induced paths may recognise the similarity of \cref{lem:splitp4} with the following `splitting paths' formulation of the Gy\'arf\'as path argument~\cite{gyarfas1987}.
	
	\begin{lemma}
		[Gy\'arf\'as]
		\label{lem:gyarfas}
		For every integer $r\ge3$, and every graph $G$ with nonempty disjoint $A,B\subset V(G)$ such that $\chi(B)\ge r-2$ and $G[B]$ is connected, either:
		\begin{itemize}
			\item $A\to_rB$ in $G$ ({\em a directed split});
			
			\item for every $v\in A$, either $v$ is anticomplete to $B$ or $\chi(B\cap N_G(v))\ge\frac1{r-2}\chi(B)$ ({\em anticomplete or `dense'}); or
			
			\item there exists $u\in B$ with $\chi( N_{G[B]}(u))\ge\frac1{r-2}\chi(B)$ ({\em a large coherent substructure within $G[B]$}).
		\end{itemize}
	\end{lemma}
	
	(Here, $\chi$ denotes chromatic number; and for a graph $G$ with $S\subset V(G)$, we write $\chi(S)$ for $\chi(G[S])$ when there is no danger of ambiguity.)
	
	Given this observation, it would be natural to predict the following extension of \cref{lem:splitp4}.
	
	\begin{lemma}
		\label{lem:split0}
		For every integer $r\ge2$, there exists $b\ge2$ such that for every $x,y\in(0,\frac12]$ with $x\le y^b$, and every graph $G$ with disjoint $A,B\subset V(G)$ satisfying $\abs B>x^{-b}$, either:
		\begin{itemize}
			\item $A\to_rB$ in $G$;
			
			\item there exists $E\subset B$ with $\abs E\ge y^b\abs B$ such that $A$ is $x$-pure to $E$ in $G$; or
			
			\item $G[B]$ contains an $x$-pure $(k,\abs B/k^b)$-blockade, for some integer $k\in[y^{-1},x^{-1}]$.
		\end{itemize} 
	\end{lemma}
	
	This is true, and in fact follows from the case $y=\frac12$, whose slight strengthening below is the actual main result of this section.
	
	\begin{lemma}
		\label{lem:split}
		For every integer $r\ge2$, there exists $d\ge2$ such that for every $x\in(0,2^{-d}]$ and every graph $G$ with disjoint $A,B\subset V(G)$ satisfying $\abs B>x^{-d}$, either:
		\begin{itemize}
			\item $A\to_rB$ in $G$;
			
			\item there exists $E\subset B$ with $\abs E\ge \abs B/2$ such that $A$ is $x$-pure to $E$ in $G$; or
			
			\item $G[B]$ contains an $x$-pure $(k,\abs B/k^d)$-blockade, for some integer $k\in[2,x^{-1}]$.
		\end{itemize} 
	\end{lemma}
	
	Before presenting its proof, let us use this result to deduce \cref{lem:split0,lem:pathpure}.
	
	\begin{proof}
		[Proof of \cref{lem:split0}, assuming \cref{lem:split}]
		Let $d\ge2$ be given by \cref{lem:split}; we claim that $b:=4d$ suffices.
		To this end, assume that the first and third outcomes fail. In particular, there does not exist any integer $\ell\in[y^{-1},x^{-3/2}]$ such that $G[B]$ contains an $x^{3/2}y^{-2d}$-pure $(\ell,\abs B/\ell^{2d})$-blockade.
		(Note that $x^{3/2}y^{-2d}\le x$; and if such an $\ell$ exists with $\ell\ge x^{-1}$, then let $k:=\ceil{\ell^{1/2}}\le 2\ell^{1/2}\le 2x^{-3/4}\le x^{-1}$ and take only $k\ge\ell^{1/2}\ge x^{-1/2}\ge y^{-1}$ blocks from the blockade.)
		Thus, by \cref{thm:convsemi} with $(x,y,b,G)=(x^{3/2},y,d,G[B])$, there exists $D\subset B$ with $\abs D\ge y^d\abs B$ such that there does not exist any integer $k\in[2,x^{-3/2}]$ for which $G[D]$ contains an $x^{3/2}$-pure $(k,\abs D/k^d)$-blockade.
		Since $\abs D\ge y^d\abs B>x^{d-b}\ge x^{-2d}$, the choice of $d$ and \cref{lem:split} (with $x^{3/2}$ replacing $x$) implies that either:
		\begin{itemize}
			\item $A\to_rD$ in $G$; or
			
			\item there exists $E\subset D$ with $\abs E\ge\abs D/2$ such that $A$ is $x^{3/2}$-pure to $E$ in $G$.
		\end{itemize}
		
		If the first bullet holds then the first outcome of the lemma holds, contrary to our assumption.
		If the second bullet holds then the second outcome of the lemma holds since $\abs E\ge\abs D/2\ge y^d\abs B/2\ge y^b\abs B$ by the choice of $b$. This proves \cref{lem:split0}.
	\end{proof}
	\begin{proof}
		[Proof of \cref{lem:pathpure}, assuming \cref{lem:split}]
		Let $d\ge2$ be given by \cref{lem:split} with $r=t+1$, and let $b:=8d\ge3$.
		Let $K:=2b(2b+1)+b$.
		It suffices to show that for every $\eps\in(0,\frac12]$ and every $P_{t+1}$-free graph $G$ with $\abs G>\eps^{-K}$, there is an $\eps$-pure $(\eps^{-1},\eps^K\abs G)$-blockade in $G$.
		To this end, let $x:=\eps^{2b+1}$. We claim that:
		
		\begin{clm}
			{\ref*{lem:pathpure}.1}
			\label{claim:nice}
			For every induced subgraph $F$ of $G$ with $\abs F\ge \eps^b\abs G$, there exist an integer $k\in[2,1/x]$ and an $x$-pure $(k,\abs F/k^b)$-blockade in $F$.
		\end{clm}
		\begin{subproof}
			Note that $\abs F\ge\eps^b\abs G\ge \eps^{b-K}=\eps^{-2b(2b+1)}=x^{-2b}$ by the choice of $K$ and $x$. Let $(A,B)$ be a partition of $V(F)$ with $\abs A=\floor{\abs F/2}$; then $\abs A\ge \abs F/4$ and $\abs B\ge \abs F/2$.
			By the choice of $d$ and \cref{lem:split}, either:
			\begin{itemize}
				\item $A\to_{t+1}B$ in $F$;
				
				\item there exists $E\subset B$ with $\abs E\ge\abs B/2$ such that $A$ is $x^2$-pure to $E$ in $F$; or
				
				\item $F[B]$ contains an $x$-pure $(k,\abs B/k^d)$-blockade, for some integer $k\in[2,x^{-2}]$.
			\end{itemize}
			
			The first bullet cannot hold since $F$ is $P_{t+1}$-free. 
			If the third bullet holds then the claim also holds since $\abs B/k^d\ge \frac12\abs F/k^d\ge \abs F/\ceil{k^{1/3}}^b$ by the choice of $b$ (note that $\ceil{k^{1/3}}\le 2k^{1/3}\le  2x^{-2/3}\le x^{-1}$ and $\ceil{k^{1/3}}^{8}\ge k$).
			Thus, we may assume that the second bullet holds. Then there exists $D\subset A$ with $\abs D\ge\abs A/2\ge\abs F/8$ such that $D$ is $x^2$-sparse or $(1-x^2)$-dense to $E$. Thus, there exists $E'\subset E$ with $\abs{E'}\ge(1-x)\abs E\ge \frac12\abs E$ such that $E'$ is $x$-sparse or $(1-x)$-dense to $D$; and so $(D,E')$ is $x$-pure.
			Hence, the claim holds with $k=2$, since $\abs D,\abs{E'}\ge\abs F/8\ge\abs F/2^b$.
			This proves \cref{claim:nice}.
		\end{subproof}
		
		Now, \cref{claim:nice,thm:convsemi} (with $y=\eps$) together give an $x\eps^{-2b}$-pure $(\eps^{-1},x^{2b}\abs G)$-blockade in $G$. Since $x\eps^{-2b}=\eps$ and $x^{2b}=\eps^{2b(2b+1)}\ge\eps^K$, this proves \cref{lem:pathpure}.
	\end{proof}
	
	\subsection{The main proof}
	\label{subsec:mainsplit}
	We now present the proof of \cref{lem:split}, with a strong emphasis on bookkeeping blockades and splitting paths, as follows.
	
	\begin{proof}
		[Proof of \cref{lem:split}]
		We proceed by induction on $r\ge2$.
		For $r=2$, if the first outcome fails then $A$ is anticomplete to $B$; and so the second outcome holds with $E=B$.
		
		Now, for $r\ge2$, we prove the lemma for $r+1$, assuming that it holds for $r$.
		By induction, there exists $b\ge2$ such that for every $x\in(0,2^{-b}]$, and every graph $G$ with disjoint $A,B\subset V(G)$ such that $\abs B>x^{-b}$, either:
		\begin{itemize}
			\item $A\to_rB$ in $G$;
			
			\item there exists $E\subset B$ with $\abs E\ge \abs B/2$ such that $A$ is $x$-pure to $E$ in $G$; or
			
			\item $G[B]$ contains an $x$-pure $(k,\abs B/k^b)$-blockade, for some integer $k\in[2,1/x]$.
		\end{itemize}
		
		We claim that $d:=50b(b+1)$ suffices.
		To this end, let $x\in(0,2^{-d}]$, and let $G$ be a graph with disjoint $A,B\subset V(G)$ such that $\abs B>x^{-d}$.
		Assume that the first and third outcomes of the lemma fail for $r+1$ and $d$; that is, $A\not\to_{r+1}B$ in $G$, and $G[B]$ contains no $x$-pure $(k,\abs B/k^d)$-blockade, for any integer $k\in[2,1/x]$.
		
		The following claim formalises the process of bookkeeping blockades; see  \cref{fig:bookkeep} for a proof illustration.
		
		\begin{clm}
			{\ref*{lem:split}.1}
			\label{claim:bookkeep}
			Let $(B_1,\ldots,B_k)$ be a blockade in $G[B]$ of length $k\in[2,x^{-2}]$ and width larger than $x^{-4b(b+1)}$. Then either:
			\begin{itemize}
				\item there exist distinct $i,j\in[k]$ such that $B_i\to_r B_j$;
				
				\item for some $i\in[k]$ and some integer $\ell\in[k,x^{-4b-2}]$, there is an $x$-pure $(\ell,\abs{B_i}/\ell^{2b})$-blockade in $G[B_i]$; or
				
				\item for each $i\in[k]$, there exists $D_i\subset B_i$ with $\abs{D_i}\ge \abs{B_i}/(4k^b)$, such that $(D_1,\ldots,D_k)$ is $x$-pure.
			\end{itemize}
		\end{clm}
		\begin{subproof}
			
			\begin{figure}[ht]
				\centering
				\begin{tikzpicture}[
					every node/.style={font=\small},
					stage/.style={draw=black!60,rounded corners=2pt,fill=blue!5,
						minimum width=1.4cm,minimum height=.55cm},
					kept/.style={stage,draw=blue!65!black,fill=blue!15},
					contain/.style={->,draw=black!65,line width=.6pt}
					]
					\node[stage] (bi) at (-2.7,2.55) {$B_i$};
					\node[stage] (bj) at (2.7,2.55) {$B_j$};
					\node[stage] (ci) at (-2.7,1.7) {$C_i$};
					\node[stage] (cj) at (2.7,1.7) {$C_j$};
					\node[stage] (ei) at (-2.7,.85) {$E_i$};
					\node[stage] (ej) at (2.7,.85) {$E_j$};
					\node[kept] (di) at (-2.7,0) {$D_i$};
					\node[kept] (dj) at (2.7,0) {$D_j$};
					\draw[contain] (bi) -- (ci);
					\draw[contain] (bj) -- (cj);
					\draw[contain] (ci) -- (ei);
					\draw[contain] (cj) -- (ej);
					\draw[contain] (ei) -- (di);
					\draw[contain] (ej) -- (dj);
					\node[font=\footnotesize,text=black!70] at (0,2.12) {conversion theorem};
					\node[font=\footnotesize,text=black!70] at (0,1.27) {induction};
					\node[font=\footnotesize,text=black!70] at (0,.42) {pairwise pruning};
					\draw[<->,draw=blue!65!black,line width=.7pt] (di) --
					node[fill=white,inner sep=2pt,font=\footnotesize,yshift=-0.4mm] {$x$-pure} (dj);
				\end{tikzpicture}
				\caption{The bookkeeping step in \cref{claim:bookkeep} when its first two outcomes fail. The inclusions hold for every index, and simultaneous pruning yields $|D_i|\ge |B_i|/(4k^b)$ with every pair $(D_i,D_j)$ $x$-pure.}
				\label{fig:bookkeep}
			\end{figure}
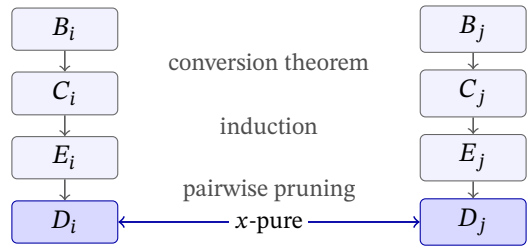
			
			Assume that the first two outcomes fail.
			(Note that the condition `no further useful cross-block (anti)comb' as shown in \cref{fig:bookkeeping-routes} is essentially contained in the failure of the first outcome.)
			Let $\xi:=x^{4b+2}\in(0,k^{-2b-1}]$.
			Since the second outcome fails, for each $i\in[k]$, \cref{thm:convsemi} (with $G=G[B_i]$, $x=\xi$, and $y=1/k$) gives $C_i\subset B_i$ with $\abs{C_i}\ge \abs{B_i}/k^b$ such that there does not exist an integer $q\in[2,1/\xi]$ for which there is a $\xi$-pure $(q,\abs{C_i}/q^b)$-blockade in $G[C_i]$.
			Let $A_i:=\bigcup_{j\ne i}C_j$.
			Since $\abs{C_i}\ge \abs{B_i}/k^b>x^{-b(4b+2)}=\xi^{-b}$, by the choice of $b$ and the induction hypothesis (with  $(A,B,x)=(A_i,C_i,\xi)$), either:
			\begin{itemize}
				\item $A_i\to_rC_i$ in $G$; or
				
				\item there exists $E_i\subset C_i$ with $\abs{E_i}\ge\abs{C_i}/2$ such that $A_i$ is $\xi$-pure to $E_i$ in $G$.
			\end{itemize}
			
			If the first bullet holds then there exists $j\ne i$ for which $C_j\to_rC_i$; and so the first outcome of the claim holds, violating our assumption.
			Hence, the second bullet holds.
			Now, for distinct $i,j\in[k]$, since $E_i$ is $\xi$-pure to $E_j$ and $E_j$ is $\xi$-pure to $E_i$, \cref{lem:xpure} (with $x^2$ replacing $x$, note that $\xi\le x^6$) gives $E_i^j\subset E_i$ and $E_j^i\subset E_j$ such that $\abs{E_i^j}\le x^4\abs{E_i}$, $\abs{E_j^i}\le x^4\abs{E_j}$, and $(E_i\setminus E_i^j,E_j\setminus E_j^i)$ is $x^4$-pure in $G$.
			Let $D_i:=E_i\setminus\left(\bigcup_{j\ne i}E_i^j\right)$ for all $i\in[k]$; then
			$\abs{D_i}\ge (1-kx^4)\abs{E_i}\ge \frac12\abs{E_i}$ since $k\le x^{-2}$.
			Therefore $(D_i,D_j)$ is $2x^4$-pure for all distinct $i,j\in[k]$; and so $(D_1,\ldots,D_k)$ is $x$-pure in $G$ because $2x^4\le x$.
			This proves \cref{claim:bookkeep}.
		\end{subproof}

		Now, we will handle the vertices in $A$ according to their degrees towards $B$ in $G$.
		Let $\delta:=2^{-20}$;
		the following claim is a replacement of R\"odl's theorem \ref{thm:rodl} with constant density parameter.
		
		\begin{clm}
			{\ref*{lem:split}.2}
			\label{claim:rodl}
			$A$ is $\delta$-pure to $B$.
		\end{clm}
		\begin{subproof}
			Suppose not; then there exists $v\in A$ with at least $\delta\abs B$ neighbours and at least $\delta\abs B$ non-neighbours in $B$.
			Let $X:=B\cap N_G(v)$ and $Y:=B\setminus N_G(v)$;
			then $\abs Y\ge \delta\abs B>\delta x^{-d}=2^{-20}x^{-d}\ge x^{-2b}$ by the choice of $d$.
			Thus, by the choice of $b$ (with $x^2$ replacing $x$), either:
			\begin{itemize}
				\item $X\to_rY$ in $G[B]$;
				
				\item there exists $Z\subset Y$ with $\abs Z\ge \abs Y/2$ such that $X$ is $x^2$-pure to $Z$ in $G$; or
				
				\item $G[Y]$ contains an $x^2$-pure $(k,\abs Y/k^b)$-blockade, for some integer $k\in[2,x^{-2}]$.
			\end{itemize}
			
			If the first bullet holds, then appending $v$ gives an induced $P_{r+1}$ with one endpoint $v$ and other vertices in $B$, contrary to our assumption.
			
			If the third bullet holds, then $\abs Y/k^b\ge \delta\abs B/k^b\ge \abs B/k^{b+20}\ge \abs B/\ceil{k^{1/3}}^d$; and so the third outcome of the lemma holds for $\ceil{k^{1/3}}\in[2,x^{-1}]$, a contradiction.
			
			Hence, the second bullet above holds.
			Then there exists $W\subset X$ such that $\abs W\ge\abs X/2\ge \delta\abs B/2\ge 2^{-d}\abs B$ and $W$ is $x^2$-sparse or $(1-x^2)$-dense to $Z$.
			Note that $\abs Z\ge \abs Y/2\ge \delta\abs B/2\ge 2^{1-d}\abs B$ by the choice of $d$.
			Assume without loss of generality that $W$ is $x^2$-sparse to $Z$.
			Then there exists $U\subset Z$ with $\abs U\ge (1-x)\abs Z\ge\frac12\abs Z\ge2^{-d}\abs B$ for which $U$ is $x$-sparse to $W$. Since $2x^2\le x$, we deduce that $W$ is also $x$-sparse to $U$; and so the third outcome of the lemma holds, a contradiction.
			This proves \cref{claim:rodl}.
		\end{subproof}

		By \cref{claim:rodl}, $A$ admits a partition $(P,Q)$ such that $P$ is $\delta$-sparse to $B$ and $Q$ is $(1-\delta)$-dense to $B$.
		The following claim removes a small set of vertices from $B$ such that $P$ is very sparse to the rest.
		
		\begin{clm}
			{\ref*{lem:split}.3}
			\label{claim:sparsify}
			There exists $X\subset B$ such that $\abs X\le \abs B/4$ and $P$ is $2x^2$-sparse to $B\setminus X$.
		\end{clm}
		\begin{subproof}
			We may assume $P\ne\emptyset$.
			The proof incorporates \cref{lem:comb} and also the greedy covering argument in its proof given in~\cite{csss2023a}.
			Let $s:=\abs P$.
			We choose an ordering $(v_1,v_2,\ldots,v_s)$ of $P$ as follows. Set $X_0:=\emptyset$. For $i=1,2,\ldots,s$ in turn, assuming that $v_1,\ldots,v_{i-1}$ have been chosen, let $X_{i-1}$ be the set of vertices in $B$ with a neighbour in $\{v_1,\ldots,v_{i-1}\}$, and choose $v_i$ with the largest number of neighbours in $B\setminus X_{i-1}$ among those in $P\setminus\{v_1,\ldots,v_{i-1}\}$. Then set $X_i$ to be the set of vertices in $B$ with a neighbour in $\{v_1,\ldots,v_i\}$.
			If $\abs{X_1}< x^2\abs B$ then we are done by taking $X$ empty; and so we may assume that $\abs{X_1}\ge x^2\abs B$.
			Thus, there exists $p\in [s]$ maximal such that $\abs{X_p\setminus X_{p-1}}\ge x^2\abs B$, where $X_0=\emptyset$.
			Hence $p\le \min(s,x^{-2})$.
			Let $A':=\{v_1,\ldots,v_p\}$ and $X:=X_p$; then every vertex in $X$ has a neighbour in $A'$.
			If $\abs X\le\abs B/4$ then we are done; so suppose for a contradiction that $\abs X>\abs B/4$.
			Now, by \cref{lem:comb} with $(A,B,\Delta)=(A',X,\delta\abs B)$, either:
			\begin{itemize}
				\item $\abs{X}\le 20\sqrt{\delta\abs B\cdot \abs{X}}$; or
				
				\item for some integer $k\ge1$, there is a $(k,\abs{X}/k^2)$-comb $((a_i,B_i):i\in[k])$ in $G$, where $a_i\in A'$ and $B_i\subset X$ for all $i\in[k]$.
			\end{itemize}
			
			The first bullet is equivalent to $\abs X\le 400\delta\abs B$, which cannot happen due to our supposition and $400\delta<1/4$ by the choice of $\delta$.
			Hence, the second bullet holds.
			Then $k\le \abs{A'}=p\le x^{-2}$, and 
			$\delta\abs B>\abs X/k^2>\abs B/(4k^2)$
			which yields $k>(4\delta)^{-1/2}=2^9$ by the choice of $\delta$.
			Then for all $i\in[k]$, the choice of $d$ gives
			$$\abs{B_i}\ge\abs X/k^2>\abs B/(4k^2)\ge  \abs B/k^3>x^{6-d}\ge x^{-4b(b+1)}.$$
			Thus, \cref{claim:bookkeep} implies that either:
			\begin{itemize}
				\item there exist distinct $i,j\in[k]$ such that $B_i\to_r B_j$;
				
				\item for some $i\in[k]$ and some integer $\ell\in[k,x^{-4b-2}]$, there is an $x$-pure $(\ell,\abs{B_i}/\ell^{2b})$-blockade in $G[B_i]$; or
				
				\item for each $i\in[k]$, there exists $D_i\subset B_i$ with $\abs{D_i}\ge \abs{B_i}/(4k^b)$, such that $(D_1,\ldots,D_k)$ is $x$-pure.
			\end{itemize}
			
			If the first bullet holds then appending $a_i$ gives an induced $P_{r+1}$ in $G$ with one endpoint $a_i\in A$ and the other vertices in $B$ (see the left-hand side of \cref{fig:apex-extension}), a contradiction.
			
			If the second bullet holds then the third outcome of the lemma holds because
			\[\abs{B_i}/\ell^{2b}\ge \abs B/(4k^2\ell^{2b})\ge \abs B/\ell^{2b+3}\ge\abs B/\ceil{\ell^{1/(8b)}}^d\]
			by the choice of $d$, a contradiction.
			
			If the third bullet holds then the third outcome of the lemma also holds because
			\[\abs{B_i}/(4k^b)\ge 2^{-4}\abs B/k^{b+2}\ge \abs B/k^{b+3}\ge \abs B/\ceil{k^{1/3}}^d\]
			which also gives a contradiction. This proves \cref{claim:sparsify}.
		\end{subproof}
		
		\begin{figure}[htbp]
			\centering
			\begin{tikzpicture}[
				panel/.style={draw=black!28,rounded corners=3pt,fill=black!1},
				pathbox/.style={draw=blue!60!black,rounded corners=2pt,fill=blue!6,
					minimum height=.7cm,font=\small},
				apex/.style={circle,draw=blue!65!black,fill=blue!17,
					minimum size=.75cm,font=\small},
				path/.style={draw=black!75,line width=.85pt}
				]
				\draw[panel] (0,0) rectangle (6.9,3);
				\node[font=\small\bfseries] at (3.45,2.57) {Comb: tooth $B_i$};
				\node[apex] (ai) at (.85,1.5) {$a_i$};
				\node[pathbox,minimum width=1.65cm] (ui) at (2.85,1.5) {$u_1\in B_i$};
				\node[pathbox,minimum width=2.35cm] (taili) at (5.25,1.5) {$u_2,\ldots,u_r\in B_j$};
				\draw[path] (ai) -- (ui) -- (taili);
				\node[font=\footnotesize] at (3.45,.45) {$a_i\in P$ sees $B_i$ and misses $B_j$};
				\begin{scope}[xshift=7.3cm]
					\draw[panel] (0,0) rectangle (6.9,3);
					\node[font=\small\bfseries] at (3.45,2.57) {Anticomb: tooth $B_j$};
					\node[apex] (aj) at (.85,1.5) {$a_j$};
					\node[pathbox,minimum width=1.65cm] (uj) at (2.85,1.5) {$u_1\in B_i$};
					\node[pathbox,minimum width=2.35cm] (tailj) at (5.25,1.5) {$u_2,\ldots,u_r\in B_j$};
					\draw[path] (aj) -- (uj) -- (tailj);
					\node[font=\footnotesize] at (3.45,.45) {$a_j\in Q$ sees $B_i$ and misses $B_j$};
				\end{scope}
			\end{tikzpicture}
			\caption{Extending a directed split $B_i\to_r B_j$. In the comb case (\cref{claim:sparsify}) use $a_i$, whose tooth is $B_i$; in the anticomb case (\cref{claim:densify}) use $a_j$, whose tooth is $B_j$. The last box represents the rest of the path in $B_j$. In each case, the apex is adjacent to $u_1$ and anticomplete to $B_j$, giving an induced $P_{r+1}$.}
			\label{fig:apex-extension}
		\end{figure}
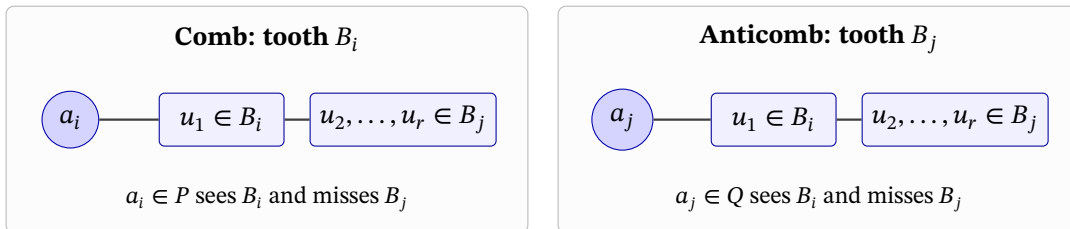
		
		Similarly, we will remove a small portion from $B$ so that $Q$ is very dense to the rest, as follows.
		
		\begin{clm}
			{\ref*{lem:split}.4}
			\label{claim:densify}
			There exists $Y\subset B$ such that $\abs Y\le \abs B/4$ and $Q$ is $(1-2x^2)$-dense to $B\setminus Y$.
		\end{clm}
		\begin{subproof}
			The proof is similar to that of \cref{claim:sparsify}.
			Here, by assuming $Q\ne\emptyset$ and supposing for a contradiction that $\abs Y>\abs B/4$, we obtain $A''\subset Q$ with $\abs{A''}\le x^{-2}$, such that for some integer $k\ge1$, there is a $(k,\abs{Y}/k^2)$-anticomb $((a_i,B_i):i\in[k])$ in $G$, where $a_i\in A''$ and $B_i\subset Y$ for all $i\in[k]$.
			Then $k\le \abs{A''}\le x^{-2}$, and 
			$\delta\abs B>\abs Y/k^2>\abs B/(4k^2)$
			which yields $k>(4\delta)^{-1/2}=2^9$ by the choice of $\delta$.
			Then for all $i\in[k]$, the choice of $d$ gives
			$$\abs{B_i}\ge\abs Y/k^2>\abs B/(4k^2)\ge  \abs B/k^3>x^{6-d}\ge x^{-4b(b+1)}.$$
			Thus, \cref{claim:bookkeep} implies that either:
			\begin{itemize}
				\item there exist distinct $i,j\in[k]$ such that $B_i\to_r B_j$;
				
				\item for some $i\in[k]$ and some integer $\ell\in[k,x^{-4b-2}]$, there is an $x$-pure $(\ell,\abs{B_i}/\ell^{2b})$-blockade in $G[B_i]$; or
				
				\item for each $i\in[k]$, there exists $D_i\subset B_i$ with $\abs{D_i}\ge \abs{B_i}/(4k^b)$, such that $(D_1,\ldots,D_k)$ is $x$-pure.
			\end{itemize}
			
			If the first bullet holds then appending $a_j$ gives an induced $P_{r+1}$ in $G$ with one endpoint $a_j\in A$ and the other vertices in $B$ (see the right-hand side of \cref{fig:apex-extension}), a contradiction. For the last two bullets, we repeat the argument at the end of the proof of \cref{claim:sparsify} to reach a contradiction. This proves \cref{claim:densify}.
		\end{subproof}
		
		Now, by \cref{claim:sparsify,claim:densify}, we have that $\abs{B\setminus (X\cup Y)}\ge\abs B/2$; and so $A=P\cup Q$ is $4x^2$-pure to $B\setminus(X\cup Y)$. Since $4x^2\le x$, this verifies the second outcome of the lemma for $r+1$.
		This completes the induction step and proves \cref{lem:split}.
	\end{proof}
	
	\subsection{Extensions to \texorpdfstring{$S_{1,2,t}$}{S1,2,t} and \texorpdfstring{$S_{1,2,t}'$}{S1,2,t-prime}}
	In this subsection we briefly explain necessary changes to the proof of \cref{lem:split} so that \cref{lem:pathpure} holds for $S_{1,2,t}$ and $S_{1,2,t}'$, for all $t\ge1$.
	To do so, we employ the following two analogues of \cref{lem:splitp4}.
	
	\begin{lemma}
		\label{lem:splite}
		Let $x,y\in(0,2^{-8}]$ with $x\le y$, and let $G$ be a graph with disjoint $A,B\subset V(G)$ satisfying $\abs B>x^{-4}$. Then either:
		\begin{itemize}
			\item $G$ contains an induced path $v_1\text-v\text-v_2\text-v_3$ such that $v\in A$ and $v_1,v_2,v_3\in B$;
			
			\item there exists $E\subset B$ with $\abs E\ge y^2\abs B$ such that every vertex in $A$ is either complete or $x$-sparse to $E$; or
			
			\item $G[B]$ contains a pure $(k,\abs B/k^8)$-blockade, for some integer $k\in[y^{-1/2},x^{-1}]$.
		\end{itemize}
	\end{lemma}
	\begin{proof}
		[Proof sketch]
		Repeat the proof of \cref{lem:splitp4} by applying \cref{lem:p3} to the {\em complement} $\overline G[B]$ of $G[B]$, then considering a vertex in $A$ with at least $x\abs E$ neighbours and at least one non-neighbour in $E$.
	\end{proof}
	\begin{lemma}
		\label{lem:splitbird}
		Let $x,y\in(0,2^{-8}]$ with $x\le y$, and let $G$ be a graph with disjoint $A,B\subset V(G)$ satisfying $\abs B>x^{-4}$. Then either:
		\begin{itemize}
			\item there are $v\in A$ and $v_1,v_2,v_3\in B$ such that $\{v,v_1,v_2\}$ is a triangle in $G$, and $v_3$ is adjacent to $v_2$ and nonadjacent to $v,v_1$;
			
			\item there exists $E\subset B$ with $\abs E\ge y^2\abs B$ such that every vertex in $A$ is either complete or $x$-sparse to $E$; or
			
			\item $G[B]$ contains a pure $(k,\abs B/k^8)$-blockade, for some integer $k\in[y^{-1/2},x^{-1}]$.
		\end{itemize}
	\end{lemma}
	\begin{proof}
		[Proof sketch]
		Repeat the proof of \cref{lem:splitp4} by applying \cref{lem:p3} to $G[B]$, then considering a vertex in $A$ with at least $x\abs E$ neighbours and at least one non-neighbour in $E$.
	\end{proof}
	Recall the definition in \cref{sec:intro} that, for an integer $r\ge1$, the graph $S_{1,2,r}$ has vertex set $\{r_0,q_1\}\cup\{p_1,p_2\}\cup\{z_1,\ldots,z_r\}$,
	where $r_0\text-q_1$, $r_0\text-p_1\text-p_2$, and $r_0\text-z_1\text-\cdots\text-z_r$ are the three branches of $S_{1,2,r}$; and $S_{1,2,r}'$ is obtained from $S_{1,2,r}$ by adding the edge $p_1q_1$.
	Let the {\em principal leaf} of both $S_{1,2,r}$ and $S_{1,2,r}'$ be $z_r$ in this case.
	Given what we have discussed so far, the following is a suitable analogue of \cref{lem:split} for $S_{1,2,r}$ and $S_{1,2,r}'$.
	
	\begin{lemma}
		\label{lem:splitspider}
		Let $r\ge1$ be an integer, and let $H$ be either $S_{1,2,r}$ or $S_{1,2,r}'$, with principal leaf $z_r$.
		Then there exists $d\ge2$ such that for every $x\in(0,2^{-d}]$ and every graph $G$ with disjoint $A,B\subset V(G)$ satisfying $\abs B>x^{-d}$, either:
		\begin{itemize}
			\item there is a copy $\varphi$ of $H$ in $G$ where $\varphi(z_r)\in A$ and $\varphi(V(H)\setminus\{z_r\})\subset B$;
			
			\item there exists $E\subset B$ with $\abs E\ge \abs B/2^d$ such that $A$ is $x$-pure to $E$ in $G$; or
			
			\item $G[B]$ contains an $x$-pure $(k,\abs B/k^d)$-blockade, for some integer $k\in[2,x^{-1}]$.
		\end{itemize} 
	\end{lemma}
	\begin{proof}
		[Proof sketch]
		Proceed by induction on $r\ge 1$.
		To prove the base case $r=1$, repeat the argument in the proof of \cref{lem:split} and incorporate \cref{lem:splite} (if $H=S_{1,2,1}$) or \cref{lem:splitbird} (if $H=S_{1,2,1}'$), to grow a pendant edge (and the principal leaf) out of a vertex of degree two.
		To finish the induction step, repeat the proof of \cref{lem:split} again to grow the principal leaf by one edge.
	\end{proof}
	
	It is now not hard to use \cref{lem:splitspider} to prove that $S_{1,2,t}$ and $S_{1,2,t}'$ are pure for all $t\ge1$, in a manner similar to the deduction of \cref{lem:pathpure} from \cref{lem:split}. We omit the detailed proof.
	
	\begin{lemma}
		\label{lem:spiderpure}
		For every integer $t\ge1$, $S_{1,2,t}$ and $S_{1,2,t}'$ are pure.
	\end{lemma}
	
	\section{From niceness to full \erh{}}
	\label{sec:nicetoeh}
	This section presents a proof of \cref{lem:niceeh}, by building on the argument in~\cite[Section 7]{density7}.
	Thus, our setting will be a $(1-y)$-dense $P_{t+1}$-free graph $G$ where $y$ is some sufficiently small density parameter.
	We aim to use the niceness of $P_{t+1}$ to generate a $(1-y^b)$-semidense blockade $\mac B$ of width at least $\poly(y)\abs G$, for some $b\ge2$.
	We can arrange each block of $\mac B$ to induce a connected subgraph of $G$, via \cref{lem:conn}.
	Also, we may assume that $G$ has a vertex $v$ outside of $\mac B$ that is mixed on many blocks of $\mac B$; let us say these blocks form a sub-blockade $\mac B'$ of $\mac B$, and $L$ is the pattern graph of $\mac B'$ in $G$.
	At this stage, the structural argument extracting $P_5$ in~\cite[Section 7]{density7} no longer appears to work, so we will be using the \erh{} property of $P_t$ instead.
	Our goal is to examine the pattern graph $L$ in depth. 
	Note that:
	\begin{itemize}
		\item the semi-density parameter $y^b$ is not small enough (in other words, still a nontrivial amount of `noise'); and
		
		\item because $G$ is $(1-y)$-dense, we cannot hope to make the proportion of non-neighbours of $v$ in each block of $\mac B'$ exceed any fixed small constant threshold.
	\end{itemize}
	
	Because of these reasons, it might not be feasible to ask the pattern graph $L$ to be $P_t$-free so that we could apply the \erh{} property of $P_t$ immediately.
	To get around this issue, we do the following:
	\begin{itemize}
		\item firstly, apply the well-known {\em tensor power/amplification} trick to form a large blow-up $J$ of $P_t$, so that $J$ still has the \erh{} property; and
		
		\item secondly, use the structure of $J$ to show that, under the $P_{t+1}$-free condition of $G$, if $L$ contains an induced copy of $J$, then we can make the dense pairs in $\mac B'$ corresponding to the edges of $J$ complete, and then grow a copy of $P_{t+1}$ going from $v$ to $\mac B'$ via these complete pairs. 
	\end{itemize}
	
	These are the two main ideas behind the proof of \cref{lem:niceeh}; and the actual proof uses a slightly modified version of the second idea.
	To formalise the first idea, for two graphs $G$ and $H$, we let $G[H]$ be the graph obtained by substituting a copy of $H$ for every vertex of $G$.
	\subsection{Extending a path by switching endpoints}
	\label{subsec:switch}
	This subsection describes how to combine the two aforesaid ideas.
	To clarify our rooted $P_{t+1}$ construction, it may be helpful to consider the following relaxation of the semidense property.
	In what follows, for a graph $G$ with disjoint $A,B\subset V(G)$, we say that $A$ {\em covers} $B$ in $G$ if every vertex in $A$ has a neighbour in $B$. 
	A blockade $\mac B$ in $G$ is {\em organised} if for every two distinct blocks $B$ and $B'$ of $\mac B$, either:
	\begin{itemize}
		\item $B$ and $B'$ are anticomplete to each other in $G$; or
		
		\item $B$ and $B'$ cover each other in $G$.
	\end{itemize}
	
	Note that the organised property is substantially weaker than the semidense property.
	If $\mac B$ is organised, let the {\em pattern graph} of $\mac B$ in $G$ be the graph with vertices the blocks of $\mac B$, where two blocks $B,B'$ are adjacent if and only if they cover each other in $G$.
	For each adjacent pair of blocks with suitable properties in the pattern graph of $\mac B$, we aim to turn their covering property into completeness.
	The following simple lemma is one step in this direction.
	
	\begin{lemma}
		\label{lem:complend}
		Let $t\ge2$ be an integer, and let $G$ be a $P_{t+1}$-free graph with an organised blockade $(B_1,\ldots,B_t)$ with pattern $P_t$ in $G$, such that $G[B_1]$ is connected.
		Then $B_2$ is complete to $B_1$ in $G$.
	\end{lemma}
	\begin{proof}
		Suppose not; then there exists $v_2\in B_2$ mixed on $B_1$. Since $v_2$ has a neighbour in $B_1$ and $G[B_1]$ is connected, there is an edge $v_0v_1$ in $G[B_1]$ such that $v_2$ is adjacent to $v_1$ and nonadjacent to $v_0$.
		Because $(B_1,\ldots,B_t)$ is organised with pattern $P_t$ in $G$, there is an induced path $v_2\text-v_3\text-\cdots\text-v_t$ in $G$ starting at this chosen vertex $v_2$, such that $v_r\in B_r$ for all $r\in\{2,3,\ldots,t\}$ (see \cref{fig:complete-end}).
		Then $v_0\text-v_1\text-v_2\text-\cdots\text-v_t$ is an induced $P_{t+1}$ in $G$, a contradiction. This proves \cref{lem:complend}.
	\end{proof}
	
	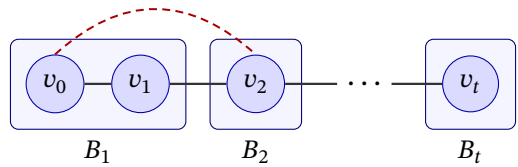
\begin{figure}[htbp]
		\centering
		\begin{tikzpicture}[
			vertex/.style={circle,draw=blue!65!black,fill=blue!14,minimum size=.65cm,font=\small},
			block/.style={draw=blue!45!black,rounded corners=3pt,fill=blue!4},
			edge/.style={draw=black!80,line width=.85pt},
			nonedge/.style={draw=red!65!black,densely dashed,line width=.75pt}
			]
			\draw[block] (0,.12) rectangle (2.32,1.31);
			\draw[block] (2.64,.12) rectangle (3.85,1.31);
			\draw[block] (5.49,.12) rectangle (6.68,1.31);
			\node[vertex] (vzero) at (.59,.72) {$v_0$};
			\node[vertex] (vone) at (1.72,.72) {$v_1$};
			\node[vertex] (vtwo) at (3.245,.72) {$v_2$};
			\node[font=\large] (omitted) at (4.68,.72) {$\ldots$};
			\node[vertex] (vlast) at (6.085,.72) {$v_t$};
			\draw[edge] (vzero) -- (vone) -- (vtwo) -- (omitted) -- (vlast);
			\draw[nonedge] (vzero.north) .. controls (1.475,1.91) and (2.36,1.91) .. (vtwo.north);
			\node[font=\small] at (1.16,-.18) {$B_1$};
			\node[font=\small] at (3.245,-.18) {$B_2$};
			\node[font=\small] at (6.085,-.18) {$B_t$};
		\end{tikzpicture}
		\caption{Why the first covering pair must be complete. If $v_2\in B_2$ is mixed on connected $G[B_1]$, an edge $v_0v_1$ of $G[B_1]$ crosses its non-neighbours and neighbours. Mutual covering along the remaining blocks extends the induced path to $v_t$, contradicting $P_{t+1}$-freeness.}
		\label{fig:complete-end}
	\end{figure}
	
	Now, let $H_1,H_2$ be the one-vertex graph; and define the sequence $(H_r)_{r\ge1}$ of repeated blow-ups 
	\begin{equation}
		\label{eq:sub1}
		H_{r+1}:=P_t[H_r]\quad\text{for every integer $r\ge2$.}
	\end{equation}
	
	The substitution in \eqref{eq:sub1} is shown in \cref{fig:iterated-blowup} below.
	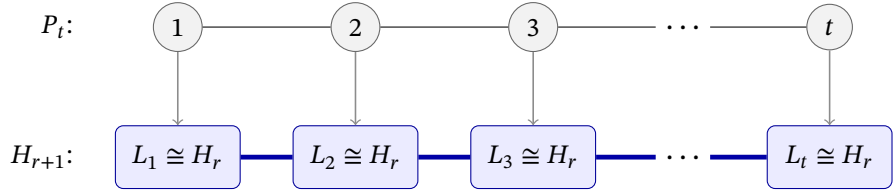
\begin{figure}[htbp]
		\centering
		\begin{tikzpicture}[
			quotient/.style={circle,draw=black!60,fill=black!5,minimum size=.57cm,font=\small},
			cluster/.style={draw=blue!60!black,rounded corners=3pt,fill=blue!8,
				minimum width=1.65cm,minimum height=.83cm,align=center,font=\small},
			clusterlink/.style={draw=blue!65!black,line width=1.8pt},
			projection/.style={->,draw=black!45,line width=.6pt}
			]
			\node[font=\small,anchor=east] at (-.25,2.11) {$P_t$:};
			\node[quotient] (qone) at (1,2.11) {$1$};
			\node[quotient] (qtwo) at (3.35,2.11) {$2$};
			\node[quotient] (qthree) at (5.7,2.11) {$3$};
			\node[font=\large] (qdots) at (7.68,2.11) {$\ldots$};
			\node[quotient] (qlast) at (9.62,2.11) {$t$};
			\draw (qone) -- (qtwo) -- (qthree) -- (qdots) -- (qlast);
			\node[font=\small,anchor=east] at (-.25,.39) {$H_{r+1}$:};
			\node[cluster] (lone) at (1,.39) {$L_1\cong H_r$};
			\node[cluster] (ltwo) at (3.35,.39) {$L_2\cong H_r$};
			\node[cluster] (lthree) at (5.7,.39) {$L_3\cong H_r$};
			\node[font=\large] (ldots) at (7.68,.39) {$\ldots$};
			\node[cluster] (llast) at (9.62,.39) {$L_t\cong H_r$};
			\draw[clusterlink] (lone) -- (ltwo);
			\draw[clusterlink] (ltwo) -- (lthree);
			\draw[clusterlink] (lthree) -- (ldots) -- (llast);
			\foreach \top/\bottom in {qone/lone,qtwo/ltwo,qthree/lthree,qlast/llast} {
				\draw[projection] (\top.south) -- (\bottom.north);
			}
		\end{tikzpicture}
		\caption{The repeated blow-up $H_{r+1}=P_t[H_r]$. Each vertex $j$ of $P_t$ is substituted by a copy $L_j$ of $H_r$. Consecutive copies are joined in the pattern graph, while non-consecutive copies are anticomplete. For an organised blockade with this pattern, its blocks initially only cover one another across a join.}
		\label{fig:iterated-blowup}
	\end{figure}
	
	Now, we are given a vertex $v$ outside of $\mac B$ that is mixed on all blocks of $\mac B$. If $G$ is $P_{t+1}$-free and $\mac B$ exhibits an $H_t$ pattern, then we can use the structure of $H_t$ to construct an induced $P_{t+1}$ with endpoint $v$, a contradiction.
	To do so, we extend the path one vertex at a time, using the following simple observation.
	Assume that $\mac B=(B_1,\ldots,B_t)$ is the blockade given in \cref{lem:complend}, and that $G$ contains an induced path $v_2\text-\cdots\text-v_{t-1}\text-v_t$, such that $v_i\in B_i$ for all $i\in\{2,\ldots,t\}$, and $v$ is anticomplete to $\{v_2,\ldots,v_{t-1}\}$.
	If $v$ is adjacent to $v_t$, choose a non-neighbour $v_1$ of $v$ in $B_1$; and if $v$ is nonadjacent to $v_t$, then choose a neighbour $v_1$ of $v$ in $B_1$.
	Since $B_2$ is complete to $B_1$ by \cref{lem:complend}, we have that $v_1,v_2$ are always adjacent; and so one of
	\[v\text-v_1\text-v_2\text-\cdots\text-v_{t-1}\text-v_t\quad\text{if $v$ and $v_t$ are nonadjacent}\]
	or
	\[v\text-v_t\text-v_{t-1}\text-\cdots\text-v_{2}\text-v_1\quad\text{if $v$ and $v_t$ are adjacent}\]
	is a copy of $P_{t+1}$ rooted at $v$, whose other vertices are in $\mac B$.
	The two choices of endpoint are compared in \cref{fig:switch-ends}.
	\begin{figure}[htbp]
		\centering
		\begin{tikzpicture}[
			panel/.style={draw=black!30,rounded corners=3pt,fill=black!1},
			vertex/.style={circle,draw=blue!65!black,fill=blue!12,minimum size=.58cm,font=\small},
			outside/.style={vertex,draw=orange!70!black,fill=orange!19},
			edge/.style={draw=black!80,line width=.83pt},
			nonedge/.style={draw=red!65!black,densely dashed,line width=.7pt}
			]
			\draw[panel] (0,0) rectangle (5.95,2.75);
			\node[font=\small\bfseries] at (2.98,2.43) {$v$ misses $v_t$};
			\node[outside] (vleft) at (.55,1.13) {$v$};
			\node[vertex] (voneleft) at (1.68,1.13) {$v_1$};
			\node[vertex] (vtwoleft) at (2.8,1.13) {$v_2$};
			\node[font=\large] (dotsleft) at (3.91,1.13) {$\ldots$};
			\node[vertex] (vtleft) at (5.08,1.13) {$v_t$};
			\draw[edge] (vleft) -- (voneleft) -- (vtwoleft) -- (dotsleft) -- (vtleft);
			\draw[nonedge] (vleft.north) .. controls (2.06,2.08) and (3.59,2.08) .. (vtleft.north);
			\node[font=\footnotesize] at (1.68,.35) {$v_1\in B_1$};
			\node[font=\footnotesize] at (5.08,.35) {$v_t\in B_t$};
			\begin{scope}[xshift=6.33cm]
				\draw[panel] (0,0) rectangle (5.95,2.75);
				\node[font=\small\bfseries] at (2.98,2.43) {$v$ sees $v_t$};
				\node[outside] (vright) at (.55,1.13) {$v$};
				\node[vertex] (vtright) at (1.68,1.13) {$v_t$};
				\node[font=\large] (dotsright) at (2.8,1.13) {$\ldots$};
				\node[vertex] (vtworight) at (3.91,1.13) {$v_2$};
				\node[vertex] (voneright) at (5.08,1.13) {$v_1$};
				\draw[edge] (vright) -- (vtright) -- (dotsright) -- (vtworight) -- (voneright);
				\draw[nonedge] (vright.north) .. controls (2.06,2.08) and (3.59,2.08) .. (voneright.north);
				\node[font=\footnotesize] at (1.68,.35) {$v_t\in B_t$};
				\node[font=\footnotesize] at (5.08,.35) {$v_1\in B_1$};
			\end{scope}
		\end{tikzpicture}
		\caption{The endpoint switch. If $v$ misses $v_t$, choose $v_1\in B_1$ adjacent to $v$ and use the path from the left. If $v$ sees $v_t$, choose $v_1\in B_1$ nonadjacent to $v$ and reverse the path. The dashed arc is the relevant nonedge; $v$ also misses all interior vertices $v_2,\ldots,v_{t-1}$, and $B_1$ is complete to $B_2$.}
		\label{fig:switch-ends}
	\end{figure}
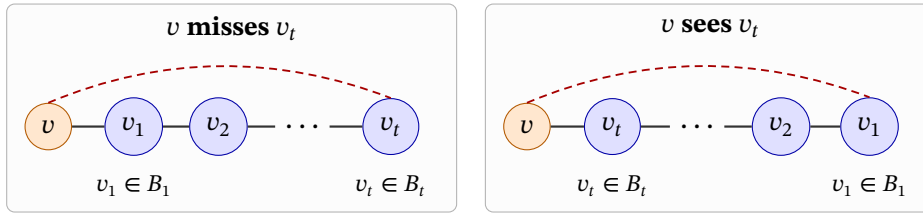
	
	This is the key idea behind the following lemma.
	\begin{lemma}
		\label{lem:switching}
		Let $r,t$ be integers with $t\ge r\ge 2$.
		Let $G$ be a $P_{t+1}$-free graph, with $v\in V(G)$ and an organised blockade $\mac B$ with pattern $H_r$ in $G\setminus v$, such that each block of $\mac B$ induces a connected subgraph in $G$ and $v$ is mixed on every block of $\mac B$.
		Then $G$ contains an induced $P_{r+1}$ where $v$ is an endpoint and the other vertices are in $\bigcup_{B\in\mac B}B$.
	\end{lemma}
	\begin{proof}
		We proceed by induction on $r$. If $r=2$ then $\mac B=\{B\}$ for some $B\subset V(G)$; and the lemma holds since $v$ is mixed on $B$ and $G[B]$ is connected.
		
		Now, for $r<t$, we prove the lemma for $r+1$ assuming that it is true for $r$.
		To see this, let $\mac B=(B_z:z\in V(H_{r+1}))$. For every $j\in[t]$, let $L_j$ be the vertex set of the copy of $H_r$ in $H_{r+1}$ that substitutes the $j$-th vertex of $P_{t}$, and let $C_j:=\bigcup_{z\in L_j}B_z$.
		Then $G[C_j]$ is connected (since $H_r$ and each $G[B_z]$ are connected), and $(C_1,\ldots,C_t)$ is an organised blockade with pattern $P_t$ in $G$.
		In particular, \cref{lem:complend} implies that $C_1$ is complete to $C_2$ in $G$.
		We claim that:
		
		\begin{claim}
			\label{claim:complete}
			For every $j\in[r-1]$, every $u\in C_{j+1}$ is complete to some block in $(B_z:z\in L_j)$.
		\end{claim}
		\begin{subproof}
			Suppose not. Since $\mac B$ is organised, $u$ has a neighbour in $B_z$ for each $z\in L_j$; and so $u$ is mixed on $B_z$ for all such $z$.
			By induction applied to $u$ versus $(B_z:z\in L_j)$, $G$ contains an induced path $P$ of length $r$ where $u$ is an endpoint and the other vertices are in $C_j$.
			Since $\mac B$ is organised, there is an induced path $Q$ in $G$ such that $u\in Q$ and $\abs{Q\cap C_i}=1$ for all $i\in\{j+1,\ldots,t\}$.
			Then $P\text-u\text-Q$ is an induced path in $G$ of length $r+(t-j-1)\ge t$ (since $j\le r-1$), contrary to the $P_{t+1}$-freeness of $G$.
			This proves \cref{claim:complete}.
		\end{subproof}
		
		Now, let $v_r\in C_r\setminus N_G(v)$. Then there exists $j\in\{2,\ldots,r\}$ minimal such that $G$ contains an induced path $v_r\text-v_{r-1}\text-\cdots\text-v_j$ where $v_i\in C_i\setminus N_G(v)$ for all $i\in\{j,j+1,\ldots,r\}$.
		If $j>2$, then by \cref{claim:complete}, $v_j\in C_j$ is complete in $G$ to $B_z$ for some $z\in L_{j-1}$; and so there exists $v_{j-1}\in B_z\setminus N_G(v)$, violating the minimality of $j$. Therefore $j=2$.
		
		Next, since $\mac B$ is organised, there exists $v_{r+1}\in C_{r+1}\cap N_G(v_r)$.
		Let $v_1$ be a neighbour of $v$ in $C_1$ if $v_{r+1}\in C_{r+1}\setminus N_G(v)$, and let $v_1$ be a non-neighbour of $v$ in $C_1$ if $v_{r+1}\in C_{r+1}\cap N_G(v)$.
		Then $v_1$ is always adjacent to $v_2$, since $C_2$ is complete to $C_1$.
		Thus, \[v\text-v_1\text-v_2\text-\cdots\text-v_{r}\text-v_{r+1}\quad\text{(if $v_1\in C_1\cap N_G(v)$)}
		\qquad\text{or}\qquad v\text-v_{r+1}\text-v_{r}\text-\cdots\text-v_2\text-v_1\quad \text{(if $v_1\in C_1\setminus N_G(v)$)}\]
		is an induced $P_{r+2}$ in $G$ where $v$ is an endpoint and the other vertices are in $\bigcup_{B\in\mac B}B$, completing the induction step.
		This proves \cref{lem:switching}.
	\end{proof}
	
	Now, \cref{lem:switching} with $r=t$ implies that:
	
	\begin{lemma}
		\label{lem:switch}
		Let $t$ be an integer with $t\ge2$. Let $G$ be a $P_{t+1}$-free graph with $v\in V(G)$.
		Then $G\setminus v$ does not contain any organised blockade $\mac B$ with pattern $H_t$, such that $v$ is mixed on every block of $\mac B$ and $G[B]$ is connected for all $B\in \mac B$.
	\end{lemma}
	
	\subsection{Densification}
	\label{subsec:densify}
	Given \cref{lem:switch}, we are now ready to prove \cref{lem:niceeh} by means of iterative densification.
	The proof is more or less routine given~\cite{density6} and~\cite[Section 7]{density7};
	and to make our argument less heavy, we also adapt the dyadic partitioning step shown in~\cite[Section 9]{2025pc5}.
	
	\begin{lemma}
		\label{lem:pathdense1}
		Let $t\ge2$ be an integer. Assume that $H_t$ has the \erh{} property witnessed by $b\ge2$, and that $P_{t+1}$ is nice as witnessed by $a\ge2$. Then for $d=11ab$, the following holds.
		Let $y\in(0,2^{-d}]$, and let $G$ be a $(1-y)$-dense $P_{t+1}$-free graph with $\abs G>y^{-d}$. Then $G$ contains either:
		\begin{itemize}
			\item a $y^3$-restricted induced subgraph with at least $y^{d}\abs G$ vertices;
			
			\item an anticomplete $(y^{-1},y^{d}\abs G)$-blockade; or
			
			\item a complete pair $(X,Y)$ with $\abs X\ge y^{d}\abs G$ and $\abs Y\ge (1-4y)\abs G$.
		\end{itemize}
	\end{lemma}
	\begin{proof}
		Assume that the first two outcomes fail. 
		Let $\eps:=y^{5b}$ and $\ell:=\ceil{\eps^{-1}}$; then $y^3\eps^{2a}/4\ge y^{4+10ab}\ge y^d$.
		Let $S\subset V(G)$ with $\abs S=\ceil{y\abs G}\le 2y\abs G$.
		
		\begin{claim}
			\label{claim:getsemiconn}
			$G[S]$ contains a $(1-\eps)$-semidense $(\ell,y^3\eps^{2a}\abs G)$-blockade $(B_1,\ldots,B_\ell)$, such that $G[B_i]$ is connected for all $i\in[\ell]$.
		\end{claim}
		\begin{subproof}
			Since $\abs S\ge y\abs G> y^{1-d}\ge \eps^{-2a}$, the niceness of $P_{t+1}$ (with $\eps^2$ replacing $\eps$) gives a $(1-\eps^2)$-semidense $(\eps^{-2},\eps^{2a}\abs S)$-blockade in $G[S]$. We retain $\ell=\ceil{\eps^{-1}}$ blocks $A_1,\ldots,A_\ell$ from this blockade.
			Note that $\eps^{2a}\abs S\ge y\eps^{2a}\abs G$.
			For each $j\in[\ell]$, by \cref{lem:conn} with $k=\ceil{y^{-1}}\le 2y^{-1}$, $G[A_j]$ contains either:
			\begin{itemize}
				\item an anticomplete $(y^{-1},y^2\abs{A_j}/4)$-blockade; or
				
				\item a connected induced subgraph with at least $y\abs {A_j}/2$ vertices.
			\end{itemize}
			
			If the first bullet holds then the second outcome of the lemma holds since 
			$$y^2\abs{A_j}/4\ge y^3\eps^{2a}\abs G/4=y^{10ab+3}\abs G/4\ge y^d\abs G.$$
			Thus, the second bullet holds; and so there exists $B_j\subset A_j$ such that $\abs{B_j}\ge y\abs{A_j}/2\ge y^2\eps^{2a}\abs G/2\ge y^3\eps^{2a}\abs G$ and $G[B_j]$ is connected.
			In particular, since $(A_1,\ldots,A_\ell)$ is $(1-\eps^2)$-semidense and $\eps^2\le y\eps/2$, the blockade $(B_1,\ldots,B_\ell)$ is $(1-\eps)$-semidense.
			This proves \cref{claim:getsemiconn}.
		\end{subproof}
		
		Now, let $B:=V(G)\setminus S$; then $\abs B=\abs G-\abs S\ge (1-2y)\abs G$.
		
		\begin{claim}
			\label{claim:pathmixed}
			Every vertex in $B$ is mixed on fewer than $y^{-4b}$ blocks among $B_1,\ldots,B_{\ell}$.
		\end{claim}
		
		\begin{subproof}
			Suppose that there exists $v\in B$ mixed on at least $y^{-4b}$ blocks among $B_1,\ldots,B_\ell$;
			and assume that it is mixed on $B_1,\ldots,B_h$, where $h\ge y^{-4b}$.
			Since $(B_1,\ldots,B_h)$ is $(1-\eps)$-semidense (and so organised) and $G[B_i]$ is connected for all $i\in[h]$,
			\cref{lem:switch} implies that the pattern graph $L$ of this blockade is $H_t$-free; and so the choice of $b$ gives a clique or stable set $J$ in $L$ with $\abs J\ge \abs L^{1/b}\ge h^{1/b}\ge y^{-4}$.
			We may assume that $V(L)=[h]$.
			If $J$ is stable in $L$ then $(B_j:j\in J)$ satisfies the second outcome of the lemma since $y^3\eps^{2a}\abs G\ge y^d\abs G$, violating our supposition.
			Thus, $J$ is a clique in $L$.
			Since $y^3\eps^{2a}\abs G\le\abs{B_i}\le \abs S\le 2y\abs G$ for all $i\in J$, dyadic partitioning gives some  $I\subset J$ with
			$$\abs I\ge \frac {\abs J}{2\log(y^{-3}\eps^{-2a})}\ge\frac{y^{-4}}{4d\log (y^{-1})}\ge 16y^{-3} $$
			(note that $d=11ab\ge40$ and $y\in(0,2^{-d}]$) 
			and $\abs{B_i}\le 2\abs{B_j}$ for all $i,j\in I$.
			Let $m:=\min_{i\in I}\abs{B_i}\ge\frac12\max_{i\in I}\abs{B_i}$.
			For every $i\in I$, let $Q_i\subset B_i$ satisfy $\abs{Q_i}=m$; then $(Q_i:i\in I)$ is $(1-2\eps)$-dense.
			Let $Q:=\bigcup_{i\in I}Q_i$; then $\abs Q=\abs I m$ and the number of non-edges of $G[Q]$ is at most
			\[\abs I\cdot m^2+\abs I^2\cdot 2\eps m^2
			\le y^3(\abs Im)^2/16+2\eps(\abs Im)^2< y^3\abs Q^2/8\]
			where the last inequality holds since $y\le 2^{-d}$ and $2\eps=2y^{5b}< y^{3}/16$.
			Now, let $R$ be the set of vertices in $G[Q]$ with at least $y^{3}\abs Q/2$ non-neighbours in $Q$.
			Then $\abs R<\abs Q/2$; and so $\abs{Q\setminus R}>\abs Q/2=\abs Im/2\ge m\ge y^3\eps^{2a}\abs G\ge y^{d}\abs G$.
			Also, every vertex in $Q\setminus R$ has fewer than $y^{3}\abs Q/2< y^{3}\abs{Q\setminus R}$ non-neighbours in $G[Q\setminus R]$.
			Therefore, $G[Q\setminus R]$ is an induced subgraph of $G$ satisfying the first outcome of the lemma, contrary to our assumption.
			This proves \cref{claim:pathmixed}.
		\end{subproof}
		
		\cref{claim:pathmixed} gives some $i\in[\ell]$ such that $B$ has fewer than $y^b\abs B\le y\abs G$ vertices mixed on $B_i$.
		Since $G$ is $(1-y)$-dense, there are fewer than $y\abs G$ vertices in $B$ with no neighbours in $B_i$; and so $B$ has fewer than $2y\abs G$ vertices with at least one non-neighbour in $B_i$.
		Let $Y$ be the set of vertices in $B$ complete to $B_i$. Then $\abs Y\ge \abs B-2y\abs G$; and because $\abs B\ge (1-2y)\abs G$, we have
		$\abs Y\ge \abs B-2y\abs G\ge  (1-4y)\abs G$.
		Therefore, the third outcome of the lemma holds since $\abs{B_i}\ge y^3\eps^{2a}\abs G\ge y^{d}\abs G$.
		This proves \cref{lem:pathdense1}.
	\end{proof}
	Iterating \cref{lem:pathdense1} gives:
	
	\begin{lemma}
		\label{lem:pathdense}
		Let $t\ge2$ be an integer. Assume that $H_t$ has the \erh{} property witnessed by $b\ge2$, and that $P_{t+1}$ is nice as witnessed by $a\ge2$. Then for $d=12ab$, the following holds.
		Let $x,y\in(0,2^{-d}]$ with $x\le y$, and let $G$ be a $(1-y)$-dense $P_{t+1}$-free graph with $\abs G>x^{-2d}$. Then $G$ contains either:
		\begin{itemize}
			\item a $y^2$-restricted induced subgraph with at least $y^{d}\abs G$ vertices; or
			
			\item a complete or anticomplete $(y^{-1/2},y^{d}\abs G)$-blockade.
		\end{itemize}
	\end{lemma}
	\begin{proof}
		Suppose that all outcomes fail. Let $p:=y^{d}\abs G$. We claim that:
		
		\begin{claim}
			\label{claim:pathdense}
			For every induced subgraph $F$ of $G$ with $\abs F\ge\frac14\abs G$, there is a complete pair $(A,B)$ with $\abs A\ge p$ and $\abs B\ge (1-16y)\abs F$.
		\end{claim}
		\begin{subproof}
			Since $\abs F\ge\frac14\abs G>\frac14 x^{-2d}>x^{-d}$, we have that $F$ is $(1-4y)$-dense. Thus, \cref{lem:pathdense1} (with $4y\le 4\cdot 2^{-12ab}\le 2^{-11ab}$) implies that $F$ contains either:
			\begin{itemize}
				\item a $(4y)^3$-restricted induced subgraph with at least $(4y)^{d}\abs F$ vertices;
				
				\item an anticomplete $((4y)^{-1},(4y)^{d}\abs F)$-blockade; or
				
				\item a complete pair $(X,Y)$ with $\abs X\ge (4y)^{d}\abs F$ and $\abs Y\ge (1-16y)\abs F$.
			\end{itemize}
			
			Since both outcomes of the lemma fail, $(4y)^3\le y^2$, $(4y)^{-1}\ge y^{-1/2}$, and $(4y)^{d}\abs F\ge y^{d}\abs G=p$, we see that the first two bullets cannot hold.
			Hence the third bullet holds; and so the claim holds because $(4y)^{d}\abs F\ge p$.
			This proves \cref{claim:pathdense}.
		\end{subproof}
		
		Now, by \cref{claim:pathdense,lem:iteration} (applied to $\overline G$, with $16y$ replacing $y$ and $x=\abs G^{-1}$ to ensure completeness), $G$ contains a complete $((16y)^{-1},p)$-blockade.
		Since $(16y)^{-1}\ge y^{-1/2}$, the second outcome holds, a contradiction. This proves \cref{lem:pathdense}.
	\end{proof}
	We are now ready to prove \cref{lem:niceeh} via the following lemma.
	
	\begin{lemma}
		\label{lem:patheh}
		Let $t\ge2$ be an integer. Assume that $P_t$ has the \erh{} property, and that $P_{t+1}$ is nice as witnessed by $a\ge2$. 
		Then there exists $q\ge2$ such that for every $x\in(0,2^{-q}]$ and every $P_{t+1}$-free graph $G$, either:
		\begin{itemize}
			\item $G$ has a $(1-x)$-dense induced subgraph with at least $x^q\abs G$ vertices; or
			
			\item there is a complete or anticomplete $(k,\abs G/k^q)$-blockade in $G$, for some integer $k\ge2$.
		\end{itemize}
	\end{lemma}
	\begin{proof}
		Since $P_t$ has the \erh{} property, \cref{thm:aps} implies that $H_t$ also has the \erh{} property; let $b\ge2$ witness the \erh{} property of $H_t$.
		Let $d:=12ab$.
		Let $s\ge2$ be given by \cref{lem:treesparse} with $H=P_{t+1}$, and let $c:=2^{-\max(d,s)}$. By \cref{thm:rodl}, there exists $h\ge2$ such that every $P_{t+1}$-free graph $G$ contains a $c$-restricted induced subgraph with at least $c^h\abs G$ vertices; by increasing $h$ if necessary we may assume $h\ge d/2$.
		We claim that $q:=s(2h+d)$ suffices.
		To this end, let $x\in(0,2^{-q}]$, and let $G$ be a $P_{t+1}$-free graph. Assume that $\abs G>x^{-q}$ (else the first outcome of the lemma holds by taking any induced subgraph of $G$ with $\ceil{x^q\abs G}\in\{0,1\}$ vertices), and that $G$ contains no complete or anticomplete $(k,\abs G/k^q)$-blockade, for any integer $k\ge 2$.
		We claim that:
		
		\begin{claim}
			\label{claim:nosparse}
			For every $y\in[x^2,c]$, $G$ contains no $y$-sparse induced subgraph $F$ with $\abs F\ge y^{q/s-1}\abs G$.
		\end{claim}
		\begin{subproof}
			Suppose that such an induced subgraph $F$ exists. Since $\abs F\ge y^{q/s-1}\abs G> y^{q/s-1-q/2}\ge y^{-1}$ (since $s\ge2$) and $y\le c\le 2^{-s}$, the choice of $s$ and \cref{lem:treesparse} (with $\eps=y^{1/s}$) together yield an anticomplete $(y^{-1/s},y\abs F)$-blockade in $F$.
			For $k=\ceil{y^{-1/s}}\ge2$, this blockade is an anticomplete $(k,\abs G/k^q)$-blockade in $G$, since $y\abs F\ge y^{q/s}\abs G\ge \abs G/k^q$.
			This violates our supposition, and so proves \cref{claim:nosparse}.
		\end{subproof}
		
		Now, the choice of $h$ gives a $c$-restricted induced subgraph $G'$ of $G$ with $\abs{G'}\ge c^h\abs G$.
		Since $c^h\ge c^{q/s-1}$ by the choice of $q$,
		\cref{claim:nosparse} with $y=c$ implies that $G'$ is $(1-c)$-dense.
		We next claim that:
		
		\begin{claim}
			\label{claim:path}
			For every $y\in[x,c]$ and every $(1-y)$-dense induced subgraph $F$ of $G$ with $\abs F\ge y^{2h}\abs G$, there is a $(1-y^2)$-dense induced subgraph with at least $y^{2h}\abs F$ vertices.
		\end{claim}
		\begin{subproof}
			Observe that $\abs F\ge y^{2h}\abs G>y^{2h}x^{-q}\ge x^{2h-q}\ge x^{-2d}$ by the choice of $q$ and $h$.
			Thus, by \cref{lem:pathdense}, $F$ contains either:
			\begin{itemize}
				\item a $y^2$-restricted induced subgraph $J$ with at least $y^{d}\abs F$ vertices; or
				
				\item a complete or anticomplete $(y^{-1/2},y^{d}\abs F)$-blockade.
			\end{itemize}
			
			If the second bullet holds, then since $y^{d}\abs F\ge y^{d+2h}\abs G\ge y^{q/2}\abs G$ by the choice of $q$,
			there would be a complete or anticomplete $(k,\abs G/k^q)$-blockade in $G$ where $k=\ceil{y^{-1/2}}$, contrary to our supposition.
			Thus, the first bullet holds.
			Since $\abs J\ge y^{d}\abs F\ge y^{d+2h}\abs G\ge y^{2q/s-2}\abs G$ by the choice of $q$, \cref{claim:nosparse} implies that $J$ is $(1-y^2)$-dense.
			This proves \cref{claim:path}.
		\end{subproof}
		
		Now, \cref{claim:path,lem:sparsify} (applied to $\overline G$, with $h$ replacing $t$) together give a $(1-x)$-dense induced subgraph of $G$ with at least $x^{4h}\abs G\ge x^q\abs G$ vertices.
		This verifies the first outcome of the lemma, and so proves \cref{lem:patheh}.
	\end{proof}
	
	\begin{proof}
		[Proof of \cref{lem:niceeh}]
		This follows by combining \cref{lem:patheh,lem:eh} (with $q$ replacing $b$).
	\end{proof}
	
	\subsection{Extensions}
	We now explain several extensions of our argument in this section thus far.
	Such extensions will, in particular, enable us to prove \cref{lem:niceeh} for $S_{1,2,t}$ under the assumption that it is nice and that the graph obtained by removing its principal leaf has the \erh{} property.
	Combined with \cref{lem:spiderpure}, this would lead to the full \erh{} property of $S_{1,2,t}$.
	A similar procedure can be done for $S_{1,2,t}'$: indeed, since this graph contains a triangle and so cannot be nice, we will use \cref{lem:nearnice} instead, where the sparse outcome is still good because we can use \cref{lem:antiviral} and assume that $S_{1,2,t-1}'$ is viral.
	
	Thus, in what follows we will only focus on the setting of \cref{subsec:switch}: semidense blockades $\mac B$ in a dense graph $G$, with a vertex $v$ outside that is mixed on the blocks of $\mac B$.
	In our setup, for an integer $s\ge2$, an {\em $s$-lollipop} is a graph $H$ with an induced path
	\[R=g_0\text-g_1\text-\cdots\text-g_s\]
	such that $H\setminus (R\setminus\{g_s\})$ is connected with at least two vertices, and $V(H)\setminus R$ is anticomplete to $R\setminus\{g_s\}$.
	(This is the principal structure behind the Gy\'arf\'as path argument.)
	The following is an extension of \cref{lem:complend} for $H$ and $F:=H\setminus g_0$.
	
	\begin{lemma}
		\label{lem:completend}
		Let $s\ge2$ be an integer, and let $H$ be an $s$-lollipop. Let $G$ be an $H$-free graph with a $(1-\abs F^{-1})$-semidense blockade $\mac B=(B_z:z\in V(F))$ with pattern $F$ in $G$, where $G[B_{g_1}]$ is connected.
		If $g_2$ is the unique neighbour of $g_1$ in $F$, then $B_{g_2}$ is complete to $B_{g_1}$ in $G$.
	\end{lemma}
	\begin{proof}
		Suppose not; then there exists $u\in B_{g_2}$ mixed on $B_{g_1}$.
		Since $G[B_{g_1}]$ is connected, it contains an edge $u_1u_2$ such that $u$ is adjacent to $u_1$ and nonadjacent to $u_2$.
		Then embed $g_0,g_1,g_2$ into $u_2,u_1,u$, respectively.
		Now, since $\mac B$ is $(1-\abs F^{-1})$-semidense, we can greedily embed the rest of $H$ into $(B_z:z\in V(F)\setminus\{g_1\})$.
		This proves \cref{lem:completend}.
	\end{proof}
	
	Now, we formulate a quick extension of \cref{lem:switching} in the semidense setting.
	For an $s$-lollipop $H$, we define a sequence of repeated blow-ups $(J_r)_{r\ge1}$ in a manner similar to the definition \eqref{eq:sub1}, as follows: $J_1,J_2$ are the one-vertex graph; and (recall that $F=H\setminus g_0$)
	\begin{equation}
		\label{eq:sub2}
		J_{r+1}:=F[J_r]\quad\text{for every integer $r\ge2$.}
	\end{equation}
	
	The aforesaid extension of \cref{lem:switching} to this setting is as follows, with a similar proof.
	
	\begin{lemma}
		\label{lem:switchin}
		Let $r,s$ be integers with $2\le r\le s+1$ and $s\ge2$, and let $H$ be an $s$-lollipop.
		Let $G$ be an $H$-free graph, with $v\in V(G)$ and a $(1-\abs{J_{r+1}}^{-1})$-semidense blockade $\mac B$ with pattern $J_r$ in $G\setminus v$, such that each block of $\mac B$ induces a connected subgraph in $G$ and $v$ is mixed on every block of $\mac B$.
		Then $G$ contains an induced $P_{r+1}$ where $v$ is an endpoint and the other vertices are in $\bigcup_{B\in\mac B}B$.
	\end{lemma}
	
	\begin{proof}
		We proceed by induction on $r$. If $r=2$ then $\mac B=\{B\}$ for some $B\subset V(G)$; and the lemma holds since $v$ is mixed on $B$ and $G[B]$ is connected.
		
		Now, for $r\le s$, we prove the lemma for $r+1$ assuming that it is true for $r$.
		To see this, let $\mac B=(B_z:z\in V(J_{r+1}))$. For every $j\in[s]$, let $L_j$ be the vertex set of the copy of $J_r$ in $J_{r+1}$ that substitutes $g_j$, and let $C_j:=\bigcup_{z\in L_j}B_z$.
		Then $G[C_j]$ is connected (since $J_r$ and each $G[B_z]$ are connected), and $(C_1,\ldots,C_{s})$ is a $(1-\abs{J_{r+1}}^{-1})$-semidense blockade with pattern $P_{s}$ in $G$.
		Also, for each $g\in V(H)\setminus V(R)$, let $L_g$ be the vertex set of the copy of $J_r$ in $J_{r+1}$ that substitutes $g$, and let $C_g:=\bigcup_{z\in L_g}B_z$.
		Again, $G[C_g]$ is connected for each such $g$.
		The blockade $(C_1,\ldots,C_s)\cup(C_g:g\in V(H)\setminus V(R))$ is a $(1-\abs{J_{r+1}}^{-1})$-semidense blockade with pattern $H\setminus g_0$ in $G$; and so, since $\abs{J_{r+1}}=\abs F^{r-1}\ge\abs F$, \cref{lem:completend} implies that $C_2$ is complete to $C_1$ in $G$.
		We claim that:
		
		\begin{claim}
			\label{claim:compl}
			For every $j\in[r-1]$, every $u\in C_{j+1}$ is complete to some block in $(B_z:z\in L_j)$.
		\end{claim}
		\begin{subproof}
			Suppose not. Since $\mac B$ is $(1-\abs{J_{r+1}}^{-1})$-semidense, $u$ has a neighbour in $B_z$ for each $z\in L_j$; and so $u$ is mixed on $B_z$ for all such $z$.
			By induction applied to $u$ versus $(B_z:z\in L_j)$, $G$ contains an induced path $P$ of length $r$ where $u$ is an endpoint and the other vertices are in $C_j$.
			Since $\mac B$ is $(1-\abs{J_{r+1}}^{-1})$-semidense, there is an induced path $Q$ in $G$ such that $u\in Q$ and $\abs{Q\cap C_i}=1$ for all $i\in\{j+1,\ldots,s\}$.
			Let $\mac C:=(C_s)\cup(C_g:g\in V(H)\setminus V(R))$.
			Let $u_s\in V(Q)\cap C_s$, embed $g_s$ into $u_s$, then greedily embed a copy $\varphi$ of $H\setminus(V(R)\setminus\{g_s\})$ along $\mac C$ such that $\varphi(g_s)=u_s$, which is possible since $\mac B$ is  $(1-\abs{J_{r+1}}^{-1})$-semidense.
			Since $P\text-u\text-Q$ is an induced path of length at least $s$, assembling $P\text-u\text-Q$ and $\varphi$ then gives a copy of $H$ in $G$, a contradiction.
			This proves \cref{claim:compl}.
		\end{subproof}
		
		Now, since $H\setminus(V(R)\setminus\{g_s\})$ is connected with at least two vertices, $g_s$ is adjacent to some $g_{s+1}\in V(H)\setminus(V(R)\setminus\{g_s\})$.
		Let $C_{s+1}:=C_{g_{s+1}}$; this allows for the existence of $C_{r+1}$.
		
		Next, let $v_{r}\in C_{r}\setminus N_G(v)$. Then there exists $j\in\{2,\ldots,r\}$ minimal such that $G$ contains an induced path $v_{r}\text-v_{r-1}\text-\cdots\text-v_j$ where $v_i\in C_i\setminus N_G(v)$ for all $i\in\{j,j+1,\ldots,r\}$.
		If $j>2$, then by \cref{claim:compl}, $v_j\in C_j$ is complete in $G$ to $B_z$ for some $z\in L_{j-1}$; and so there exists $v_{j-1}\in B_z\setminus N_G(v)$, violating the minimality of $j$. Therefore $j=2$.
		
		Now, since $\mac B$ is $(1-\abs{J_{r+1}}^{-1})$-semidense, there exists $v_{r+1}\in C_{r+1}\cap N_G(v_{r})$.
		Let $v_1$ be a neighbour of $v$ in $C_1$ if $v_{r+1}\in C_{r+1}\setminus N_G(v)$, and let $v_1$ be a non-neighbour of $v$ in $C_1$ if $v_{r+1}\in C_{r+1}\cap N_G(v)$.
		Then $v_1$ is always adjacent to $v_2$, since $C_2$ is complete to $C_1$.
		Thus, \[v\text-v_1\text-v_2\text-\cdots\text-v_{r}\text-v_{r+1}\quad\text{(if $v_1\in C_1\cap N_G(v)$)}
		\qquad\text{or}\qquad v\text-v_{r+1}\text-v_{r}\text-\cdots\text-v_1\quad \text{(if $v_1\in C_1\setminus N_G(v)$)}\]
		is an induced $P_{r+2}$ in $G$ where $v$ is an endpoint and the other vertices are in $\bigcup_{B\in\mac B}B$, completing the induction step.
		This proves \cref{lem:switchin}.
	\end{proof}
	
	\cref{lem:switchin} suggests the following definition. 
	For a graph $H$ and a tree $T$ with $r_0\in V(T)$, we say that a graph $J$ {\em forces $(T,r_0)$ via $H$} if for every $H$-free graph $G$, with $v\in V(G)$ and a $(1-(\abs H\abs J)^{-1})$-semidense blockade $\mac B$ with pattern $J$ in $G\setminus v$,
	such that each block of $\mac B$ induces a connected subgraph in $G$ and $v$ is mixed on every block of $\mac B$, there is a copy $\varphi$ of $T$ in $G$ with $\varphi(r_0)=v$ and $\varphi(V(T)\setminus\{r_0\})\subset \bigcup_{B\in\mac B}B$.
	Thus, \cref{lem:switchin} with $r=s$ implies that if $H$ is an $s$-lollipop, then $J_s$ forces $(P_{s+1},r_0)$ via $H$, where $r_0$ is an endpoint of $P_{s+1}$ (note that $\abs{J_{s+1}}=\abs F\abs{J_s}$).
	The following lemma follows straightforwardly from the definition.
	
	\begin{lemma}
		\label{lem:force}
		Let $H$ be a graph, and let $h\ge2$ be an integer.
		For each $i\in[h]$, let $F_i$ force $(T_i,r_i)$ via $H$. 
		Let $T$ be the tree obtained from the disjoint union of $T_1,\ldots,T_h$ by identifying $r_1,\ldots,r_h$ into one vertex $r_0$.
		Then the disjoint union $F_0$ of $F_1,\ldots,F_h$ forces $(T,r_0)$ via $H$.
	\end{lemma}
	\begin{proof}
		Let $G$ be an $H$-free graph with $v\in V(G)$, and a $(1-(\abs H\abs {F_0})^{-1})$-semidense blockade $\mac B$ with pattern $F_0$ in $G\setminus v$,
		such that each block of $\mac B$ induces a connected subgraph in $G$ and $v$ is mixed on every block of $\mac B$.
		For each $i\in[h]$, let $D_i$ be the union of the blocks with pattern $F_i$; then $D_1,\ldots,D_h$ are pairwise disjoint and anticomplete.
		For each $i\in[h]$, there is a copy $\varphi_i$ of $(T_i,r_i)$ in $G$ with $\varphi_i(r_i)=v$ and $\varphi_i(V(T_i)\setminus\{r_i\})\subset D_i$.
		Then $(\varphi_i(V(T_i)\setminus\{r_i\}):i\in[h])$ are pairwise disjoint and anticomplete.
		Assembling $\varphi_1,\ldots,\varphi_h$ then gives a copy $\varphi$ of $T$ in $G$ with $\varphi(r_0)=v$ and $\varphi(V(T)\setminus\{r_0\})\subset \bigcup_{B\in\mac B}B$.
		This proves \cref{lem:force}.
	\end{proof}
	
	We are now ready for our applications to $S_{1,2,r}$, in the more general setting of {\em subdivided stars}.
	Here, for integers $h\ge1$ and $n_1,\ldots,n_h\ge1$, let $S_{n_1,\ldots,n_h}$ be the graph obtained from the star $K_{1,h}$ by subdividing its $i$-th edge exactly $n_i-1$ times, for all $i\in[h]$; and the {\em centre} of this graph is the centre of the original $K_{1,h}$.
	Then for every $s_1,\ldots,s_h\ge2$, the subdivided star $S_{1,s_1,\ldots,s_h}$ is clearly an $s_i$-lollipop for all $i\in[h]$.
	Combining \cref{lem:switchin,lem:force} then yields the following.
	
	\begin{lemma}
		\label{lem:forcesubstar}
		Let $h\ge1$ and $s_1,\ldots,s_h\ge2$ be integers, let $s_0=1$, and let $H:=S_{s_0,s_1,\ldots,s_h}$  with centre $r_0$ and $s_h=\max_{i\in[h]}s_i$.
		Let $F$ be obtained from $H$ by removing the leaf in the $h$-th branch.
		Then, given the definition \eqref{eq:sub2}, the disjoint union of $J_{s_0},J_{s_1},\ldots,J_{s_h}$ forces $(H,r_0)$ via $H$.
	\end{lemma}
	\begin{proof}
		By \cref{lem:switchin} and since $s_h=\max_{i\in[h]}s_i$, each copy of $J_{s_i}$ forces $(P_{s_i+1},r_0)$ via $H$, for all $i\in[h]$; and clearly $J_{s_0}=J_1=K_1$ forces $(P_2,r_0)$ via $H$. The conclusion follows from \cref{lem:force}.
	\end{proof}
	
	Now we give a proof sketch of \cref{lem:niceeh} for $S_{1,s_1,\ldots,s_h}$, as follows.
	
	\begin{lemma}
		\label{lem:niceehsubstar}
		Let $h\ge1$ and $s_1,\ldots,s_h\ge2$ be integers, let $s_0=1$, and let $H:=S_{s_0,s_1,\ldots,s_h}$  with centre $r_0$ and $s_h=\max_{i\in[h]}s_i$.
		Let $F$ be obtained from $H$ by removing the leaf in the $h$-th branch.
		If $F$ has the \erh{} property and $H$ is nice, then $H$ has the \erh{} property.
	\end{lemma}
	\begin{proof}
		[Proof sketch]
		For each $i\in[h]$, let $J_{s_i}$ be as in \eqref{eq:sub2} with $r=s_i$.
		Let $J$ be the disjoint union of $J_{s_0},J_{s_1},\ldots,J_{s_h}$.
		By \cref{lem:forcesubstar}, $J$ forces $(H,r_0)$ via $H$.
		By \cref{thm:aps}, each $J_{s_i}$ has the \erh{} property since $F$ does; and so $J$ has the \erh{} property as well.
		The rest of the proof then just numerically modifies the argument in~\cref{subsec:densify}; we omit the details.
	\end{proof}
	
	\cref{lem:niceehsubstar} with $h=2$, $s_1=2$ and $s_2=t$ implies the following, which is our objective at the beginning of this subsection for $S_{1,2,t}$.
	
	\begin{lemma}
		\label{lem:niceehspider}
		For every integer $t\ge2$, if $S_{1,2,t-1}$ has the \erh{} property and $S_{1,2,t}$ is nice, then $S_{1,2,t}$ has the \erh{} property.
	\end{lemma}
	
	It remains to handle $S_{1,2,t}'$; let us now briefly sketch what needs to change.
	To begin with, $S_{1,2,t}'$ is a $t$-lollipop for every $t\ge2$, where $g_t$ is a vertex of degree two of the remaining paw (recall that the {\em paw} is obtained by adding a pendant edge to a triangle).
	In particular, we can use $S_{1,2,t-1}'$ to force a copy of $P_{t+1}$ rooted at $v$.
	Our change here is to strengthen the mixed property to `dense-mixed', meaning that for each $B\in \mac B$, we require $v$ to have at least $\frac12\abs B$ neighbours and at least one non-neighbour in $B$; we can arrange such a `dense-mixed' property in the proof of \cref{lem:pathdense1}.
	Now, to force a copy of $S_{1,2,t}'$ provided that it satisfies \cref{lem:nearnice} and $S_{1,2,t-1}'$ has the \erh{} property,
	let $J$ be a sufficiently large repeated blow-up of $S_{1,2,t-1}'$ so that $J$ can force a copy of $P_{t+1}$ rooted at $v$.
	To embed the remaining paw so that $g_t$ is embedded into $v$, we employ a $P_3$ pattern in $\mac B$.
	The `dense-mixed' property allows us to first embed the leaf of the paw into a non-neighbour of $v$ in a leaf block,
	then use the semidense and dense-mixed properties to embed the neighbours of $g_t$ into the other two blocks appropriately.
	Consequently, the disjoint union of $J$ and $P_3$ would force a copy of $S_{1,2,t}'$, as desired.
	We omit the detailed proof here.
	
\section{Additional remarks}
	We conclude the paper with the following common generalisation of \cref{lem:splitp4,lem:split0}.
\begin{lemma}
	\label{lem:split2}
	For every integer $r\ge2$, there exists $b\ge2$ such that for every $x,y\in(0,\frac12]$ with $x\le y^b$, and every graph $G$ with disjoint $A,B\subset V(G)$ satisfying $\abs B>x^{-b}$, either:
	\begin{itemize}
		\item $A\to_rB$ in $G$;
		
		\item there exists $E\subset B$ with $\abs E\ge y^b\abs B$ such that every vertex in $A$ is either anticomplete or $(1-x)$-dense to $E$ in $G$; or
		
		\item $G[B]$ contains a $(1-x)$-semidense $(k,\abs B/k^b)$-blockade, for some integer $k\in[y^{-1},x^{-1}]$.
	\end{itemize} 
\end{lemma}
	
	Similarly, we also have the following extension of \cref{lem:split}.
	
\begin{lemma}
	\label{lem:split1}
	For every integer $r\ge2$, there exists $d\ge2$ such that for every $x\in(0,2^{-d}]$ and every graph $G$ with disjoint $A,B\subset V(G)$ satisfying $\abs B>x^{-d}$, either:
	\begin{itemize}
		\item $A\to_rB$ in $G$;
		
		\item there exists $E\subset B$ with $\abs E\ge \abs B/4$ such that every vertex in $A$ is either anticomplete or $(1-x)$-dense to $E$ in $G$; or
		
		\item $G[B]$ contains a $(1-x)$-semidense $(k,\abs B/k^d)$-blockade, for some integer $k\in[2,x^{-1}]$.
	\end{itemize} 
\end{lemma}

	We observe that \cref{lem:split1} implies \cref{lem:split2}, in a manner similar to how \cref{lem:split} implies \cref{lem:split0} as shown in \cref{subsec:split}.
	
	Let us give a proof sketch of \cref{lem:split1}. The main inputs are the viral property of all paths, which follows from \cref{thm:path,thm:bfp}, and the following variant of \cref{lem:blockcount} with $k=2$.
\begin{lemma}
	\label{lem:dualcount}
	Let $x\in(0,\frac12]$, and let $H$ be a graph with a vertex $v$ of degree one.
	Let $G$ be a graph with an $x$-sparse pair $(X,Y)$.
	Then for every $y\in(0,1)$, either:
	\begin{itemize}
		\item there is a copy $\varphi$ of $H$ in $G$ such that $\varphi(v)\in Y$ and $\varphi(V(H)\setminus\{v\})\subset X$;
		
		\item there exists $Q\subset X$ with $\abs Q\ge(1-y)\abs X$ such that $Q$ is anticomplete to $Y$ in $G$; or
		
		\item there exists $R\subset X$ with $\abs R\ge y\abs X$ such that $\ind_{H\setminus v}(G[R])<\abs H\cdot xy^{-1}\abs R^{\abs H-1}$.
	\end{itemize}
\end{lemma}
\begin{proof}
	Assume that the first two outcomes fail; then $\abs H\ge3$.
	Let $u$ be the unique neighbour of $v$ in $H$.
	Let $R$ be the set of vertices in $X$ with a neighbour in $Y$; then $\abs R\ge y\abs X$ since the second outcome fails.
	For each $z\in R$, let $S_z$ be the set of copies $\phi$ of $H\setminus v$ in $G[R]$ with $\phi(u)=z$.
	For each such $z$, let $z_Y\in Y$ be a neighbour of $z$; then for every $\phi\in S_z$, since the first outcome of the lemma fails, we see that $\phi(V(H)\setminus\{u,v\})\cap N_G(z_Y)$ is nonempty.
	Hence, since $Y$ is $x$-sparse to $X$, we obtain $\abs{S_z}< (\abs H-2)\cdot x\abs X\cdot \abs R^{\abs H-3}$.
	Because this holds for all $z\in R$, we deduce that
	\[\ind_{H\setminus v}(G[R])=\sum_{z\in R}\abs{S_z}< \abs H\cdot x\abs X\cdot \abs R^{\abs H-2}\le \abs H\cdot xy^{-1}\abs R^{\abs H-1}.\]
	This verifies the third outcome of the lemma, and in turn proves \cref{lem:dualcount}.
\end{proof}
	We now sketch the proof of \cref{lem:split1}, as follows.
\begin{proof}
	[Proof sketch of \cref{lem:split1}]
	We repeat the proof of \cref{lem:split} by induction on $r$ as in \cref{subsec:mainsplit}, where $x$ is replaced by $x^a$ for some large $a\ge2$ depending on $r$, and `$x$-pure' is replaced by `anticomplete or $(1-x^a)$-dense', with appropriate analogues of \cref{lem:xpure,thm:convsemi}.
	Now, after the proof of \cref{claim:sparsify}, apply \cref{lem:dualcount} with $(x,y,X,Y)=(x^a,x,B\setminus X,P)$. 
	If the second outcome holds then we can remove a small portion of $B\setminus X$ such that the rest is anticomplete to $P$; so we may assume that third outcome holds.
	We then apply the viral property of $P_{r-1}$ via the choice of $a$ to get an $x^2$-restricted induced subgraph $F$ of $G[R]$, with $\abs F\ge x^b\abs R$ for some $b\ge2$ much smaller than $a$.
	The $(1-x^2)$-dense case is clear; so we assume that $F$ is $x^2$-sparse.
	Note that every vertex in $V(F)\subset R$ has a neighbour in $P$, and $P$ is still very sparse to $V(F)$.
	Then it suffices to run an $r$-vertex Gy\'arf\'as path from $P$ into $V(F)$, in a manner similar to~\cite[Lemma 3.1]{density5}, to get an anticomplete $(x^{-1},\poly(x)\abs F)$-blockade within $F$.	
	We omit the detailed proof.
\end{proof}
	
	\appendix
	\section{Building combs}
	\label{sec:buildcombs}
	This section provides a proof of \cref{lem:getcomb0}, which follows the proof of~\cite[Lemma 6.8]{2025thes} and modifies the proof of~\cite[Lemma 5.2]{density7}.
	We need the following lemma, which is~\cite[Lemma 4.2]{density7}.
	\begin{lemma}
		\label{lem:locate}
		Let $x>0$, and let $G$ be a bipartite graph with bipartition $(A,B)$ where $A\ne\emptyset$ and every vertex in $B$ has at least $x\abs A$ neighbours in $A$.
		Then there exists $A'\subset A$ such that $\abs {A'}\le1/x$ and there are at least $\frac12\abs B$ vertices in $B$ with a neighbour in $A'$.
	\end{lemma}
	
	We are now ready to prove \cref{lem:getcomb0}, which we restate here in the sparse setting for notational convenience.
	
	\begin{lemma}
		\label{lem:getcomb}
		Let $x,y\in(0,2^{-99}]$ with $x\le y$. Then every $y$-sparse graph $G$ with $\abs G> y^{-1}$ contains either:
		\begin{itemize}
			
			\item a $y^2$-sparse induced subgraph with at least $\abs G/2$ vertices;
			
			\item an $x$-sparse $(y^{-1/6},y^3\abs G)$-blockade; or
			
			\item non-adjacent $u,v\in V(G)$ and a $(k,\abs G/k^{99})$-comb $((a_i,B_i):i\in[k])$ in $G\setminus\{u,v\}$, where $k\in[y^{-1/24},x^{-1}]$, such that $a_i\in N_G(u)\setminus N_G(v)$ and $B_i\subset N_G(v)\setminus N_G(u)$ for all $i\in[k]$.
		\end{itemize}
	\end{lemma}
	\begin{proof}
		Suppose that none of the outcomes holds.
		Let $s$ be the maximum integer with $0\le s\le \ceil{y^{-1/6}}$, such that there are $u_1,\ldots,u_s\in V(G)$ and disjoint $V_1,\ldots,V_s\subset V(G)\setminus\{u_1,\ldots,u_s\}$ satisfying:
		\begin{itemize}
			\item for all $i\in[s]$, $V_i\subset N_G(u_i)$ and $\abs{V_i}\ge y^2\abs G/2$;
			
			\item for all $i,j\in[s]$ with $i<j$, $u_i$ has no neighbour in $\{u_j\}\cup V_j$; and
			
			\item for all $i,j\in[s]$ with $i<j$, $\{u_j\}\cup V_j$ is $y^{1/3}$-sparse to $V_i$.
		\end{itemize}
		(These conditions are clearly satisfied for $s=0$.)
		
		\begin{claim}
			\label{claim:build1}
			$s= \ceil{y^{-1/6}}$.
		\end{claim}
		\begin{subproof}
			Suppose not; then $s<y^{-1/3}/2$. Let $E:=\bigcup_{i\in[s]}N_G[u_i]$; then
			\[\abs E\le s\cdot (1+y\abs G)\le 2sy\abs G<y^{2/3}\abs G.\]
			In particular, every vertex in $E$ has fewer than $y\abs G<2y\abs{G\setminus E}$ neighbours in $V(G)\setminus E$.
			For every $i\in[s]$, let $E_i$ be the set of vertices in $V(G)\setminus E$ with at least $y^{1/3}\abs{V_i}$ neighbours in $V_i$;
			then $\abs{E_i}\le 2y^{2/3}\abs{G\setminus E}$ by averaging.
			Let $F:=(G\setminus E)\setminus(E_1\cup\cdots\cup E_s)$; then
			\[\abs F\ge (1-2sy^{2/3})\abs{G\setminus E}
			\ge (1-y^{1/3})(1-y^{2/3})\abs G\ge \abs G/2.\]
			Hence, since the first outcome of the lemma fails, there exists $u_{s+1}\in V(F)$ with degree in $F$ at least $y^2\abs F\ge y^2\abs G/2$.
			Let $V_{s+1}:=N_F(u_{s+1})$;
			then $((u_i,V_i):i\in[s+1])$ violates the maximality of $s$, a contradiction.
			This proves \cref{claim:build1}.
		\end{subproof}
		
		Now, for every $i\in[s]$, let $D_i$ be the set of vertices in $V_i$ with no neighbour in $\{u_j:j>i\}$; then
		\[\abs{D_i}\ge (1-(s-i)y^{1/3})\abs{V_i}\ge  (1-(s-1)y^{1/3})\abs{V_i}\ge \abs{V_i}/2\ge y^2\abs G/4\]
		and in particular $D_j\subset V_j$ is $2y^{1/3}$-sparse to $D_i$ for all $j>i$.
		Let $r\in[s]$ be minimal such that there exists $Q_j\subset D_j$ for each $j\in\{r,r+1,\ldots,s\}$ satisfying:
		\begin{itemize}
			\item $\abs{Q_j}\ge\abs{D_j}/2$ for all $j\in\{r,r+1,\ldots,s\}$; and
			
			\item $Q_i$ is $x$-sparse to $Q_j$ for all $i,j\in\{r,r+1,\ldots,s\}$ with $i<j$.
		\end{itemize}
		(These conditions are clearly satisfied for $r=s$, by taking $Q_s:=D_s$.)
		Observe that $(Q_s,\ldots,Q_{r+1},Q_r)$ is $x$-sparse; and for all $j\in\{r,r+1,\ldots,s\}$,
		\[\abs{Q_j}\ge \abs{D_j}/2\ge\abs{V_j}/4\ge y^2\abs G/8\ge y^3\abs G.\]
		Thus, since the second outcome of the lemma fails, we deduce that $r>1$.
		Let $Y:=D_{r-1}$. For each $j\in\{r,r+1,\ldots,s\}$, let $P_j$ be the set of vertices in $Y$ with at least $x\abs{Q_j}$ neighbours in $Q_j$.
		
		\begin{claim}
			\label{claim:build2}
			$\abs{P_j}\le y^{1/6}\abs Y/2$ for all $j\in\{r,r+1,\ldots,s\}$.
		\end{claim}
		\begin{subproof}
			Suppose not.
			By \cref{lem:locate} with $(A,B)=(Q_j,P_j)$, there exists $X\subset Q_j$ such that $\abs X\le 1/x$ and at least $\frac12\abs{P_j}$ vertices in $P_j$ have a neighbour in $X$.
			By \cref{lem:comb} with $(A,B,\Delta)=(X,P_j,2y^{1/3}\abs Y)$, either:
			\begin{itemize}
				\item at most $20\sqrt{\abs{P_j}\cdot 2y^{1/3}\abs Y}$ vertices in $P_j$ have a neighbour in $X$; or
				
				\item for some integer $k\ge1$, $G$ contains a $(k,\abs{P_j}/k^2)$-comb $((a_i,B_i):i\in[k])$ where $a_i\in X$ and $B_i\subset P_j$ for all $i\in[k]$.
			\end{itemize}
			
			Our supposition and $y\le 2^{-99}$ together imply 
			\[20\sqrt{\abs{P_j}\cdot 2y^{1/3}\abs Y}<40y^{1/12}\abs{P_j}<\abs{P_j}/2\]
			and so the first bullet cannot hold.
			Hence, the second bullet holds.
			It follows that $k\le \abs X\le 1/x$ and $\abs{P_j}/k^2\le 2y^{1/3}\abs Y\le 4y^{1/6}\abs{P_j}$, which yields $k\ge \frac12y^{-1/12}\ge y^{-1/24}\ge 8$.
			Therefore
			\[\abs{P_j}/k^2\ge y^{1/6}\abs Y/(2k^2)\ge y^{13/6}\abs G/(8k^2)
			\ge \abs G/k^{99}\]
			and so the last outcome of the lemma holds with $(u,v)=(u_j,u_{r-1})$, contrary to our supposition.
			This proves \cref{claim:build2}.
		\end{subproof}
		
		Now, let $Q_{r-1}:=Y\setminus\left(\bigcup_{j\ge r}P_j\right)$. Then $Q_{r-1}$ is $x$-sparse to $Q_j$ for all $j\in\{r,r+1,\ldots,s\}$; and \cref{claim:build2} implies that
		\[\abs{Q_{r-1}}\ge (1-(s-r+1)y^{1/6}/2)\abs Y
		\ge (1-(s-1)y^{1/6}/2)\abs Y\ge \abs Y/2=\abs{D_{r-1}}/2\]
		which violates the minimality of $r$.
		This proves \cref{lem:getcomb}.
	\end{proof}
	
	\section{Splitting three-vertex paths by avoiding large modules}
	\label{sec:p3}
	This section proves \cref{lem:p3}, which we restate here for the reader's convenience.
	\begin{lemma}
		\label{lem:p30}
		Let $x,y\in(0,2^{-8}]$ with $x\le y$. Then every graph $G$ with $\abs G> x^{-4}$ contains either:
		\begin{itemize}
			
			\item a pure $(k,\abs G/k^{8})$-blockade, for some integer $k\in[y^{-1/2},x^{-1}]$; or
			
			\item an $x$-splittable induced subgraph with at least $y^2\abs G$ vertices.
		\end{itemize}
	\end{lemma}
	
	Here, recall that a graph $F$ is {\em $x$-splittable} if for every partition $(X,Y)$ of $V(F)$ with $\abs X\ge x\abs F$ and $Y\ne\emptyset$, there is an induced path $v_1\text-v_2\text-v_3$ in $F$ with $v_1\in Y$ and $v_2,v_3\in X$, or equivalently $Y\to_3X$.
	
	Let us first give an idea of how the proof of \cref{lem:p30} would proceed; to avoid technicalities, we relax the first outcome to pure $(k,\abs G/\poly(k))$-blockades.
	Now, assume that the second outcome fails.
	Then each induced subgraph $F$ of $G$ with $\abs F\ge y^2\abs G$ is not $x$-splittable; and so $V(F)$ admits a partition $(X,Y)$ such that $Y\not\to_3X$, $\abs X\ge x\abs F$, and $Y\ne\emptyset$.
	If $F[X]$ has no component of size larger than $x\abs X$ say, then \cref{lem:conn} gives an anticomplete $(x^{-1},\poly(x)\abs X)$-blockade in $F[X]$ (and so in $G$); and so the first outcome holds since $\poly(x)\abs X\ge \poly(x)\abs F\ge \poly(x)\abs G$.
	Thus, $F[X]$ has a large component $C$ say; and since $Y\not\to_3V(C)$, every vertex in $V(F)\setminus V(C)$ is pure to $V(C)$.
	
	Consequently, \cref{lem:p30} amounts to showing that for every large graph $G$, either there is a suitable pure blockade in $G$, or $G$ has a sufficiently large induced subgraph with no large module.
	Here, as is well-known, a \emph{module} (or {\em homogeneous set}) in a graph $F$ is a nonempty subset $M\subset V(F)$ such that every vertex in $V(F)\setminus M$ is pure to $M$; and we say that $M$ is {\em proper} if $\abs M<\abs F$.
	Then $M$ is a module of $F$ if and only if it is a module of $\overline F$; and a subset $N$ of $M$ is a module of $F[M]$ if and only if it is a module of $F$.
	Given these definitions, the proof of \cref{lem:p30} uses the following lemma.
	
	\begin{lemma}
		\label{lem:module}
		Let $x,y\in(0,2^{-8}]$ with $x\le y$. Then every graph $G$ with $\abs G>x^{-4}$ contains either:
		\begin{itemize}
			\item a pure $(k,\abs G/k^{8})$-blockade, for some integer $k\in[y^{-1/2},x^{-1}]$; or
			
			\item an induced subgraph $F$ such that $\abs F\ge y^2\abs G$ and all proper modules of $F$ have size less than $4x^4\abs G$.
		\end{itemize}
	\end{lemma}
	\begin{proof}
		[Proof of \cref{lem:p30}, assuming \cref{lem:module}]
		Assume that the first outcome does not hold.
		Then \cref{lem:module} gives an induced subgraph $F$ of $G$ such that $\abs F\ge y^2\abs G$ and all proper modules of $F$ have size less than $4x^4\abs G$.
		We claim that $F$ is $x$-splittable.
		Let $(X,Y)$ be a partition of $V(F)$ with $\abs X\ge x\abs F$ and $Y\ne\emptyset$, and put $k:=\floor{(4x)^{-1}}$.
		Then $y^{-1/2}\le x^{-1/2}\le(8x)^{-1}\le k\le x^{-1}$.
		Moreover,
		\[k^7\abs X
		\ge (8x)^{-7}xy^2\abs G
		=8^{-7}x^{-6}y^2\abs G
		\ge 8^{-7}x^{-4}\abs G
		>2\abs G.\]
		Thus $\abs X/(2k)\ge\abs G/k^8$.
		Since the first outcome fails, \cref{lem:conn} implies that $F[X]$ has a component $C$ with
		\[\abs C\ge\abs X/k\ge4x\abs X\ge4x^2\abs F\ge4x^2y^2\abs G\ge4x^4\abs G.\]
		As $V(C)\subseteq X\subsetneq V(F)$, it cannot be a module of $F$.
		Hence some vertex $v_1\in V(F)\setminus V(C)$ is mixed on $V(C)$; because $C$ is a component of $F[X]$, this vertex lies in $Y$.
		By the connectedness of $C$, there is an edge $v_2v_3$ of $C$ such that $v_1$ is adjacent to $v_2$ and nonadjacent to $v_3$.
		Thus $v_1\text-v_2\text-v_3$ is an induced path with $v_1\in Y$ and $v_2,v_3\in X$, proving that $F$ is $x$-splittable.
	\end{proof}
	
	The proof of \cref{lem:module} is via the machinery of modular decomposition developed by Gallai~\cite{gallai1967}, which we discuss in detail below. On a side note, to the best of our knowledge of the literature on the \erh{} conjecture \ref{conj:eh}, a (very) distant form of modular decomposition was previously employed by Choromanski, Falik, Liebenau, Patel, and Pilipczuk~\cite{cdlpp2018}.
	
	\subsection*{Modular decomposition}
	Observe that for every two disjoint modules $M,N$ of $G$, $(M,N)$ is a pure pair in $G$.
	For $\abs G\ge2$ and a collection $\mac M=\{M_1,\ldots,M_k\}$ of pairwise disjoint proper modules in $G$ that together partition $V(G)$, the {\em quotient graph} $G/\mac M$ is defined as the graph on vertex set $\mac M$ where $M_i,M_j$ are adjacent if and only if they are complete to each other in $G$.
	
	We now define $\mac P(G)$ to be the collection of:
	\begin{itemize}
		\item the vertex sets of all components of $G$ if $G$ is disconnected;
		
		\item the vertex sets of all components of $\overline G$ if $\overline G$ is disconnected; and
		
		\item the maximal proper modules of $G$ if both $G$ and $\overline G$ are connected.
	\end{itemize}
	
	We say that $G$ is {\em prime} if it has at least four vertices and all of its proper modules have size $1$.
	We collect the following simple fact due to Gallai~\cite{gallai1967}, for which we give a proof for completeness.
	
	\begin{lemma}
		[Gallai]
		\label{lem:gal}
		For $\abs G\ge2$, the members of $\mac P(G)$ always partition $V(G)$; in particular the graph $G/\mac P(G)$ is well-defined.
		Moreover, this graph  is edgeless if $G$ is disconnected, is complete if $\overline G$ is disconnected, and is prime if both $G$ and $\overline G$ are connected.
	\end{lemma}
	\begin{proof}
		Let $\mac P:=\mac P(G)$. It suffices to show that if both $G$ and $\overline G$ are connected, then the members of $\mac P(G)$ partition $V(G)$ and $G/\mac P$ is prime.
		
		To show the partition property, observe that for each $v\in V(G)$, the set $\{v\}$ is a module in $G$, and so is contained in some maximal proper module of $G$.
		Thus $V(G)=\bigcup_{M\in \mac P}M$. Now, for distinct $M,N\in \mac P$, suppose that there exists $v\in M\cap N$. The maximality of $M,N$ implies that $N\setminus M$ is nonempty.
		Observe that $M$ is pure to $N\setminus M$; and every vertex in $N\setminus M$ has the same type of adjacency to each vertex in $M$ as to $v$.
		Thus $(M,N\setminus M)$ is a pure pair in $G$, which implies that $M\cup N\subsetneq V(G)$ since both $G$ and $\overline G$ are connected.
		Then every vertex in $V(G)\setminus(M\cup N)$ has the same type of adjacency to each vertex in $M\cup N$ as to $v$; and so $M\cup N$ is a proper module of $G$, violating the maximality of $M$ and $N$.
		Therefore, the members of $\mac P$ are pairwise disjoint and so partition $V(G)$.
		
		Thus, $F:=G/\mac P$ is well-defined and $\abs F\ge2$. Now, let $A$ be a proper module of $F$; then $\bigcup_{M\in A}M$ is a proper module of $G$, which implies that $\abs A=1$ since each $M\in A$ is maximal.
		Since $F$ and $\overline F$ are connected, $\abs F\ge4$; and so $F$ is prime.
		This proves \cref{lem:gal}.
	\end{proof}
	
	The graph $G/\mac P(G)$ is called the {\em canonical quotient} of $G$.
	Note that if $G$ has at least two vertices then so does its canonical quotient.
	
	In what follows, a {\em rooted tree} is a pair $(T,\circ)$ where $T$ is a tree and $\circ\in V(T)$, and we call $\circ$ the {\em root} of $(T,\circ)$.
	For distinct nodes $u,v$ in $(T,\circ)$, we say that $u$ is a {\em parent} of $v$ in $(T,\circ)$, or equivalently $v$ is a {\em child} of $u$ in $(T,\circ)$, if $u$ is adjacent to $v$ in $T$ and $u$ lies on the unique path of $T$ joining $\circ$ and $v$.
	We say that $u,v$ are {\em comparable} if one lies in the unique path between the other and the root $\circ$, and {\em incomparable} if they are not comparable.
	A {\em chain} in $(T,\circ)$ is a sequence of pairwise comparable nodes.
	A {\em leaf} of $(T,\circ)$ is a node with no children.
	For a graph $G$, the {\em modular decomposition tree} of $G$, denoted by $\mac D(G)$, is the rooted tree with the following properties:
	\begin{itemize}
		\item the root of $\mac D(G)$ is $V(G)$, and the leaves of $\mac D(G)$ are the singletons of $V(G)$; and
		
		\item each node $M$ of $\mac D(G)$ is a module of $G$, and if $M$ is an internal node then its children in $\mac D(G)$ are the members of $\mac P(G[M])$.
	\end{itemize}
	
	(Note that this is well-defined because each singleton is a module; and for all modules $M$ of $G$ with at least two vertices, every member of $\mac P(G[M])$ has fewer than $\abs M$ vertices.)
	Observe that every two incomparable nodes in $\mac D(G)$ are disjoint and pure to each other in $G$, because they are subsets of two distinct children of some node in $\mac D(G)$.
	
	We say that an internal node $M$ in $\mac D(G)$ is {\em prime} if the canonical quotient of $G[M]$ is prime.
	The proof of \cref{lem:module} is via the following fact about $\mac D(G)$.
	
	\begin{lemma}
		\label{lem:decomp}
		Let $x,y\in(0,2^{-8}]$ with $x\le y$, and let $G$ be a graph with $\abs G>x^{-4}$. Then either:
		\begin{itemize}
			\item\label{out:pure}
			there exist an integer $k\in[y^{-1/2},x^{-1}]$ and a pure $(k,\abs G/k^{8})$-blockade in $G$; or
			
			\item\label{out:prime}
			some prime node $M$ of $\mac D(G)$ has a collection
			$\mac C$ of children such that
			\[
			\abs C<4x^{4}\abs G\quad\text{for all $C\in\mac C$},
			\qquad\text{and}\qquad
			\abs*{\,\bigcup_{C\in\mac C}C\,}\ge y^2\abs G.
			\]
		\end{itemize}
	\end{lemma}

	\begin{proof}
		[Proof of \cref{lem:module}, assuming \cref{lem:decomp}]
		Assume that the first outcome does not hold.
		Then \cref{lem:decomp} gives a prime node $M$ of $\mac D(G)$ with a collection $\mac C$ of its children such that
		\[\abs C<4x^{4}\abs G\quad\text{for all $C\in\mac C$},
		\qquad\text{and}\qquad
		\abs*{\,\bigcup_{C\in\mac C}C\,}\ge y^2\abs G.\]
		For each child $N$ of $M$ that is not in $\mac C$, choose a vertex $v_N\in N$.
		Let $F$ be the subgraph of $G$ induced on $\bigcup_{C\in\mac C}C\cup\{v_N:N\notin \mac C\}$.
		Then $\abs F\ge y^2\abs G$, and $\mac P(F)=\mac C\cup\{\{v_N\}:N\notin \mac C\}$ since $M$ is a prime node of $\mac D(G)$.
		Thus every proper module of $F$ has size $1$ or less than $4x^{4}\abs G$.
		Since $\abs G>x^{-4}$, this proves \cref{lem:module}.
	\end{proof}
	
	We now present a proof of \cref{lem:decomp}, as follows.
	
	\begin{proof}
		[Proof of \cref{lem:decomp}]
		Suppose that both outcomes fail.
		Let $w:=y^2\abs G>y^{-2}>1$. We say that a node $M$ of $\mac D(G)$ is {\em heavy} if it has size at least $w$, and {\em light} if it has size less than $w$.
		For each heavy node $M$, which is necessarily internal, let $\mac L(M):=\{C\in\mac P(G[M]):\abs C<w\}$ be the collection of its light children, and let
		\[b(M):=\sum_{C\in \mac L(M)}\abs C.\]
		
		We claim that $b(M)$ is small for all heavy nodes $M$, as follows.
		\begin{claim}
			\label{claim:light}
			$b(M)\le (5/2)y^{3/2}\abs G$ for all heavy nodes $M$ of $\mac D(G)$.
		\end{claim}
		\begin{subproof}
			Suppose that $M$ is heavy and $b(M)>(5/2)y^{3/2}\abs G$.
			If one of $G[M]$ or $\overline G[M]$ is disconnected, let $U$ be the union of the members of $\mac L(M)$ and put $k:=\ceil{y^{-1/2}}$.
			Since $y\le2^{-8}$ and $x\le y$, we have $y^{-1/2}\le k\le\frac{17}{16}y^{-1/2}\le x^{-1}$.
			Also $b(M)/k>\frac{40}{17}w>w$.
			The components of $G[U]$ or $\overline G[U]$, respectively, are precisely the members of $\mac L(M)$, and so each has size less than $b(M)/k$.
			By \cref{lem:conn}, the corresponding graph has an anticomplete $(k,b(M)/(2k))$-blockade, which is pure in $G$.
			Its width exceeds $\abs G/k^4\ge\abs G/k^8$, because $b(M)k^3>(5/2)\abs G>2\abs G$.
			This contradicts the failure of the first outcome.
			
			Thus, both $G[M]$ and $\overline G[M]$ are connected. Let $\mac C$ be the collection of the children of $M$ of size less than $4x^4\abs G$.
			Then since the second outcome of the lemma fails, the union of the members of $\mac C$ has size less than $y^2\abs G$.
			It remains to consider the children of $M$ of size at least $4x^4\abs G$ and less than $w=y^2\abs G$, that is, those in $\mac L(M)\setminus\mac C$.
			Let $s:=\floor{\log_2(y/x^2)}$. For every $j\in[s]$, let
			\[\mac C_j:=\{C\in\mac L(M)\setminus\mac C:4^{-j}\cdot w\le \abs C< 4^{1-j}\cdot w\}.\]
			Since $4^{-s}w=4^{-s}y^2\abs G\le 4x^4\abs G$, we have $\mac L(M)\setminus\mac C=\mac C_1\cup\cdots\cup\mac C_s$.
			For each $j\in[s]$, put $q_j:=2^{j/2}y^{-1/2}$ and $k_j:=\min\{\ceil{q_j},\floor{x^{-1}}\}$.
			Then $y^{-1/2}\le k_j\le x^{-1}$.
			If $k_j=\ceil{q_j}$, every member of $\mac C_j$ has size at least $4^{-j}w=\abs G/q_j^4\ge\abs G/k_j^4\ge\abs G/k_j^8$.
			If $k_j=\floor{x^{-1}}$, every member has size at least $4x^4\abs G\ge\abs G/k_j^4\ge\abs G/k_j^8$; here we use $k_j\ge x^{-1}-1>2^{-1/2}x^{-1}$.
			The members of $\mac C_j$ are pairwise pure, so the failure of the first outcome gives $\abs{\mac C_j}<k_j$.
			As $\abs{\mac C_j}$ is an integer and $k_j\le\ceil{q_j}$, we obtain $\abs{\mac C_j}<q_j$.
			Therefore
			\[\sum_{C\in\mac L(M)\setminus \mac C}\abs C =\sum_{j\in[s]}\sum_{C\in\mac C_j}\abs C
			< \sum_{j\ge1}2^{2-3j/2}y^{3/2}\abs G
			=\frac{4}{2^{3/2}-1}y^{3/2}\abs G
			<\frac94y^{3/2}\abs G.\]
			Since $y^{1/2}\le1/16$, it follows that
			\[b(M)<y^2\abs G+\frac94y^{3/2}\abs G
			\le\frac{37}{16}y^{3/2}\abs G<\frac52y^{3/2}\abs G,\]
			a contradiction.
			This proves \cref{claim:light}.
		\end{subproof}
		
		We now explore the structure of $\mac D(G)$. Because every non-root heavy node has a heavy parent, the set of heavy nodes induces a rooted subtree $\mact$ of $\mac D(G)$.
		Since every leaf of $\mac D(G)$ (that is, every singleton of $V(G)$) is light and so is contained in a maximal light node in $\mac D(G)$, we see that
		\begin{equation}
			\label{eq:lightpartition}
			\bigcup_{M\in V(\mact)}\mac L(M)\quad\text{is a partition of $V(G)$.}
		\end{equation}
		
		Let $\ell$ be the number of leaves of $\mact$. Because the leaves in $\mact$ are incomparable nodes in $\mac D(G)$, they form a pure blockade in $G$ of width at least $w=y^2\abs G$.
		Hence, since the first outcome of the lemma fails, we obtain $\ell<y^{-1/2}$.
		
		Now, let $V_1$ be the set of nodes in $\mact$ with exactly one child in $\mact$, and let $V_2$ be the set of nodes with zero or at least two children in $\mact$. 
		Since every rooted tree has at least as many leaves as the number of nodes with at least two children, we have $\abs{V_2}\le 2\ell< 2y^{-1/2}$.
		Hence, \cref{claim:light} yields
		\begin{equation}
			\label{eq:v2}\sum_{M\in V_2} b(M)< 2y^{-1/2}\cdot\frac52y^{3/2}\abs G=5y\abs G.
		\end{equation}
		
		It remains to bound the sum of $b(M)$ over all $M\in V_1$ via the following claim.
		
		\begin{claim}
			\label{claim:v1}
			$\sum_{M\in V_1}b(M)< 14y^{1/2}\abs G$.
		\end{claim}
		\begin{subproof}
			Let $\mac F$ be the collection of the vertex sets of the components of the induced forest $\mact[V_1]$.
			Each member of $\mac F$ is a chain in $\mac D(G)$, and its minimal member (with respect to inclusion) has no child in $V_1$.
			For each chain $P\in \mac F$, let $C(P)$ be the minimal member of $P$; then the unique heavy child of $C(P)$ in $\mac D(G)$ is in $V_2$.
			In particular, this gives an injection from $\mac F$ to $V_2$; and so $\mac F$ has at most $\abs{V_2}< 2y^{-1/2}$ chains.
			
			Now, fix $P\in\mac F$, and let $P':=P\setminus\{C(P)\}$.
			For each $M\in P'$, let $L_0(M)$ be the union of members of $\mac L(M)$ anticomplete to $C(P)$.
			For distinct $M,N\in P'$ such that $N$ lies between $M$ and $C(P)$ along $P$, the set $L_0(M)\subset M\setminus N$ is anticomplete to $C(P)\subset N$ in $G$, and so is anticomplete to $N\supset L_0(N)$ in $G$.
			Therefore, the nonempty members of $\{L_0(M):M\in P'\}$ form an anticomplete blockade in $G$.
			Let $G_0$ be the subgraph of $G$ induced on $\bigcup_{M\in P'}L_0(M)$.
			Because $\abs{L_0(M)}\le b(M)\le(5/2)y^{3/2}\abs G$ for all $M\in P'$ by \cref{claim:light}, every component of $G_0$ has size at most $(5/2)y^{3/2}\abs G$.
			If $\abs{G_0}\ge3y\abs G$, then for $k:=\ceil{y^{-1/2}}\le(17/16)y^{-1/2}$ we have
			\[\frac{\abs{G_0}}{k}\ge\frac{48}{17}y^{3/2}\abs G>\frac52y^{3/2}\abs G.\]
			Thus, by \cref{lem:conn} applied to $G_0$, there is an anticomplete $(k,\abs{G_0}/(2k))$-blockade in $G$.
			Since $k\ge2$ and $k^2\ge y^{-1}$, its width is at least $\abs{G_0}/k^2\ge3y\abs G/k^2\ge\abs G/k^4\ge\abs G/k^8$; also $k\in[y^{-1/2},x^{-1}]$.
			This contradicts the failure of the first outcome.
			Hence
			\[\sum_{M\in P'}\abs{L_0(M)}=\abs{G_0}\le 3y\abs G.\]
			Similarly, for each $M\in P'$, let $L_1(M)$ be the union of the members of $\mac L(M)$ complete to $C(P)$; then
			\[\sum_{M\in P'}\abs{L_1(M)}\le 3y\abs G.\]
			Therefore, since $b(C(P))\le(5/2)y^{3/2}\abs G$ by \cref{claim:light}, we deduce that
			\[\sum_{M\in P}b(M)=b(C(P))+\sum_{M\in P'}\abs{L_0(M)}+\sum_{M\in P'}\abs{L_1(M)}\le \frac52y^{3/2}\abs G+6y\abs G\le 7y\abs G.\]
			
			Because this estimate holds for every $P\in\mac F$ and $\abs{\mac F}< 2y^{-1/2}$, we obtain
			\[\sum_{M\in V_1}b(M)=\sum_{P\in \mac F}\sum_{M\in P}b(M)
			< 2y^{-1/2}\cdot 7y\abs G=14y^{1/2}\abs G.\qedhere\]
		\end{subproof}
		
		Now, \eqref{eq:lightpartition}, \eqref{eq:v2}, and \cref{claim:v1} together imply that
		\[\abs G=\sum_{M\in V_1}b(M)+\sum_{M\in V_2}b(M)< 5y\abs G+14y^{1/2}\abs G<\abs G\]
		where the last inequality holds since $y\le 2^{-8}$.
		This contradiction proves \cref{lem:decomp}.
	\end{proof}

\section*{Declaration of generative AI use}
\subsection*{Mathematical involvement}
	ChatGPT-6 Astra was mathematically involved in \cref{subsec:mainsplit,subsec:switch} and \cref{sec:p3}; the ideas in the rest of this paper are entirely due to the author.
	The involvement of Astra is as follows:
	\begin{itemize}
		\item The ideas of bookkeeping blockades and splitting paths, as described and motivated in \cref{subsec:bookkeep,subsec:split}, are due to the author.
		As a step towards \cref{lem:pathpure} for $P_6$ via \cref{lem:p3}, the author asked Astra to prove a weaker version of \cref{lem:module} (with `$x$-pure' replacing `pure') by using modular decomposition.
		Astra then succeeded in proving the full purity version via \cref{lem:decomp}.
		The author, having noticed the connections between \cref{lem:p3,lem:splitp4,lem:gyarfas}, was then convinced that a statement similar to \cref{lem:split0} should hold, and was also aware that \cref{lem:split0} follows from \cref{lem:split}.
		Following a few more discussions on bookkeeping blockades and splitting paths directed by the author, Astra discovered a proof of \cref{lem:split} for $r=5$.
		The author then observed that the proof can be extended to work for general $r$, which led to the content of \cref{subsec:split}.
		
		\item The ideas of repeated substitution (\`a la the tensor power trick) and extending paths by switching endpoints, as presented in \cref{sec:nicetoeh}, are due to the author. An early version of the argument in \cref{subsec:switch} for semidense blockades was developed solely by the author, which also suffices for \cref{subsec:densify}.
		This was later cleaned up by Astra in the setting of organised blockades, and refined further by the author to its current form with \cref{lem:complend,lem:switching}.
	\end{itemize}
\subsection*{Presentation}
	The writing was done by the author, with  copy-editing help from ChatGPT-6 Luna. All figures were crafted by ChatGPT-6 Sol and revised by the author.
\subsection*{Proofreading and mathematical review}
	ChatGPT-6 Astra assisted with proofreading, notation, and mathematical review of this paper.
	
\section*{Remarks on open access}
	For the purpose of open access, the author has applied a CC BY public copyright licence to
	any author accepted manuscript arising from this submission.
	
\setstretch{1}
\bibliographystyle{abbrv}

\end{document}